\documentclass[11pt,twoside,reqno]{amsart}
\usepackage{amsmath}
\usepackage{amssymb}
\usepackage{color}
\usepackage{graphicx}
\usepackage{amsfonts,amsmath, amssymb}

\numberwithin{equation}{section}

\usepackage{graphics}
\usepackage{epsfig}
\newtheorem{claim}{\bf \t}[part]

\newtheorem{theorem}{Theorem}[section]
\newtheorem{corollary}{Corollary}[section]
\newtheorem{lemma}{Lemma}[section]
\newtheorem{proposition}{Proposition}[section]
\newtheorem{remark}{Remark}[section]
\newtheorem{definition}{Definition}[section]

\def\cu{\mathcal{U}}
\def\bu{\mathbf{u}}
\def\bx{\mathbf{x}}
\def\by{\mathbf{y}}
\def\bz{\mathbf{z}}
\def\bn{\mathbf{n}}
\def\bt{\boldsymbol{\tau}}
\def\xxi{\boldsymbol{\xi}}
\def\nnu{\boldsymbol{\nu}}
\def\v{\varepsilon}

\def\t{\theta}

\newcommand{\dd}{{\rm d}}

\newcommand \R{\mathbb{R}}
\begin{document}

\title[Steady Discontinuous Euler Flows with Large Vorticity]{Steady Euler Flows with Large Vorticity
and Characteristic Discontinuities \\in Arbitrary Infinitely Long Nozzles}

\author{Gui-Qiang G. Chen}
\address{Gui-Qiang G. Chen, Mathematical Institute, University of Oxford,
Oxford, OX2 6GG, UK; School of Mathematical Sciences, Fudan University, Shanghai 200433, China;
AMSS \& UCAS, Chinese Academy of Sciences, Beijing 100190, China}
\email{chengq@maths.ox.ac.uk}

\author{Fei-Min Huang}
\address{Fei-Min Huang, School of Mathematical Sciences, University of Chinese Academy of Sciences, Beijing 100049, China;
Academy of Mathematics and System Sciences, Chinese Academy of Sciences, Beijing 100190, China}
\email{fhuang@amt.ac.cn}

\author{Tian-Yi Wang}
\address{Tian-Yi Wang, Department of Mathematics, School of Science, Wuhan University of Technology, Wuhan, Hubei 430070, China;
Gran Sasso Science Institute, viale Francesco Crispi, 7, 67100 L'Aquila, Italy}
\email{tianyiwang@whut.edu.cn; tian-yi.wang@gssi.infn.it;  wangtianyi@amss.ac.cn}

\author{Wei Xiang}
\address{Wei Xiang, Department of Mathematics, City University of Hong Kong, Kowloon, Hong Kong, China}
\email{weixiang@cityu.edu.hk}

\keywords{Steady Euler flows, large vorticity,
characteristic discontinuities, nozzles,
free boundary,
existence, uniqueness,
smooth solutions, weak solutions,
subsonic-sonic limit, incompressible limit,
vortex sheet, entropy wave}
\subjclass[2010]{35Q31, 35Q35, 35D30, 35D35, 35B30, 35F60, 35M30, 76D03, 76B03, 76N10, 76G25}
\date{\today}

\begin{abstract} We establish the existence and uniqueness of smooth solutions with large vorticity
and weak solutions with vortex sheets/entropy waves for the steady Euler equations
for both compressible and incompressible fluids in arbitrary infinitely long nozzles.
We first develop a new approach to establish the existence of smooth solutions without assumptions
on the sign of the second derivatives of the horizontal velocity, or the Bernoulli and entropy functions,
at the inlet for the smooth case.
Then the existence for the smooth case can be applied to construct approximate solutions to
establish the existence of weak solutions with vortex sheets/entropy waves
by nonlinear arguments.
This is the first result on the global existence of solutions of the multidimensional steady compressible
full Euler equations with free boundaries, which are not necessarily small perturbations of piecewise constant background solutions.
The subsonic-sonic limit of the solutions is also shown.
Finally, through the incompressible limit, we establish the existence and uniqueness of incompressible Euler flows
in arbitrary infinitely long nozzles for both the smooth solutions with large vorticity
and the weak solutions with vortex sheets.
The methods and techniques developed here will be useful for solving other problems involving similar difficulties.
\end{abstract}
\maketitle

\section{Introduction}

We are concerned with the full Euler system for steady compressible fluids in two space dimensions, which can be written as
\begin{eqnarray}\label{1.5}
\begin{cases}
\mbox{div}_{\bx}(\rho \bu)=0,\\
\mbox{div}_{\bx}(\rho \bu\otimes \bu)+ \nabla p=0,\\
\mbox{div}_{\bx} (\rho \bu E+ p \bu)=0,
\end{cases}
\end{eqnarray}
where $\bx=(x_1,x_2)\in \R^2$ represents the space coordinates, $\bu=(u_1, u_2)\in \R^2$ is the flow velocity,
$\rho$, $p$, and $E$ are the density, pressure, and total energy, respectively.
Note that $\rho$, $\bu$, $p$, and $E$ are not independent,
and they are connected by the constitutive relation.
In particular, for ideal polytropic gases,
$$
E=\frac{q^2}{2}+\frac{p}{(\gamma-1)\rho}
$$
with adiabatic exponent $\gamma>1$ and speed  $q^2=|\bu|^2=u_1^2+u_2^2$.

For smooth solutions, system \eqref{1.5} consists of four equations,
two of which are transport equations corresponding to the linearly degenerate characteristic fields.
The other two equations are of mixed elliptic-hyperbolic type:
They are elliptic if and only if the flow is subsonic.

More precisely, the two transport equations can be derived as the two conservation laws along the streamlines:
\begin{eqnarray}\label{1.9}
\begin{cases}
\mbox{div}_{\bx} (\rho \bu B)=0,\\
\mbox{div}_{\bx} (\rho \bu S)=0,
\end{cases}
\end{eqnarray}
where $B$ in \eqref{1.9} is the Bernoulli function determined by Bernoulli's law:
\begin{eqnarray}\label{ber}
B= \frac{q^2}{2}+\frac{\gamma p}{(\gamma-1)\rho},
\end{eqnarray}
and  $S$ in \eqref{1.9} is the entropy function defined by
\begin{eqnarray}\label{ent}
S=\frac{\gamma p}{(\gamma-1)\rho^\gamma}.
\end{eqnarray}

The sound speed $c$ of the flow is
\begin{equation}\label{1.7}
c=\sqrt{\frac{\gamma p}{\rho}},
\end{equation}
and the Mach number is
\begin{equation}\label{1.8}
M=\frac{q}{c}.
\end{equation}
Then, for a fixed Bernoulli function $B$,
there is the critical speed $q_{\rm cr}=\sqrt{\frac{2(\gamma-1)}{\gamma+1}B}$
such that, when $q \le q_{\rm cr}$, the flow is subsonic-sonic (\emph{i.e.} $M \le 1$);
otherwise, it is supersonic (\emph{i.e.} $M>1$).

\smallskip
In this paper, we are concerned with the existence and uniqueness of compressible subsonic smooth flows
with large vorticity without assumptions on the sign of the second-order derivatives of the horizontal velocity,
or the Bernoulli functions and the entropy functions, at the inlet,
as well as the weak solutions with vortex sheets/entropy waves in arbitrary infinitely long nozzles
which are not small perturbations of trivial background solutions such as piecewise constant solutions.
In particular, we develop a new nonlinear approach to deal with the problem with discontinuities, which does not need
the solution to be a small perturbation of a piecewise constant solution.
The earlier results
on the existence of subsonic compressible flows for large vorticity
or with a discontinuity in infinitely long nozzles
are under either small perturbation conditions or some signed conditions (\emph{cf.} \cite{Bae,CC,DXX,DL}).
Thus, this is the first result on the global existence of weak solutions containing vortex sheets and entropy waves,
which
are not merely perturbations around piecewise constant solutions.
In addition, it is also the first result on the global existence of weak solutions
of multidimensional steady compressible full Euler equations with free boundaries,
which are not perturbations of given background solutions.
The free boundary can be either a characteristic discontinuity or shock.
Moreover, the results in \cite{CC,DXX,DL} cannot directly be applied to construct a sequence of approximate solutions for our purpose here,
since their second derivatives are not always non-positive or non-negative and the first derivatives are not uniformly bounded owing to
the fact that the vorticity is a Dirac measure concentrated on the contact discontinuity.
Therefore, the results and methods we obtain here are both new.
Finally, we take the incompressible limit as $\gamma\rightarrow\infty$ to obtain
the existence and uniqueness of incompressible Euler flows in arbitrary infinitely long nozzles.

\smallskip
There are four main difficulties for these problems.
The first is that the Bernoulli function $\mathbb{B}$ and the entropy function $\mathbb{S}$ with respect to
the stream function have jumps owing to the contact discontinuities,
where $(\mathbb{B}, \mathbb{S})$ are introduced by \eqref{equ:2.3mathbb} later through $(B,S)$.
The second is that the first derivatives of $(\mathbb{B}, \mathbb{S})$ are not uniformly bounded,
and $\mathbb{B}''$ and $\mathbb{S}''$ have no sign, so that the energy method
as in \cite{ChenDengXiang,DXX} cannot directly be applied.
The third is to prove that there is no stagnation point inside the flow such that
several important inequalities can be derived from the equations of the pressure and the flow angle,
and the resulting solutions are the solutions of the full Euler equations.
The fourth difficulty is that there is no trivial background solution to provide the information
on the position where the vorticity should be large.

\smallskip
Our basic strategy is to start with the smooth case, for which we do not face the first difficulty.
To overcome the second and third difficulties, by sifting the coordinates,
we develop a new and global nonlinear method to prove that there is no stagnation point without assumptions
on the sign of $\mathbb{B}''$ and $\mathbb{S}''$,
and then to make several  \textit{a priori} estimates to show that the ellipticity of the flow is independent of
the bounds of the first derivatives of $(\mathbb{B}, \mathbb{S})$
essentially,
by applying the equations for the pressure and the flow angle.
Finally, we work on the equations in the Lagrangian coordinates
to overcome the fourth difficulty to obtain the far field behavior and the uniqueness of solutions.

In addition, there are several new difficulties for the discontinuous case.
One of them is the compactness of solutions.
On one hand, we can obtain only the uniform $L^{\infty}$ estimates
of pressure $p^{\varepsilon}$ and flow angle $\theta^{\varepsilon}$
of the approximate solutions
with respect to
the regularization parameter $\varepsilon$.
On the other hand, we need the additional
information of the limit functions $(\rho,\bu, p)$ to make sense for them,
especially the regularity of the discontinuity and the traces of $(\rho, \bu, p)$ along it.
Therefore, we cannot loss any regularity in order to gain the compactness.
This is the main reason why the compensated compactness argument
is required to show that the sequence of the approximate solutions $(\rho^{\varepsilon},\bu^{\varepsilon},p^{\varepsilon})$
is compact, and the limit is a weak solution.
Moreover, since the weak solution is only $L^{\infty}$,
a better regularity than $L^{\infty}$ for the discontinuity can not been obtained directly:  In fact,
we do not know whether the discontinuity is a curve, directly from the $L^\infty$ boundedness.
Thus, we introduce a contradiction argument to show that the discontinuity of the limit is actually a Lipschitz graph,
by tracing the regularity from the approximate solutions.

There has been some earlier analysis of the infinitely long nozzle problems.
For potential flows, Chen-Feldman \cite{ChenFeldman0,ChenFeldman} first established the existence
and stability of multidimensional
transonic flows through an infinitely long nozzle of arbitrary cross-sections;
also see \cite{CDSW}.
Xie-Xin \cite{xx1} established the existence of global subsonic homentropic flows and obtained
the critical upper bound of mass flux under the assumption that the derivative of
the Bernoulli function equals to zero on the two boundaries.
Du-Xie-Xin \cite{DXX} then established the existence of global subsonic homentropic flows
for large vorticity with the assumptions on the sign of the second derivatives of the horizontal velocity
at the inlet.
Following this method, Chen \cite{CC} and Duan-Luo \cite{DL} established the existence of
subsonic non-isentropic Euler flows with large vorticity in two-dimensional nozzles and in axisymmetric nozzles, respectively,
with similar assumptions on the sign of the second derivatives.
For the steady full Euler equations,
Chen-Chen-Feldman \cite{CCF} established the first existence of global transonic flows
in two-dimensional infinitely long nozzles of slowly varying cross-sections.
Then Chen-Deng-Xiang \cite{ChenDengXiang} focused on the full Euler equations for
the infinitely long nozzle problem with general varying cross-sections by developing some useful new techniques;
also see \cite{DWX} for the existence and uniqueness of smooth subsonic flows
with nontrivial swirl in axisymmetric nozzles.
For vortex sheets, the stability of a subsonic flat contact discontinuity in nozzles
by the perturbation argument has been established in Bae \cite{Bae} and Bae-Park \cite{BP-1,BP-2}.
Some further related results can be found in \cite{BF,CKL,CR,CZZ,JC,CY,GW,HKWX,hww,Serre,Yuan}
and the references cited therein.

The organization of the paper is as follows:
In \S \ref{sec:2b}, we introduce the notions of weak solutions and corresponding discontinuities,
formulate the problem as  \textbf{Problem 2.1($m$)},
and state the main theorems, Theorems 2.1--2.2.
In \S \ref{sec:mathsetting}, we further reformulate \textbf{Problem 2.1($m$)}
into \textbf{Problem 3.1($m$)} and \textbf{Problem 3.2($m$)},
and introduce several preliminary results.
In \S \ref{sec:existence smooth},
the existence of solutions of a modified elliptic problem in any infinitely long nozzles for large vorticity is established.
Subsequently, in \S \ref{sec:uniqueness}, we show the uniqueness of the solutions obtained in \S \ref{sec:existence smooth},
then we obtain the critical mass flux in both the stationary and subsonic-sonic sense,
and finally complete the proof of the first result, \emph{i.e.} Theorem \ref{thm:smooth} for the existence of smooth solutions
with large vorticity.
In \S \ref{sec:mathsettingweak}, we prove the second result, \emph{i.e.}
Theorem \ref{thm:discontinuity} for the existence of piecewise smooth solutions, by applying the compensated compactness
argument as in \cite{ChenHuangWang} (also see \cite{Chen} and the references cited therein).
Theorems \ref{thm5.2}--\ref{thm5.3} for the subsonic-sonic limit and the incompressible limit
are proved in \S \ref{section:7b},
and Theorem \ref{thm:smoothIC} for the existence and uniqueness of solutions of the incompressible Euler equations
is proved in \S \ref{section:8b}.
In \S 9, we give some remarks on steady full Euler flows with conservative exterior force, which can be reduced to the case without the exterior force,
so that all the corresponding results in Theorems 2.1--2.2 are still valid for this case.

\section{Mathematical Formulation and Main Results}\label{sec:2b}

In this section, we formulate the nozzle problem into a mathematical problem -- \textbf{Problem 2.1($m$)},
present the basic properties of weak solutions and the requirements on the velocity and entropy function at the inlet,
and state the main theorems and related remarks.

\subsection{Mathematical formulation of the problem}
Now we formulate the nozzle problem into a mathematical problem.
The infinitely long nozzle is defined as
\begin{equation*}
\Omega=\{\bx \, :\, w_1(x_1)<x_2<w_2(x_1), \,-\infty<x_1<\infty\}
\end{equation*}
with the nozzle walls $\partial\Omega:=W_1\cup W_2$ (see Fig \ref{Fig1}),
where
\begin{equation*}
W_j=\{\bx\, :\, x_2=w_j(x_1),
~-\infty<x_1<\infty\}, \qquad j=1,2.
\end{equation*}
Suppose that $W_1$ and $W_2$ satisfy
\begin{align}
&w_2(x_1)>w_1(x_1) ~ \qquad\qquad\qquad\qquad \mbox{for} ~x_1\in(-\infty, \infty),\nonumber\\
&w_1(x_1)\rightarrow 0, \quad w_2(x_1)\rightarrow 1\qquad\qquad \mbox{as} ~x_1\rightarrow -\infty, \nonumber\\
&w_1(x_1)\rightarrow a, \quad w_2(x_1)\rightarrow b>a
  \qquad \,\, \mbox{as} ~x_1\rightarrow \infty, \label{Con:Boundary1}
\end{align}
and there exists $\alpha>0$ such that
\begin{equation}\label{Con:Boundary2}
\|w_j\|_{C^{2,\alpha}(\mathbb{R})}\leq C,
\qquad j=1,2,
\end{equation}
for some positive constant $C$.
It is clear that $\Omega$ satisfies the uniform exterior sphere condition
with some uniform radius $r>0$.

\begin{figure}[htbp]
{\small \centering
\includegraphics[width=8cm]{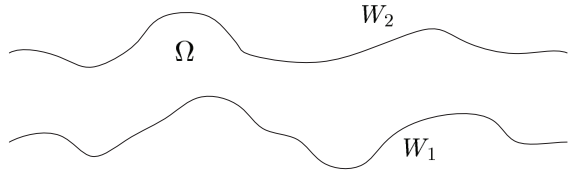}
\caption{Two-dimensional infinitely long nozzle}}
\label{Fig1}
\end{figure}

Suppose that the nozzle has impermeable solid walls so that the flow satisfies
the slip boundary condition:
\begin{equation}\label{cdx-5}
\bu \cdot\nnu =0 \qquad\, \mbox{on} ~\partial \Omega,
\end{equation}
where $\bu=(u_1, u_2)\in \mathbb{R}^2$ is the velocity and $\nnu$
is
the unit outward normal to the nozzle wall.
For the flow without vacuum, \eqref{cdx-5} can be written as
\begin{equation}\label{cdx-sc}
(\rho \bu) \cdot\nnu =0 \qquad\, \mbox{on} ~\partial \Omega.
\end{equation}

It follows from $(\ref{1.5})_1$ and $(\ref{cdx-sc})$ that
\begin{equation}\label{cdx-6}
\int_\ell\, (\rho \bu) \cdot \mathbf{n} \, \dd s\equiv m
\end{equation}
holds for some constant $m$, which is the mass flux,
where $\ell$ is any curve transversal to the $x_1$--direction,
and $\mathbf{n}$ is the normal of $\ell$ in the positive $x_1$--axis
direction.

We assume that the upstream entropy function is given as
\begin{equation}\label{cdx-8}
\frac{\gamma p}{(\gamma-1)\rho^\gamma}\longrightarrow S_-(x_2)
\qquad\, \mbox{as} ~x_1\rightarrow -\infty,
\end{equation}
and the upstream
horizontal velocity is given as
\begin{equation}\label{cdx-11}
u_1\rightarrow u_{1-}(x_2)>0
\qquad\,\mbox{as} ~x_1\rightarrow -\infty,
\end{equation}
where $S_-(x_2)$ and $u_{1-}(x_2)$ are functions defined on $[0,1]$.

\medskip
{\bf Problem 2.1($m$)}:
Solve the full Euler system $(\ref{1.5})$ with the boundary condition $(\ref{cdx-sc})$,
the mass flux condition $(\ref{cdx-6})$,
and the asymptotic conditions $(\ref{cdx-8})$--$(\ref{cdx-11})$
such that solution $(\rho, \bu, p)$ satisfies
\begin{eqnarray}
&M(\bx)<1 ~~\qquad\,\,\,\,\,\mbox{ for } a.e. \,\, \bx\in\overline{\Omega}, \label{ieq:subsonicxw}\\
&(\rho u_1)(\bx)>0 \qquad \mbox{ for } a.e. \,\, \bx\in\overline{\Omega}.\label{ieq:nondegeneracyxw}
\end{eqnarray}

\medskip
In this paper, we solve {\bf Problem 2.1($m$)}
for both smooth solutions and piecewise smooth solutions.
Such piecewise smooth subsonic flows with characteristic discontinuities also appear in many other physical problems including
the Mach reflection configurations ({\it cf.} Chen-Feldman \cite{CF-book}).

\subsection{Weak solutions and the Rankine-Hugoniot conditions}
To understand the piecewise smooth solutions,
we introduce the notion of weak solutions of the full Euler equations \eqref{1.5}.
\begin{definition}\label{def:2.1weak}
A piecewise smooth vector function $(\rho, \bu, p)$ is called a weak solution of the full Euler
equation \eqref{1.5} if the following holds{\rm :} For any $\phi\in C^1_0 (\Omega)$,
	\begin{equation}\label{weaksolution}
	\begin{cases}
	\int_{\Omega} \big(\rho u_1\,\partial_{x_1}\phi + \rho u_2\,\partial_{x_2}\phi \big)\,\dd\mathbf{x}=0,\\[2mm]
	\int_{\Omega} \big((\rho u_1^2+p)\,\partial_{x_1}\phi +  \rho u_1 u_2\, \partial_{x_2}\phi\big)\,\dd\mathbf{x}=0,\\[2mm]
	\int_{\Omega} \big(\rho u_1 u_2\, \partial_{x_1}\phi + \rho (u_2^2+p)\,\partial_{x_2}\phi \big)\,\dd\mathbf{x}=0,\\[2mm]
	\int_{\Omega} \big(\rho u_1 B\,\partial_{x_1}\phi + \rho u_2 B\,\partial_{x_2}\phi \big)\,\dd\mathbf{x}=0.
	\end{cases}
	\end{equation}
Moreover, on boundary $\partial\Omega$, the vector function $\rho\bu$ satisfies
the slip boundary condition \eqref{cdx-sc}.
\end{definition}

By integration by parts, we see that
\eqref{weaksolution} implies that the Rankine-Hugoniot conditions hold along the discontinuity curves for the
piecewise smooth solution.

\begin{proposition}
The vector function	$(\rho, \bu, p)$ is a piecewise smooth solution of the full Euler
system \eqref{1.5} if and only if
\begin{itemize}
\item[(i)] $(\rho, \bu, p)$ satisfies system \eqref{1.5} in the classic sense in
           the interior points of each smooth subregion of $\Omega$ bounded
            by a discontinuity curve $\Gamma${\rm ;}
\item[(ii)] the Rankine-Hugoniot conditions hold along $\Gamma$ almost everywhere{\rm :}
		\begin{equation}\label{R-HCondition}
		\begin{cases}
		n_1[\rho u_1]+n_2[\rho u_2]=0,\\[2mm]
		n_1[\rho u_1^2]+n_1[p]+n_2[\rho u_1 u_2]=0,\\[2mm]
		n_1[\rho u_1 u_2]+n_2[\rho u_2^2]+n_2[p]=0,\\[2mm]
		n_1[\rho u_1 B]+n_2[\rho u_2 B]=0,
		\end{cases}
		\end{equation}
where $\bn=(n_1, n_2)$ is the unit normal vector to $\Gamma$,
and $[h](\mathbf{x})=h_{+}(\mathbf{x})-h_{-}(\mathbf{x})$ denotes the jump
across the discontinuity curve $\Gamma$ for a piecewise smooth function $h$.
	\end{itemize}
\end{proposition}

Next, we classify the types of discontinuities.
First, we introduce $\bt=(\tau_1, \tau_2)$ as the unit tangential to $\Gamma$,
which means that $\bn\cdot\bt=0$.
Taking the dot product of ($(\ref{R-HCondition})_2$, $(\ref{R-HCondition})_3$) with $\bn$ and $\bt$ respectively,
we have
\begin{eqnarray}
&& [\rho(\bu\cdot\bn)^2+p]=0,\\
&& [\rho(\bu\cdot\bn)(\bu\cdot\bt)]=0.\label{RHTrans}
\end{eqnarray}
Then \eqref{RHTrans} can be rewritten as
\begin{equation*}
(\rho \bu\cdot\bn)_{\pm}[\bu\cdot\bt]=0,
\end{equation*}
where $(\rho \bu\cdot\bn)_{\pm}$ are the traces of $\rho \bu\cdot\bn$ on $\Gamma$ from $\Omega_{\pm}$, respectively.

If $(\rho \bu\cdot\bn)_{\pm}\neq 0$  on $\Gamma$, then  $\Gamma$ is a shock, provided that it also satisfies
the entropy condition: The density increases across $\Gamma$ in the flow direction.

If $(\rho \bu\cdot\bn)_{\pm}=0$, then $\Gamma$ is a characteristic discontinuity.
In this case, $[p]=0$.
In particular, $\Gamma$ is a vortex sheet when $[B]\neq0$ and $[S]=0$, and an entropy wave when $[B]=0$ and $[S]\neq0$.
Following the argument in \cite[page 302--303]{CFr} for the Prandtl formula,
it can be seen that only possible waves for the subsonic flows are vortex sheets and entropy waves.

\subsection{Requirements on the velocity and entropy function at the inlet}
Before stating the main results of this paper,
we analyze some requirements on $(u_{1-}, S_-)$ for a solution of {\bf Problem 2.1($m$)} to exist.
There are two main results. The first is the existence and uniqueness of smooth solutions with large
vorticity when  $(u_{1-}, S_-)\in C^{1,1}$,
and the second is the existence and uniqueness of piecewise smooth solutions with jumps.
Therefore, there are two kinds of the requirements for $(u_{1-}, S_-)$.

\subsubsection{Requirements on the smooth velocity and entropy function}
Motivated by \cite{ChenDengXiang}, we require that $(u_{1-}, S_{-})\in C^{1,1}[0,1]$
with
\begin{equation}\label{assumption:3.40}
\inf\limits_{s\in[0,1]}u_{1-}(s)>0,\qquad\,\,
\inf\limits_{s\in[0,1]}S_{-}(s)>0.
\end{equation}
On the boundary, we need the following monotone properties for $(u_{1-}, S_-)$:
\begin{equation}\label{assumption:3.47a}
(u_{1-}^2S_-^{-\frac{1}{\gamma}})'(0)\leq0,
\qquad\,\,
(u_{1-}^2S_-^{-\frac{1}{\gamma}})'(1)\geq0.
\end{equation}

\subsubsection{Requirements on the functions with jumps}
In this case, we assume that the number of the discontinuities is finite,
such that the characteristic discontinuities are separated from each other.
Without loss of generality, we assume that $(u_{1-}, S_-)$ have only one discontinuity at $x_2=x_d$.
Similar to the smooth case, we require
\begin{equation}\label{assumption:3.44dis}
(u_{1-}, S_{-})\in (C^{1,1}[0,x_d)\cup C^{1,1}(x_d,1])^2.
\end{equation}
On the boundary, we require the monotone assumption \eqref{assumption:3.47a}.
Moreover, near the discontinuity, $(u_{1-}, S_-)$ satisfy either
\begin{equation}\label{assumption:4.2}
\begin{cases}
\inf\limits_{x_{d}-\varepsilon_0<s<x_{d}}(u_{1-}^2S_-^{-\frac{1}{\gamma}})'(s)\geq0,
\qquad \inf\limits_{x_{d}-\varepsilon_0<s<x_{d}}S_{-}'(s)\geq0,\\[2mm]
[u_{1-}^2S_-^{-\frac{1}{\gamma}}]\geq0, \qquad [S_{-}]\geq0,
\end{cases}
\end{equation}
or
\begin{equation}\label{assumption:4.3}
\begin{cases}
\sup\limits_{x_{d}<s< x_{d}+\varepsilon_0}(u_{1-}^2S_-^{-\frac{1}{\gamma}})'(s)\leq0,
\qquad \sup\limits_{x_{d}<s<x_{d}+\varepsilon_0}S_{-}'(s)\leq0,\\[2mm]
[u_{1-}^2S_-^{-\frac{1}{\gamma}}]\leq0,
\qquad [S_{-}]\leq0.
\end{cases}
\end{equation}

We note that at least one of the inequalities, $\geq$ or $\leq$, in \eqref{assumption:4.2} or \eqref{assumption:4.3}
must be strict, \emph{i.e.} $>$ or $<$.
Otherwise, it reduces to the first case, since there is no discontinuity.

\begin{remark}
For the homentropic case, $S=const.$ and $u_{1-}'=\omega_-$ that is the vorticity.
Then the sign conditions in
\eqref{assumption:3.47a} and \eqref{assumption:4.2}--\eqref{assumption:4.3} are the corresponding sign conditions on the inlet vorticity $\omega_-$.
\end{remark}

\subsection{Main theorems}

In this paper, we establish the following two theorems.
The first is the global existence and uniqueness of smooth solutions with large vorticity
for the given smooth horizontal velocity $u_{1-}$ and entropy function $S_-$ at the inlet,
and the second is the global existence and uniqueness of weak solutions
with characteristic discontinuities.

\begin{theorem}[Smooth solutions with large vorticity]\label{thm:smooth}
For any given entropy function $S_{-}$ and horizontal velocity $u_{1-}$ at the inlet
satisfying \eqref{assumption:3.40}--\eqref{assumption:3.47a},
there exists a critical mass flux $m_{\rm c}>0$ that depends only on
$(u_{1-}, S_-)$ and $(w_1,w_2)$
{\rm (}the boundary of the nozzle{\rm )}
such that, for any given $m>m_{\rm c}$,
there exists a solution $(\rho,\bu,p)\in (C^{1,\alpha}(\Omega))^4$ of {\bf Problem 2.1($m$)}.
Moreover, the solution is unique with the additional properties\,{\rm :}
\begin{eqnarray}
&& p\rightarrow p_{\pm}>0, \quad u_1\rightarrow u_{1\pm}(x_2)>0,
\quad (u_2,\rho)\rightarrow(0,\rho_{\pm}(x_2;p_{\pm})), \label{equ:2.59a}\\
&& \nabla p\rightarrow0,\quad\nabla u_1\rightarrow(0,u_{1\pm}'(x_2)), \quad \nabla u_2\rightarrow (0,0),
  \quad \nabla \rho\rightarrow(0,\rho'_0(x_2;p_{\pm})), \label{equ:2.60a}
\end{eqnarray}
uniformly for $x_2\in K\Subset(0,1)$ as $x_1\rightarrow-\infty$
and for $x_2\in K\Subset(a,b)$ as $x_1\rightarrow\infty$.
The asymptotic states $\rho_{\pm}(x_2)$, $u_{1+}(x_2)$, and $p_{\pm}$ are determined
in \S {\rm \ref{subset:asymptotic state}} and \S {\rm \ref{subset:asymptotic state plus}} later.
In addition,
\begin{equation}\label{2.61b}
\min_{\partial\Omega} p\le p \le \max_{\partial\Omega} p,
\end{equation}
and the flow angle $\theta$ defined by $\tan\theta=\frac{u_2}{u_1}$
satisfies
\begin{equation}\label{equ:2.61a}
|\theta|\leq \theta_B,
\end{equation}
where $\theta_B:=\max_{W_1\cup W_2}\theta$.
Finally, the critical mass $m_{\rm c}$ satisfies that
either
\begin{equation*}
\sup_{\overline{\Omega}}(|\bu|^2-c^2)\rightarrow0 \,\qquad\mbox{as $m>m_{\rm c}$ and $m\rightarrow m_{\rm c}$},
\end{equation*}
or, for any $\sigma>0$,
there exists $m\in(m_{\rm c}-\sigma,m_{\rm c})$ such that there is a full Euler flow satisfying
\begin{equation}\label{equ:2.63a}
\sup_{m_{\rm c}-\sigma<m<m_{\rm c}}\sup_{\overline{\Omega}}(|\bu|^2-c^2)>0.
\end{equation}
\end{theorem}

\begin{remark}
In \cite{ChenDengXiang,xx1},
the existence of smooth subsonic flows is proved under some smallness assumptions of $(u_{1-}, S_{-})$.
In this paper, all the smallness assumptions of \cite{ChenDengXiang,xx1} are removed.
\end{remark}

\begin{remark}
In \cite{DXX}, a convexity condition that the second derivative of $u_{1-}$ is non-negative
is essential in the analysis. In this paper, the convexity condition is removed.
Note that the vorticity is a Dirac measure for the characteristic discontinuity case,
so that the sign of the second derivative of the velocity of the approximate solutions can not be kept.
In particular, the convexity condition is not needed in Theorem {\rm \ref{thm:smooth}},
which is essential for us to investigate the weak solutions with characteristic discontinuities.
\end{remark}

\begin{theorem}[Weak solutions with characteristic discontinuities]\label{thm:discontinuity}
For any given entropy function $S_{-}$ and horizontal velocity $u_{1-}$ at the inlet,
which satisfy assumptions \eqref{assumption:3.44dis}--\eqref{assumption:4.3},
there exists a non-negative critical mass flux $m_{\rm c}$, depending only
on $(u_{1-}, S_-)$ and $(w_1,w_2)$ {\rm (}the boundary of the nozzle{\rm )},
such that, for any given $m>m_{\rm c}$, there is a weak solution $(\rho,\bu,p)$ of {\bf Problem 2.1($m$)}.
Moreover,
the weak solution $(\rho, \bu,p)$ is unique and piecewise $(C^{1,\alpha})^4$ in $\Omega$.
The solution satisfies properties \eqref{equ:2.59a}$-$\eqref{equ:2.60a}
uniformly for $x_2\in K\Subset(0,1)$ as $x_1\rightarrow-\infty$ and for $x_2\in K\Subset(a,b)$ as $x_1\rightarrow\infty$
away from the discontinuity.
The discontinuity is a streamline with the Lipschitz regularity.
The solutions also satisfy the properties in \eqref{2.61b}--\eqref{equ:2.61a}.
The critical mass flux $m_{\rm c}$
satisfies that
either
\begin{equation*}
\sup_{\overline{\Omega}}\big(|\bu|^2-c^2)\rightarrow0,\qquad\mbox{as }~m>m_{\rm c}, ~m\rightarrow m_{\rm c},
\end{equation*}
or, for any $\sigma>0$, there is $0<\varepsilon_0\le 1$ such that, for any $\v\in (0, \varepsilon_0)$, there exists $m\in(m_{\rm c}-\sigma, m_{\rm c})$
so that there is a full Euler flow satisfying
\begin{equation*}
\sup_{m\in(m_{\rm c}-\sigma,m_{\rm c})}\sup_{\overline{\Omega}}\big(|\bu^\varepsilon|^2-(c^\varepsilon)^2\big)>0,
\end{equation*}
where $(\rho^{\varepsilon},\bu^{\varepsilon},p^{\varepsilon})$ is the solution corresponding to $(u_-^{\varepsilon}, S_-^{\varepsilon})$
satisfying the assumptions of Theorem {\rm \ref{thm:smooth}}, converging to $(u_-, S_-)$ pointwise in $[0,1]$ and $C^{1,1}$
away from the jump point $x_2=x_d$, with the mass flux $m$.
Finally, across the discontinuity, the solution satisfies
\begin{equation}\label{equ:2.28xw}
(\rho \bu\cdot\bn)_{\pm}=0 \,\qquad \mbox{on }\Gamma
\end{equation}
as normal traces in the sense of Chen-Frid \cite{Chen7},
which means the characteristic discontinuity $\Gamma$ is either a vortex sheet or an entropy wave.
\end{theorem}

\begin{remark}
There are several previous results on the global existence of piecewise smooth solutions of multidimensional steady compressible Euler
equations in the literature.
All of these results are based on the perturbation analysis around a piecewise constant background solution {\rm (}cf. \cite{Bae,BP-1,BP-2,ChenS,HKWX}{\rm )}.
As far as we know, the result in Theorem {\rm \ref{thm:discontinuity}} is the first
on the global existence of piecewise smooth solutions
of multidimensional steady compressible Euler equations,
which are not necessarily a perturbation of piecewise constat background solutions.
\end{remark}

\begin{remark}
Based on two main theorems above, we further establish Theorem {\rm \ref{thm5.2}} for the subsonic-sonic limit,
Theorem {\rm \ref{thm5.3}} for the incompressible limit, and Theorem {\rm \ref{thm:smoothIC}} for
the existence and uniqueness of solutions of inhomogeneous incompressible Euler flows.
\end{remark}

We develop a nonlinear approach based on Theorem \ref{thm:smooth} to obtain Theorem \ref{thm:discontinuity}.
Therefore, we will prove Theorem \ref{thm:smooth} first in \S \ref{sec:mathsetting}--\S \ref{sec:uniqueness}
and then Theorem \ref{thm:discontinuity} in \S \ref{sec:mathsettingweak}.

\section{Mathematical Reformulation and Preliminary Results}\label{sec:mathsetting}

In this section, we reformulate \textbf{Problem 2.1($m$)} as two mathematical problems
for the stream function $\psi$
with corresponding boundary conditions and additional properties
for smooth subsonic flow and piecewise smooth subsonic flow, respectively,
and discuss some basic properties of solutions of the two problems.

\subsection{Asymptotic states at the inlet}
\label{subset:asymptotic state}
We first introduce the states at the inlet.

By the definition of Bernoulli function, entropy function, and mass flux, for a fixed constant $p_-$, we denote
\begin{eqnarray}
&& B_-(x_2):= \frac{1}{2}|u_{1-}(x_2)|^2+\frac{\gamma p_{-}}{(\gamma-1)\rho_-(x_2)},
   \label{Ber--infty}\\
&& S_-(x_2):=\frac{\gamma p_-}{(\gamma-1)\rho_-^\gamma(x_2)},
   \label{entropy--infty}\\
&& m=\int_0^1(\rho_-u_{1-})(x_2)\, \dd x_2.
  \label{mass--infty}
\end{eqnarray}

The above formulas imply
$$
\rho_-(x_2, p_-)=\Big(\frac{\gamma p_-}{(\gamma-1) S_-(x_2)}\Big)^{\frac{1}{\gamma}}.
$$
Then
\begin{equation*}
	m=\int_0^1(\rho_-u_{1-})(x_2)\, \dd x_2
	=\Big(\frac{\gamma p_-}{\gamma-1}\Big)^{\frac{1}{\gamma}}\int_0^1(u_{1-}S_-^{-\frac{1}{\gamma}})(x_2)\, \dd x_2.
\end{equation*}
For given $(u_{1-}, S_-)(x_2)$ and $m$, we have
\begin{equation}\label{3.5}
p_{-}=\frac{\gamma-1}{\gamma}m^\gamma\Big(\int_0^1(u_{1-}S_-^{-\frac{1}{\gamma}})(x_2) \, \dd x_2\Big)^{-\gamma}.
\end{equation}
Then, by the Bernoulli law \eqref{Ber--infty},
\begin{equation}\label{equ:2.14a}
	B_-(x_2)
	=\frac{1}{2}|u_{1-}(x_2)|^2
+ m^{\gamma-1}S_-^{\frac{1}{\gamma}}(x_2) \Big(\int_0^1(u_{1-}S_-^{-\frac{1}{\gamma}})(x_2)\, \dd x_2\Big)^{1-\gamma}.
\end{equation}

We require $u_{1-}<q_{\rm cr}$ so that the flow is subsonic, where
$$
q_{\rm cr}=\Big(\frac{2(\gamma-1)}{\gamma+1}B_-\Big)^{\frac{1}{2}}.
$$
Set
$$
\hat{m}:=(\gamma-1)^{-\frac{1}{\gamma-1}}\Big(\max_{x_2\in[0,1]} (u_{1-}^2 S_-^{-\frac{1}{\gamma}})(x_2) \Big)^{\frac{1}{\gamma-1}}
   \int_0^1(u_{1-} S_-^{-\frac{1}{\gamma}})(x_2) \dd x_2
$$
so that $u_{1-}<q_{\rm cr}$ holds when $m>\hat{m}$.

\begin{remark}
The above argument allows $(u_{1-}, S_-)$ to be piecewise smooth.
\end{remark}

\subsection{Reformulations of {\bf Problem 2.1($m$)}}\label{subsec:transport equations}
If there is a smooth solution for {\bf Problem 2.1($m$)},
then, from the first equation in \eqref{1.5},
there exists a stream function $\psi$ satisfying
\begin{equation}\label{psi}
\partial_{x_1}\psi=-\rho u_2,\qquad
\partial_{x_2}\psi=\rho u_1.
\end{equation}
Note that the stream function $\psi$ is identically equal to a constant along the characteristic.
Moreover, by \eqref{cdx-6},
$\psi=m$ on the boundary $x_2=w_2(x_1)$, if $\psi=0$ on the boundary $x_2=w_1(x_1)$.

From (\ref{1.9}), there are two linear transport equations corresponding to the linearly degenerate characteristic fields.
These equations indicate that the Bernoulli function $B$ and entropy function $S$ are constant along the streamlines.
In fact,
$(B,S)$ can be regarded as functions of the stream function $\psi$ via the following argument:

If $\partial_{x_2} \psi=\rho u_1$ does not vanish, we can introduce the Bernoulli function $\mathbb{B}$
and entropy function $\mathbb{S}$, respectively, with respect to the stream function $\psi$
as $(\mathbb{B}, \mathbb{S})(\psi):=(B_-, S_-)\circ\psi_-^{-1}(\psi)$,
where
\begin{equation}\label{psi-ini}
\psi_-(x_2)=\int_0^{x_2}(\rho_-u_{1-})(y)\, \dd y.
\end{equation}
Then
\begin{equation}\label{equ:2.3mathbb}
(B,S)(\bx)=(\mathbb{B},\mathbb{S})(\psi(\bx)),
\end{equation}
and
$$
	\nabla B=\rho \mathbb{B}'(-u_2, u_1),\qquad
	\nabla S=\rho \mathbb{S}'(-u_2, u_1).
$$

Another notable quantity is the vorticity function $\omega:=\partial_{x_1}u_2-\partial_{x_2}u_1$.
We now see the connection between $(\rho,\, \mathbb{S},\, \mathbb{B})$ and $\omega$.
Differentiating the Bernoulli function $B$ with respect to $x_i$, $i=1,2$, we have
\begin{equation}\label{1.11}
	\partial_{x_1} B=  u_2 \omega+\frac{\rho^{\gamma-1}}{\gamma}\partial_{x_1} S,\qquad
	\partial_{x_2} B=  - u_1 \omega+\frac{\rho^{\gamma-1}}{\gamma}\partial_{x_2} S,
\end{equation}
which yields
$$
\omega=\frac{1}{q^2}\Big(u_2(\partial_{x_1} B-\frac{\rho^{\gamma-1}}{\gamma}\partial_{x_1} S)
 -u_1(\partial_{x_2} B-\frac{\rho^{\gamma-1}}{\gamma}\partial_{x_2} S)\Big)
=-\rho \mathbb{B}'+\frac{\rho^{\gamma}}{\gamma}\mathbb{S}'.
$$

Thus, we obtain the following  equivalent system of \eqref{1.5} when $\partial_{x_2}\psi>0$:
\begin{eqnarray}\label{streamfunction formulation}\label{equ:8.1discontinuous}
	\begin{cases}
		\nabla\psi=\rho (-u_2, u_1),\\[1mm]
		\mathbb{B}(\psi)=\frac{1}{2}|\bu|^2+\frac{\gamma p}{(\gamma-1)\rho},\\[1mm]
		\mathbb{S}(\psi)= \frac{\gamma p}{(\gamma-1)\rho^\gamma},\\[1mm]
		\partial_{x_1}u_2-\partial_{x_2}u_1=-\rho \mathbb{B}'+\frac{\rho^{\gamma}}{\gamma}\mathbb{S}'
	\end{cases}
\end{eqnarray}
with boundary conditions:
\begin{eqnarray}\label{bou:2.9}
\psi|_{x_2=w_1(x_1)}=0,
\qquad
\psi|_{x_2=w_2(x_1)}=m.
\end{eqnarray}

Based on \eqref{streamfunction formulation},
the Bernoulli function becomes
\begin{equation}\label{E:2.22}
\frac{1}{2}|\nabla_{\bx}\psi|^2+\rho^{\gamma+1}\mathbb{S}(\psi)=\rho^2\mathbb{B}(\psi),
\end{equation}
which shows that, for the subsonic flow,
density $\rho$ is an increasing function of $|\nabla_{\bx}\psi|$,
and speed $q$ becomes $q=\frac{|\nabla_{\mathbf{x}}\psi|}{\rho}$.
We write $\rho(|\nabla_{\bx}\psi|, \psi)$ as the implicit function from \eqref{E:2.22},
when $M<1$.
Meanwhile, the sonic speed $c=c(|\nabla_{\bx}\psi|, \psi)\geq0$ is defined as
$$
c^2(|\nabla_{\bx}\psi|, \psi):=(\gamma-1)\rho^{\gamma-1}(|\nabla_{\bx}\psi|, \psi)\mathbb{S}(\psi).
$$
By a direct calculation, we see that $M<1$ if and only if
\begin{equation}\label{subsonic0}
|\nabla_{\bx}\psi|^2
<(\gamma-1)\Big(\frac{2\mathbb{B}(\psi)}{\gamma+1}\Big)^{\frac{\gamma+1}{\gamma-1}}\mathbb{S}^{-\frac{2}{\gamma-1}}(\psi).
\end{equation}

In addition, the vorticity function $\omega$ is
\begin{eqnarray}
\omega
:=-\partial_{x_1}\Big(\frac{\partial_{x_1}\psi}{\rho}\Big)-\partial_{x_2}\Big(\frac{\partial_{x_2}\psi}{\rho}\Big)
=-\frac{1}{\rho^2}\big(\rho \Delta_{\bx}\psi +\nabla_\bx \psi\cdot\nabla_\bx \rho\big)
\label{equ:3.14}.
\end{eqnarray}
Taking the derivatives  with respect to  $x_i$, $i=1,2$, on \eqref{E:2.22}
yields
\begin{equation}\label{E:2.23}
\partial_{x_i}\rho
=\frac{\frac{1}{2}\partial_{x_i}(|\nabla_{\bx}\psi|^2)+(\rho^{\gamma+1}\mathbb{S}'
 -\rho^2\mathbb{B}')\partial_{x_i}\psi}{\rho(q^2-c^2)},
\end{equation}
where
\begin{eqnarray*}
\frac{1}{2}\partial_{x_i}(|\nabla_{\bx}\psi|^2)
&=&\partial_{x_1}\psi\partial_{x_1x_i}\psi+\partial_{x_2}\psi\partial_{x_ix_2}\psi.
\end{eqnarray*}
Combining \eqref{equ:3.14}--\eqref{E:2.23} with $\eqref{equ:8.1discontinuous}_4$,
we finally obtain a second-order equation for $\psi$,
which is equivalent to the full Euler system \eqref{1.5} for smooth solutions when $\psi_{x_2}>0$:
\begin{eqnarray}\label{equ:3.18}
	a_{11}\partial_{x_1x_1}\psi+a_{12}\partial_{x_1x_2}\psi+a_{22}\partial_{x_2x_2}\psi=f,
\end{eqnarray}
where
\begin{eqnarray*}
	\begin{cases}
		a_{11}:=a_{11}(\nabla_{\bx}\psi, \psi)= \rho^2c^2-(\partial_{x_2}\psi)^2,\\[1mm]
		a_{12}:=a_{12}(\nabla_{\bx}\psi, \psi)= 2\partial_{x_1}\psi\partial_{x_2}\psi,\\[1mm]
		a_{22}:=a_{22}(\nabla_{\bx}\psi, \psi)= \rho^2c^2-(\partial_{x_1}\psi)^2,\\[1mm]
		f:=f(\nabla_{\bx}\psi, \psi)=-\rho^3q^2\frac{(\gamma-1)\rho^\gamma\mathbb{S}'}{\gamma}
              +\rho^3 c^2(\rho\mathbb{B}'-\frac{\rho^\gamma}{\gamma}\mathbb{S}').
	\end{cases}
\end{eqnarray*}

Now we can reformulate {\bf Problem 2.1($m$)} into the following problems.

\medskip
\textbf{Problem 3.1($m$) (Smooth subsonic flow)}.
Find a smooth solution $\psi(\bx)$ of the nonlinear equation \eqref{equ:3.18}
for $\bx\in\Omega$  with the boundary conditions \eqref{bou:2.9}
so that
\begin{eqnarray}
&&M(\nabla_{\bx}\psi, \psi)<1 \qquad \mbox{ for } \bx\in\overline{\Omega}, \label{ieq:subsonic}\\
&&\partial_{x_2}\psi>0 \qquad\qquad\quad  \mbox{ for } \bx\in\overline{\Omega},\label{ieq:nondegeneracy0}\\
&& 0<\psi<m  \qquad\qquad\, \mbox{ for } \bx\in\Omega.\label{ieq:streamline}
\end{eqnarray}

\medskip
\textbf{Problem 3.2($m$) (Piecewise smooth subsonic flow)}.
Find a piecewise smooth solution $\psi(\bx)$ of the nonlinear equations \eqref{equ:8.1discontinuous}
for $\bx\in\Omega$ with the boundary conditions \eqref{bou:2.9} when $\psi_-'(x_2)$ has jump at $x_2=x_d$, $0<x_d<1$,
so that there is a discontinuity $\Gamma$ that is a streamline originated from the inlet with
\begin{eqnarray*}
	\psi\equiv m_d:=\int_0^{x_d}(\rho_-u_{1-})(t)\, \dd t  \qquad \mbox{at $\Gamma$},
\end{eqnarray*}
and satisfies \eqref{ieq:subsonic}, \eqref{ieq:streamline}, and
\begin{equation}\label{ieq:nondegeneracy}
	\partial_{x_2}\psi>0 \qquad \mbox{ for } \bx\in\overline{\Omega}\backslash\Gamma.
\end{equation}

\medskip
Finally, requirements \eqref{assumption:3.40}$-$\eqref{assumption:4.3} at the inlet can be reformulated
to the requirements on $(\mathbb{B},\mathbb{S})$ as follows:
By a direct calculation, \eqref{assumption:3.40}--\eqref{assumption:3.47a} imply
	\begin{equation}\label{assumption:3.40i}
	\inf\limits_{s\in[0,m]}\mathbb{B}(s)>0,\qquad\,
	\inf\limits_{s\in[0,m]}\mathbb{S}(s)>0,
	\end{equation}
	and
	\begin{equation}\label{assumption:3.47ai}
	(\mathbb{B}\mathbb{S}^{-\frac{1}{\gamma}})'(0)\leq0,
	\qquad\,
	(\mathbb{B}\mathbb{S}^{-\frac{1}{\gamma}})'(m)\geq0.
	\end{equation}

For requirements \eqref{assumption:4.2}$-$\eqref{assumption:4.3}, we have

\begin{lemma}\label{lem:3.1xw}
If \eqref{ieq:nondegeneracy} holds,
then assumptions \eqref{assumption:4.2}$-$\eqref{assumption:4.3} imply
\begin{enumerate}
\item[(i)]
$\quad \inf\limits_{m_d-\varepsilon_0<s<m_d}(\mathbb{B}\mathbb{S}^{-\frac{1}{\gamma}})'(s)\geq0$,
$\quad$ $\inf\limits_{m_d-\varepsilon_0<s<m_d}(\mathbb{B}\mathbb{S}^{-\frac{2}{\gamma(\gamma+1)}})'(s)\geq0$,

\smallskip
$[\mathbb{B}\mathbb{S}^{-\frac{1}{\gamma}}](m_d)\geq0$, $\quad$ $[\mathbb{B}\mathbb{S}^{-\frac{2}{\gamma(\gamma+1)}}](m_d)\geq0$,
	
	\smallskip
	or
	
	\smallskip
\item[(ii)]
$\quad \sup\limits_{m_d<s<m_d+\varepsilon_0}(\mathbb{B}\mathbb{S}^{-\frac{1}{\gamma}})'(s)\leq0$,
$\quad$ $\sup\limits_{m_d<s<m_d+\varepsilon_0}(\mathbb{B}\mathbb{S}^{-\frac{2}{\gamma(\gamma+1)}})'(s)\leq0$,

\smallskip
$[\mathbb{B}\mathbb{S}^{-\frac{1}{\gamma}}](m_d)\leq0$, $\quad$ $[\mathbb{B}\mathbb{S}^{-\frac{2}{\gamma(\gamma+1)}}](m_d)\leq0$.
\end{enumerate}
\end{lemma}

\begin{proof}
In the same way as in \eqref{assumption:3.47ai}, we can prove Lemma \ref{lem:3.1xw} for $(\mathbb{B}\mathbb{S}^{-\frac{1}{\gamma}})'$.
Furthermore,
note that \eqref{equ:2.14a} can be rewritten as $B_-=\frac{1}{2}u_{1-}^2+S_-^{\frac{1}{\gamma}}a(m)$, where
	$$
	a(m)=m^{\gamma-1} \Big(\int_0^1(u_{1-}S_-^{-\frac{1}{\gamma}})(x_2)\, \dd x_2\Big)^{1-\gamma}>0.
	$$
Then
	\begin{eqnarray}\label{3.27a}
	B_-S_-^{-\frac{2}{\gamma(\gamma+1)}}
	=\frac{1}{2}u_{1-}^2S_-^{-\frac{1}{\gamma}} S_-^{\frac{\gamma-1}{\gamma(\gamma+1)}}+S_-^{\frac{\gamma-1}{\gamma(\gamma+1)}}a(m).
	\end{eqnarray}
Taking $\partial_{x_2}$ on \eqref{3.27a} yields
	\begin{eqnarray*}
(B_-S_-^{-\frac{2}{\gamma(\gamma+1)}})'
=\frac{1}{2}(u_{1-}^2S_-^{-\frac{1}{\gamma}})' S_-^{\frac{\gamma-1}{\gamma(\gamma+1)}}
    + \frac{\gamma-1}{2\gamma(\gamma+1)}\Big(u_{1-}^2S_-^{-\frac{1}{\gamma}}+2 a(m)\Big)
       S_-^{-\frac{\gamma^2+1}{\gamma(\gamma+1)}}S_{-}'.\nonumber
	\end{eqnarray*}
	Therefore, Lemma \ref{lem:3.1xw} is proved.
\end{proof}

\subsection{A modification problem}
Following \cite{ChenDengXiang}, we introduce a modified problem to solve {\bf Problem 3.1($m$)}.
First, set
\begin{eqnarray*}
\mathfrak{S}(s)=
\begin{cases}
\mathbb{S}'(s) \quad&\text{if}\ 0\leq s\leq m,\\[1mm]
\mathbb{S}'(m)\frac{2m-s}{m} \quad&\text{if}\ m\leq s\leq 2m,\\[1mm]
\mathbb{S}'(0)\frac{s+m}{m} \quad&\text{if}\ -m\leq s\leq 0,\\[1mm]
0,\qquad&\text{if}\ s\geq2m \,\, \text{or}\ s\leq-m,
\end{cases}
\end{eqnarray*}
and
\begin{eqnarray*}
\mathfrak{B}(s)=
\begin{cases}
(\mathbb{B}\mathbb{S}^{-\frac{1}{\gamma}})'(s) \quad&\text{if}\ 0\leq s\leq m,\\[1mm]
(\mathbb{B}\mathbb{S}^{-\frac{1}{\gamma}})'(m)\frac{2m-s}{m} \quad&\text{if}\ m\leq s\leq 2m,\\[1mm]
(\mathbb{B}\mathbb{S}^{-\frac{1}{\gamma}})'(0)\frac{s+m}{m} \quad&\text{if}\ -m\leq s\leq 0,\\[1mm]
0,\qquad&\text{if}\ \psi\geq2m\,\,\text{or}\ s\leq-m.
\end{cases}
\end{eqnarray*}
We define
\begin{equation}\label{3.0}
\tilde{\mathbb{S}}(s)=\mathbb{S}(0)+\int_{0}^{s}\mathfrak{S}(\tau)\dd\tau,\qquad
\tilde{\mathbb{B}}(s)=\tilde{\mathbb{S}}^{\frac{1}{\gamma}}(s)
\Big(\mathbb{B}(0)\mathbb{S}^{-\frac{1}{\gamma}}(0)+\int_{0}^{s}\mathfrak{B}(\tau)\dd\tau\Big).
\end{equation}
For notational simplicity, we still denote them as $(\mathbb{B}, \mathbb{S})$ later on.
For $\v>0$, let
\begin{eqnarray*}
\zeta_{0}(s)=
\begin{cases}
s \qquad&\text{if}\ s<-2\v,\\
-\frac{3}{2}\v \qquad&\text{if}\ s\geq-\v
\end{cases}
\end{eqnarray*}
be a smooth increasing function  such that $|\zeta'_{0}|\leq1$.
Define
\begin{equation}\label{cut-off using}
\tilde{Q}(|\nabla_{\bx}\psi|,\psi)
:=\Big(\zeta_{0}(\frac{|\nabla_{\bx}\psi|}{\hat{Q}(\psi)}-1)+1\Big)\hat{Q}(\psi),
\end{equation}
so that
\begin{equation*}
\frac{1}{2}\tilde{Q}^2(|\nabla_{\bx}\psi|, \psi)+
\tilde{\mathbb{S}}(\psi)\tilde{\rho}^{\gamma+1}=\tilde{\mathbb{B}}(\psi)\tilde{\rho}^2,
\end{equation*}
where
\begin{equation}\label{hat-Q}
\hat{Q}(\psi):=
(\gamma-1)^{\frac{1}{2}}\Big(\frac{2\tilde{\mathbb{B}}(\psi)}{\gamma+1}\Big)^{\frac{\gamma+1}{2(\gamma-1)}}
\tilde{\mathbb{S}}^{-\frac{1}{\gamma-1}}(\psi).
\end{equation}
Then we can define the modified density $\tilde{\rho}$ and the sonic speed $\tilde{c}$ accordingly.
Replace $|\nabla_{\bx}\psi|$, $\rho(|\nabla_{\bx}\psi|, \psi)$, $\mathbb{B}$, and $\mathbb{S}$
by $\tilde{Q}(|\nabla_{\bx}\psi|, \psi)$, $\tilde{\rho}(|\nabla_{\bx}\psi|, \psi)$, $\tilde{\mathbb{B}}$,
and $\tilde{\mathbb{S}}$ in $f$ and $a_{ij}$, $i,j=1,2$,
and rewrite these quantities as $\tilde{f}$ and $\tilde{a}_{ij}$,
$i,j=1,2$.
Then we reformulate {\bf Problem 3.1($m$)} to the following {\bf Problem 3.3($m$)}.

\medskip
\textbf{Problem 3.3($m$)}:
Solve
\begin{eqnarray*}
\begin{cases}
\tilde{a}_{11}\partial_{x_1x_1}\psi+\tilde{a}_{12}\partial_{x_1x_2}\psi +\tilde{a}_{22}\partial_{x_2x_2}\psi=\tilde{f},\\[1mm]
\psi|_{x_2=w_1(x_1)}=0,\\[1mm]
\psi|_{x_2=w_2(x_1)}=m,\\
\end{cases}
\end{eqnarray*}
with $\partial_{x_2}\psi>0$ for $\bx\in\overline{\Omega}$, and $0<\psi<m$ for $\bx\in \Omega$.

\subsection{Euler-Lagrange coordinate transformation}
The Euler-Lagrange coordinate transformation from $\bx=(x_1, x_2)$ to $\bz=(z_1, z_2)$,
 which will be frequently used in this paper,
is given by
\begin{equation}\label{La coordinates}
(z_1, z_2)=(x_1, \psi(x_1,x_2)).
\end{equation}
The corresponding Jacobian of the transformation is
\begin{equation}\label{equ:3.2Jacob}
\left|\frac{D\bz}{D\bx}\right|
=\frac{\partial \psi}{\partial x_2}.
\end{equation}

Let $\varphi=x_2$.
Then we have
\begin{eqnarray*}
&& \partial_{z_1}\varphi=-\frac{\partial_{x_1}\psi}{\partial_{x_2}\psi},\qquad\qquad\,\,
\partial_{z_2}\varphi=\frac{1}{\partial_{x_2}\psi},\\
&& \rho u_1=\partial_{x_2}\psi=\frac{1}{\partial_{z_2}\varphi},
\qquad\,
\rho u_2=-\partial_{x_1}\psi=\frac{\partial_{z_1}\varphi}{\partial_{z_2}\varphi}.
\end{eqnarray*}
Set
\begin{equation*}
\mathcal{Q}^2(\nabla_\bz\varphi):=\frac{1}{(\partial_{z_2}\varphi)^2}\big((\partial_{z_1}\varphi)^2+1\big).
\end{equation*}
In the new coordinates, we know that $(\mathbb{B},\mathbb{S})$ are functions of $z_2$.
Then, from the Bernoulli law,
\begin{equation}\label{equ:6.3}
\frac{1}{2}\mathcal{Q}^2(\nabla_\bz\varphi)+\rho^{\gamma+1}\mathbb{S}(z_2)=\rho^2\mathbb{B}(z_2),
\end{equation}
so that the density can be regarded as a function of $(\nabla_\bz\varphi, z_2)$:
$$
\rho=\rho(\nabla_\bz\varphi, z_2).
$$
From $\eqref{weaksolution}_3$,   for $\rho u_1 \neq 0$, we see that,
for any test function $\zeta$,
\begin{eqnarray}\label{weakLequation}
0&=&\int_{\Omega} \big(\partial_{x_1}\zeta \rho u_1 u_2+ \partial_{x_2}\zeta (\rho u_2^2+p)\big) \dd\bx\nonumber\\
&=&-\int_{\mathbb{R}\times[0, m]}\Big(\rho u_1 u_2\partial_{z_1} \zeta
    +(\rho^2 u_1 u_2^2+\rho u_1 p-\rho^2 u_1 u_2^2 )\partial_{z_2}\zeta\Big)\frac{1}{\rho u_1}\dd \mathbf{z}\nonumber\\
&=&-\int_{\mathbb{R}\times[0, m]} \big(u_2\partial_{z_1} \zeta+ p\partial_{z_2}\zeta\big) \dd \mathbf{z},
\end{eqnarray}
which is equivalent to
$$
\partial_{z_1}u_2+\partial_{z_2}p=0.
$$
Then the equation for $\varphi$ is
\begin{equation}\label{equ:6.14}
\partial_{z_1}\Big(\frac{\partial_{z_1}\varphi}{\rho(\nabla_\bz\varphi, z_2)\partial_{z_2}\varphi}\Big)
 +\partial_{z_2}p(\nabla_\bz\varphi, z_2)=0.
\end{equation}

\begin{remark}
Notice that the characteristic discontinuity in the Lagrangian coordinates is simply the horizontal straight line{\rm:} $z_2=m_d$.
This is the main advantage of the Lagrangian coordinates.
\end{remark}

\section{Existence of the Stream Function $\psi$ with Large Vorticity}\label{sec:existence smooth}

This section is devoted to solving {\bf Problem 3.3($m$)}. More precisely, we have

\begin{theorem}\label{prop:5.1}
Let boundary $\partial\Omega$ satisfy \eqref{Con:Boundary1}--\eqref{Con:Boundary2}.
Let $(\mathbb{B}, \mathbb{S})$ satisfy assumption  \eqref{assumption:3.40}--\eqref{assumption:3.47a}.
Then there exists
$\tilde{m}\geq\hat{m}>0$ such that, if
$m>\tilde{m}$,  there is
a solution $\psi\in
C^{2,\alpha}(\overline{\Omega})$ of {\bf Problem 3.3($m$)} so that
the stream function $\psi$ satisfies inequalities \eqref{ieq:subsonic}--\eqref{ieq:streamline} and
\begin{equation}\label{3.11}
|\nabla_{\bx}\psi|<(1-2\v) \hat{Q}(\psi)
\end{equation}
for some $\v>0$, where $\hat{m}$ is defined in
\S {\rm \ref{subset:asymptotic state}} and  $\hat{Q}(\psi)$ is defined by \eqref{hat-Q}.
\end{theorem}

\begin{proof}
We divide the proof into five steps.

\subsection{Step 1: Existence in bounded domains}
First, we focus on the existence of the following boundary value problem:
\begin{eqnarray}\label{3.12}
\begin{cases}
\tilde{a}_{11}\partial_{x_1x_1}\psi_L+\tilde{a}_{12}\partial_{x_1x_2}\psi_L+\tilde{a}_{22}\partial_{x_2x_2}\psi_L=\tilde{f}
 \qquad\, \, \text{in}\ \Omega_L,\\[2mm]
\psi=\frac{x_2-w_1(x_1)}{w_2(x_1)-w_1(x_1)} m\qquad\, \,\, \text{on}\ \partial \Omega_{L},
\end{cases}
\end{eqnarray}
where $\Omega_{L}:=\{\bx \,:\, |x_1|<2L , w_1(x_1)\le x_2\le w_2(x_1) \} $.

\smallskip
Then, by a standard argument (similar to \cite{Lieberman} for example) such as the corner estimates,
the change of coordinates, and the odd extension near the corners to guarantee the local existence,
we conclude that there exists a solution $\psi_{L}\in C^{2, \beta}(\Omega_L) \cap C^{1,\beta}(\overline{\Omega_L})$
of problem \eqref{3.12}.

\subsection{Step 2: Uniform estimates} Based on the existence in bounded domains,
we now make some uniform estimates that are independent of $L$.

\subsubsection{$0<\psi_L<m$ in $\Omega_L$}\label{subsubsec:5.2.1}

If $\psi_L< m$ is not true, then $D_L:=\Omega_{L}\cap\{\bx\,:\,\psi_L(\bx)> m\}$ is non-empty.
If $\psi_L$ achieves its interior maximum at point $\bx_0$, then
$$
\tilde{a}_{11}\partial_{x_1x_1}\psi_L+\tilde{a}_{12}\partial_{x_1x_2}\psi_L+\tilde{a}_{22}\partial_{x_2x_2}\psi_L< 0, \quad\, \tilde{Q}(\nabla_\bx\psi_L, \psi_L)=0
\qquad\, \mbox{at $\bx_0$}.
$$
Meanwhile, the Bernoulli function satisfies $\tilde{\rho}^{\gamma-1}\tilde{\mathbb{S}}=\tilde{\mathbb{B}}$, which means
\begin{eqnarray*}
\tilde{\rho}^3 \tilde{c}^2\big(\tilde{\rho}\tilde{\mathbb{B}}'-\frac{\tilde{\rho}^\gamma}{\gamma}\tilde{\mathbb{S}}'\big)
=\tilde{\rho}^4 \tilde{c}^2\big(\tilde{\mathbb{B}}'-\frac{\tilde{\mathbb{B}}}{\gamma\tilde{\mathbb{S}}}\tilde{\mathbb{S}}'\big)
=\tilde{\rho}^4\tilde{c}^2\tilde{\mathbb{B}}\big(\mbox{ln}(\tilde{\mathbb{B}}\tilde{\mathbb{S}}^{-\frac{1}{\gamma}})\big)'\ge 0.
\end{eqnarray*}
This is a contradiction, which implies that $\psi_L$ can not achieve its maximum inside $D_L$.
Thus, it follows from the boundary conditions that $\psi_L< m$ in $\Omega_L$.
Similarly, we can obtain that $\psi_L>0$ in $\Omega_L$.

\subsubsection{Uniform H\"{o}lder estimate}

Based on the uniform $L^{\infty}$-estimate, we can obtain  higher order regularity estimates for $\psi_{L}$.
In fact,  following  \cite{GT},
we deduce that
there exists $\mu=\mu(\frac{\Lambda}{\lambda})>0$ such that,
for any $\bx_{0}\in\Omega_{L}$ and  $\psi_{k}$ with $k\geq 4L$,
\begin{equation}\label{3.20}
[\psi_{k}]_{1,\mu;B_{1}(\bx_{0})\cap \Omega_{L}}
\leq C(\frac{\Lambda}{\lambda})
\Big(1+\|\nabla_{\bx} \psi_{k}\|_{0;B_{1}(\bx_{0})\cap \Omega_{L}}
+\frac{\|\tilde{f}\|_0}{\lambda}\Big).
\end{equation}
Furthermore, using the interpolation inequality and the uniform $L^{\infty}$-estimate,
we obtain
\begin{equation*}
\|\psi_{k}\|_{1;B_{1}(\bx_{0})\cap \Omega_{L}}
\leq\tau C(\frac{\Lambda}{\lambda})\Big(1+\|\nabla_\bx \psi_{k}\|_{0;B_{1}(\bx_{0})
	\cap \Omega_{L}}+\frac{\|\tilde{f}\|_0}{\lambda}\Big)
+C_{\tau}\Big(m+ \frac{\|\tilde{f}\|_0}{\lambda}\Big).
\end{equation*}
Take $\tau_{0}$ sufficiently small so that
$\tau C(\frac{\Lambda}{\lambda})\leq \frac{1}{2}$
if $\tau\leq \tau_{0}$. Then
\begin{equation}\label{3.21}
\|\psi_{k}\|_{1;B_{1}(\bx_{0})\cap \Omega_{L}}
\leq\tau C(\frac{\Lambda}{\lambda})\Big(1+\frac{\|\tilde{f}\|_0}{\lambda}\Big)
+C_{\tau}\Big(m+\frac{\|\tilde{f}\|_0}{\lambda}\Big).
\end{equation}
Thus, the H\"{o}lder estimate \eqref{3.20} becomes
\begin{equation}\label{3.22}
\begin{array}{lll}
\|\psi_{k}\|_{1,\mu;B_{1}(\bx_{0})\cap \Omega_{L}}
&=&\|\psi_{k}\|_{1;B_{1}(\bx_{0})\cap \Omega_{L}}
+[\psi_{k}]_{1,\mu;B_{1}(\bx_{0})\cap \Omega_{L}}\\[2mm]
&\leq& \big(1+C(\frac{\Lambda}{\lambda})\big)\|\psi_{k}\|_{1;B_{1}(\bx_{0})\cap \Omega_{L}}
+C(\frac{\Lambda}{\lambda})\big(m+\frac{\|\tilde{f}\|_0}{\lambda}\big)\\[2mm]
&\leq& C(\frac{\Lambda}{\lambda})\big(1+m+\frac{\|\tilde{f}\|_0}{\lambda}\big).
\end{array}
\end{equation}
Notice that, for any $\bx_1\in\overline{\Omega_L}$,
$$
\frac{|\nabla\psi_{k}(\bx_1)-\nabla\psi_{k}(\bx_2)|}{|\bx_1-\bx_2|^{\mu}}\leq
\begin{cases}
  \|\psi_{k}\|_{1,\mu;B_{1}(\bx_{0})\cap \Omega_{L}}\quad &\text{if}\ \bx_2\in B_{1}(\bx_{0})\cap \Omega_{L},\\[1.5mm]
 2\|\psi_{k}\|_{1;B_{1}(\bx_{0})\cap \Omega_{L}}\quad&\text{if}\ \bx_2\notin B_{1}(\bx_{0})\cap \Omega_{L}.
\end{cases}
$$
This, together with \eqref{3.21}--\eqref{3.22},  yields the following H\"{o}lder estimate:
\begin{equation}\label{3.23}
[\psi_{k}]_{1,\mu;\Omega_{L}}
\leq C(\frac{\Lambda}{\lambda})\Big(1+m+\frac{\|\tilde{f}\|_0}{\lambda}\Big).
\end{equation}
Thus, it follows from the standard Schauder estimate that, for any $k\geq4L$,
\begin{equation}\label{3.24}
\|\psi_k\|_{1,\alpha;\Omega_L}\leq C\Big(\frac{\|\tilde{f}\|_{0; \Omega_L}}{\lambda}+\|\psi_k\|_{0; \Omega_L}\Big).
\end{equation}

\subsubsection{Uniform subsonic flows}
Thanks to the elliptic cut-off, there exists a constant $C>0$, independent of $p_-$ or $m$,
such that $1\leq\frac{\Lambda}{\lambda}\leq C$.
Therefore, by \eqref{3.23},
\begin{equation*}
[\psi_{k}]_{1,\mu;\Omega_{L}}
\leq C\Big(1+m+\frac{\|\tilde{f}\|_0}{\lambda}\Big).
\end{equation*}

Following the same argument as in (3.33) of \cite{DXX}, we have
\begin{equation}\label{equ:4.10a}
1+C^{-1}m+C^{-1}m^{\frac{\gamma-1}{2}}\leq 1+m+\frac{\|\tilde{f}\|_0}{\lambda}
\leq 1+Cm+Cm^{\frac{\gamma-1}{2}}.
\end{equation}

Note that
\begin{equation}\label{equ:4.11a}
\hat{Q}(\psi)\geq C^{-1}m^{\frac{\gamma+1}{2}}.
\end{equation}
There exists $\tilde{m}>0$ large enough such that,
if $m>\tilde{m}$, we obtain \eqref{3.11}.

\subsection{Step 3: Uniform direction of the flows}\label{subsec:5.3}
We now show the following lemma.

\begin{lemma}\label{lem:4.12}
The solution of \eqref{3.12} satisfies
\begin{equation}\label{ine:4.12}
\partial_{x_2}\psi >0\,\qquad\mbox{in }\Omega_L.
\end{equation}
\end{lemma}

\begin{proof}
We divide the proof into three steps.

\smallskip
1. \emph{Uniform direction of the flow at the boundary.}
On the upper boundary $x_2=w_2(x_1)$, $\psi_L=m$.
From $\psi_L<m$ in $\Omega_L$ and $\psi_L\in C(\overline{\Omega_L})$,
$\partial_{x_2}\psi_L>0$ holds by Hopf's lemma as in Lemma 6.1 of \cite{ChenDengXiang}.
In the same way, we can see that $\partial_{x_2}\psi_L>0$ on $x_2=w_1(x_1)$.

On the left boundary $x_1=-2L$, $\psi(-2L, x_2)=\frac{x_2-w_1(-2L)}{w_2(-2L)-w_1(-2L)} m$,
which implies
\begin{equation}\label{equ:5.10leftboundary}
\partial_{x_2}\psi(-2L, x_2)=\frac{m}{w_2(-2L)-w_1(-2L)}>0.
\end{equation}
It is the same on the right boundary $x_1=2L$.
Therefore, we have
\begin{equation}\label{boundary 4.13}
\partial_{x_2} \psi >0\qquad\mbox{on }\partial\Omega_L.
\end{equation}

\medskip
2. \emph{Monotonicity of $\psi$ with respect to $x_2$.}
In this step, we show
\begin{equation*}
\partial_{x_2}\psi\geq0\qquad\mbox{in }\Omega_L.
\end{equation*}

Since there is no sign assumption on the second derivatives,
we cannot use the energy estimates to show the non-negativity
of $\partial_{x_2}\psi$.
Therefore, we have to develop an alternative way to achieve our goal.
The main idea is to shift the coordinates.
Indeed, from \S \ref{subsubsec:5.2.1},
$0<\psi<m$ holds at the interior points of $\Omega_L$.
Extend solution $\psi$ such that $\psi=m$ when $x_2>w_2(x_1)$,
and $\psi=0$ when $x_2<w_1(x_1)$.
Then $\psi$ is a Lipschitz continuous function in $\mathbb{R}^2$.
Let $\psi_{\varepsilon}(x_1,x_2):=\psi(x_1,x_2+\varepsilon)$,
then $\psi<\psi_{\varepsilon}$ at the interior point of $\Omega_L$ when $\varepsilon$ is sufficiently large.
Now let $\varepsilon$ be smaller, and let $\varepsilon_0>0$ be the first value that the strict inequality
does not hold in the interior point of $\Omega_L$.
For this $\varepsilon_0$, $\psi\leq\psi_{\varepsilon_0}$ in $\Omega_L$.
Let $\Omega^0_L:=\{\bx\,:\, \psi(\bx)=\psi_{\varepsilon_0}(\bx)\}\cap\Omega_L$.

First, by \eqref{boundary 4.13}, $\psi$ is decreasing linearly near boundary $x_2=w_2(x_1)$;
that is, there is a small constant $c>0$ such that $m-\psi(\bx)>c(w_2(x_1)-x_2)$
near the upper boundary.
Clearly, $\psi<\psi_{\varepsilon_0}$ on the boundaries: $x_2=w_1(x_1)$, $x_1=2L$, and $x_1=-2L$.
Thus, we have
$$
\textrm{dist}\{\Omega_L^0,\, \partial\Omega_L\}>0.
$$
Let $P_0=(x_1^{(0)},x_2^{(0)})$ be the point on the boundary of $\Omega_L^0$ such that
$$
\textrm{dist}\{P_0,\, \partial\Omega_L\}=\textrm{dist}\{\Omega_L^0,\, \partial\Omega_L\}.
$$
Note that $\psi_{\varepsilon_0}-\psi=0$ at $P_0$,
and boundary $\partial\Omega_L^0$
satisfies the interior ball condition at $P_0$ from the side that $\psi_{\varepsilon_0}-\psi>0$.
Since  $P_0\in\Omega_L$, $0<\psi(P_0)=\psi_{\varepsilon_0}(P_0)<m$.
Near $P_0$, $\psi$ and $\psi_{\varepsilon_0}$ are both smooth and satisfy the equation:
\begin{equation*}
\rm{div}\big(\frac{\nabla\psi}{\rho}\big)
=\rho \mathbb{B}'-\frac{\rho^{\gamma}}{\gamma}\mathbb{S}',
\end{equation*}
and the equation:
\begin{equation*}
\rm{div}\big(\frac{\nabla\psi_{\varepsilon_0}}{\rho_{\varepsilon_0}}\big)
=\rho_{\varepsilon_0} \mathbb{B}_{\varepsilon_0}'
-\frac{(\rho_{\varepsilon_0})^{\gamma}}{\gamma}\mathbb{S}_{\varepsilon_0}'
\end{equation*}
respectively,
where $\rho=\rho(\psi,|\nabla\psi|^2)$, $\rho_{\varepsilon_0}=\rho(\psi_{\varepsilon_0},
|\nabla\psi_{\varepsilon_0}|^2)$, $\mathbb{B}_{\varepsilon_0}'=\mathbb{B}'(\psi_{\varepsilon_0})$,
and $\mathbb{S}_{\varepsilon_0}'=\mathbb{S}'(\psi_{\varepsilon_0})$.
Subtraction of the two equations above yields that
$W:=\psi_{\varepsilon_0}-\psi$ satisfies the following elliptic equations of second order:
\begin{equation*}
\check{a}_{11}\partial_{x_1x_1}W+\check{a}_{12}\partial_{x_1x_2}W+\check{a}_{22}\partial_{x_2x_2}W+b_1\partial_{x_1}W+b_2\partial_{x_2}W+cW=0,
\end{equation*}
where functions $\check{a}_{ij}$, $b_{i}$, and $c$ are uniformly bounded and actually smooth, due to the interior estimates.

Next, consider the equation of $W$ in the small neighborhood of $P_0$ such that $|x_1-x_1^{(0)}|\leq d_0$,
where the small constant $d_0$ will be determined later.
Let
$$
\tilde{W}:=\frac{W}{e^{2\alpha d_0}-e^{\alpha(x_1-x_1^{(0)})}}.
$$
Then $\tilde{W}(P_0)=0$, and boundary $\tilde{W}=0$ satisfies the interior ball condition at $P_0$
from the side that $\tilde{W}>0$.
By a direct calculation, $\tilde{W}$ satisfies the following second-order elliptic equation:
\begin{equation*}
\check{a}_{11}\partial_{x_1x_1}\tilde{W}+\check{a}_{12}\partial_{x_1x_2}\tilde{W}+\check{a}_{22}\partial_{x_2x_2}\tilde{W}
  +\check{b}_1\partial_{x_1}\tilde{W}+b_2\partial_{x_2}\tilde{W}+\check{c}\tilde{W}=0,
\end{equation*}
where
$$
\check{b}_1=b_1-\frac{2\alpha^2\check{a}_{11}e^{\alpha(x_1-x_1^{(0)})}}{e^{2\alpha d_0}-e^{x_1-x_1^{(0)}}}, \qquad
\check{c}=-\frac{\big(\check{a}_{11}\alpha^2+{b}_1\alpha\big)e^{\alpha(x_1-x_1^{(0)})}-c\big(e^{2\alpha d_0}
    -e^{\alpha(x_1-x_1^{(0)})}\big)}{e^{2\alpha d_0}-e^{\alpha(x_1-x_1^{(0)})}}.
$$
Since $b_1$ and $c$ are bounded, and $\check{a}_{11}\geq\alpha>0$,
then $\check{c}<0$ when $\alpha$ is sufficiently small.
Therefore, by Hopf's lemma, we have
\begin{equation*}
\partial_{\boldsymbol \nu}W\big|_{P_0}
=\big(e^{2\alpha d_0}-e^{\alpha(x_1-x_1^{(0)})}\big)
 \partial_{\boldsymbol \nu}\tilde{W}\big|_{P_0}\neq0.
\end{equation*}
This is a contradiction, which implies that $\nabla W=0$ at $P_0$.

Thus, for any $\varepsilon>0$, $\psi_{\varepsilon}>\psi$ in the interior of $\Omega_L$.
This means that $\psi$ is increasing in $\Omega_L$.
Therefore, $\partial_{x_2}\psi \geq0$ in $\Omega_L$.

\medskip
3. \emph{Strict positivity of $\partial_{x_2}\psi$}.
Now we show inequality \eqref{ine:4.12} in $\Omega_L$.
Based on Step 2, we know
$$
\cu:=\partial_{x_2}\psi \geq0 \qquad\mbox{in $\Omega_L$}.
$$
Now we derive the equation of $\cu$ as follows:
Taking derivative $\partial_{x_2}$ on \eqref{equ:3.14}, we obtain from a direct calculation that
\begin{eqnarray}\label{E:positive velocity energy}
-\partial_{x_2}\omega
&=&\partial_{x_1}\Big(\frac{\partial_{x_1x_2}\psi}{\rho}\Big)+\partial_{x_2}\Big(\frac{\partial_{x_2x_2}\psi}{\rho}\Big)
  -\partial_{x_1}\Big(\frac{\partial_{x_1}\psi\partial_{x_2}\rho}{\rho^2}\Big)-\partial_{x_2}\Big(\frac{\partial_{x_2}\psi\partial_{x_2}\rho}{\rho^2}\Big),
\end{eqnarray}
where the derivative of vorticity $\omega$ is
\begin{eqnarray*}
-\partial_{x_2}\omega
=\big(\mathbb{B}'-\rho^{\gamma-1}\mathbb{S}'\big)\partial_{x_2}\rho
+\rho \big(\mathbb{B}''-\frac{\rho^{\gamma-1}}{\gamma}\mathbb{S}''\big)\partial_{x_2}\psi.
\end{eqnarray*}

Plugging \eqref{E:2.22} and \eqref{E:2.23} into equation \eqref{E:positive velocity energy},
 we have the following second-order equation for $\cu$:
\begin{equation}\label{equ:velocity positive}
\partial_{x_i}(\bar{a}_{ij}\partial_{x_j}\cu)
+2\bar{c}_i\partial_{x_i}\cu+(\partial_{x_i}\bar{c}_i+\bar{d})\cu=0,
\end{equation}
where
\begin{equation*}
\bar{a}_{11}=\frac{\rho^2(c^2-q^2)+(\partial_{x_1}\psi)^2}{\rho^3(c^2-q^2)},\quad\,
\bar{a}_{12}=\bar{a}_{21}=\frac{\partial_{x_1}\psi\partial_{x_2}\psi}{\rho^3(c^2-q^2)},\quad\,
\bar{a}_{22}=\frac{\rho^2(c^2-q^2)+(\partial_{x_2}\psi)^2}{\rho^3(c^2-q^2)},
\end{equation*}
and
\begin{align*}
&\bar{c}_1=\frac{\partial_{x_1}\psi(\rho^{\gamma}\mathbb{S}'-\rho \mathbb{B}')}{\rho^2(c^2-q^2)},\quad\,
\bar{c}_2=\frac{\partial_{x_2}\psi(\rho^{\gamma}\mathbb{S}'-\rho \mathbb{B}')}{\rho^2(c^2-q^2)},\quad\,
\bar{d}=-\Big(\frac{(\rho^{\gamma-1}\mathbb{S}'- \mathbb{B}')^2}{c^2-q^2}+ (\mathbb{B}''-\frac{\rho^{\gamma-1}}{\gamma}\mathbb{S}'')\Big)\rho.
\end{align*}
To apply the  maximum principle, we introduce $\bar{\cu}=e^{-\alpha x_2}\cu$.
Then we have
\begin{equation}\label{equ:5.18}
\sum\limits_{i,j=1}^2\partial_{x_i}(\bar{a}_{ij}\partial_{x_j}\bar{\cu})
+2\sum\limits_{i=1}^2(\alpha\bar{a}_{2i}+\bar{c}_i)\partial_{x_i}\bar{\cu}
+\big(\alpha^2\bar{a}_{22}+\alpha(\partial_{x_i}\bar{a}_{i2}+b_2+c_2)+\bar{d}+\partial_{x_i}\bar{c}_i\big)\bar{\cu}=0,
\end{equation}
where $\bar{\cu}$ is a nonnegative function. Choosing $\alpha>0$ sufficiently large so that
$$
\alpha^2\bar{a}_{22}+\alpha\big(\partial_{x_i}\bar{a}_{i2}+b_2+c_2\big)+\bar{d}+\partial_{x_i}\bar{c}_i>0,
$$
by the strong maximum principle,
we can show that $\bar{\cu}>0$, which means that $\cu>0$ in $\Omega_L$.
This completes the proof of the proposition.
\end{proof}

\smallskip
From Lemma \ref{lem:4.12}, we have the following corollary.
\begin{corollary}\label{cor:5.10speed}
The speed is of non-degeneracy, i.e. $q>0$ in $\overline{\Omega_L}$.
\end{corollary}

\subsection{Step 4: Existence of solutions in the unbounded domain}
By the Arzel\`{a}-Ascoli lemma and a diagonal procedure, there exists a subsequence $\psi_{k_{l}}$
such that
$$
\psi_{k_{l}}\rightarrow\psi\qquad  \text{in}\ C^{2,\alpha'}(K)\,\,\,
\text{for any compact set}\ K\subset \Omega\ \text{and}\ \alpha'<\alpha.
$$
Here, $\psi$ satisfies the following problem:
$$
\begin{cases}
 \tilde{a}_{11}\partial_{x_1x_1}\psi+\tilde{a}_{12}\partial_{x_1x_2}\psi+\tilde{a}_{22}\partial_{x_2x_2}\psi=\tilde{f}\qquad
    \text{in}\ \Omega,\\[2mm]
\psi|_{x_2=w_1(x_1)}=0,\\[1mm]
\psi|_{x_2=w_2(x_1)}=m,\\
\end{cases}
$$
with the estimates that
$0\le\psi(\bx)\le m$ for  $\bx\in\Omega$
and, for any compact set $K\subset\Omega$,
\begin{equation}\label{3.25}
\|\psi\|_{2,\alpha;K}\leq C\Big(\frac{\|\tilde{f}\|_{\alpha; K}}{\lambda}+\|\psi\|_{0; K}\Big).
\end{equation}

Similarly to the previous argument in \S \ref{subsubsec:5.2.1},
we can obtain  \eqref{ieq:streamline} and \eqref{3.11} for $\bx\in\Omega$.

\subsection{Step 5: Non-degeneracy of the streamlines in the unbounded domain}
We now prove that there exists a constant $\delta>0$ such that
\begin{equation}\label{ine:5.32}
\partial_{x_2}\psi(\bx)\geq \delta>0\qquad \mbox{for }\bx\in\bar{\Omega}.
\end{equation}
From \S \ref{subsec:5.3},
we obtain that $\partial_{x_2}\psi(\bx)\geq0$ for $\bx\in\bar{\Omega}$.
Then, again by equation \eqref{equ:5.18} and the strong maximum principle,
we find that $\partial_{x_2}\psi(\bx)>0$ in $\Omega$.
Therefore, it suffices to show that inequality \eqref{ine:5.32} holds at infinity.

\subsubsection{Non-degeneracy at infinity}\label{subsubsec:non-degeneracy}
From the previous proof, we have shown that, in any compact set $K\subset\Omega$, $\partial_{x_2}\psi$ has a positive lower bound.
We now show that the streamline does not degenerate at the far field,
which means that $\partial_{x_2}\psi$ does not vanish as $x_1$ goes to $\pm\infty$.

Without loss of generality, it suffices to show this for the case that $x_1\to -\infty$, since the argument for the case that $x_1\to \infty$ is similar.
We assume
$$
\varliminf_{x_2\rightarrow-\infty}\partial_{x_2}\psi
=\varliminf_{x_1\rightarrow-\infty}\inf_{w_1(x_1)\le x_2\le w_2(x_1)}\partial_{x_2}\psi(x_1, x_2)=0.
$$
This means that there exists a sequence $\{x_1^n\}$ such that
$$
\inf_{w_1(x_1^n)\le x_2\le w_2(x_1^n)}\partial_{x_2}\psi(x_1^n, x_2)\rightarrow 0  \qquad\mbox{as $n\rightarrow\infty$}.
$$
It is convenient to introduce the following new coordinates $\by=(y_1,y_2)$ to flatten the boundaries of the
nozzle:
$$
(y_1, y_2)=(y_1(\bx), y_2(\bx)):=( x_1, \frac{x_2-w_1(x_1)}{w_2(x_1)-w_1(x_1)}).
$$
Then the nozzle becomes $\R\times[0,1]$ in the $\by$--coordinates.
Notice that the coordinate transform is invertible, since
$$
 \mbox{det}\left[\begin{array}{ccc}
\frac{\partial y_1}{\partial x_1}&\, &\frac{\partial y_1}{\partial
x_2}\\[2mm]
\frac{\partial y_2}{\partial x_1}&\,& \frac{\partial y_2}{\partial x_2}
 \end{array}\right] =
 \left| \begin{array}{ccc}
1&\,&0\\[2mm]
-\frac{x_2(w_2'(x_1)-w_1'(x_1))}{(w_2(x_1)-w_1(x_1))^2}&\, &\frac{1}{w_2(x_1)-w_1(x_1)}
 \end{array}\right| = \frac{1}{w_2(x_1)-w_1(x_1)}\neq0.
$$
From the uniform $C^{2, \alpha}$--regularity of the boundary,
we see that the map between $\bx$ and $\by$ is invertible and $C^{2, \alpha}$.

Define $\tilde{\psi}(\by)$ as $\psi(\bx)=\tilde{\psi}(y_1(\bx), y_2(\bx))$,
which is a $C^{2,\alpha}$--function due to the chain rule. Set
$$
\psi^{(n)}(\bx):=\tilde{\psi}(y_1(x_1-x_1^n, x_2), y_2(x_1-x_1^n, x_2)),
\qquad
y_2^n:=y_2(x_1-x_1^n, x_2).
$$
For $\bx\in K=[-1, 1]\times\mathbb{R}\cup\Omega$,
$(\psi^n, y_2^n)$ are uniformly $C^{2,\alpha}$--bounded.
From the Arzel\`{a}-Ascoli lemma, we obtain that
$(\psi^n, y_2^n)\rightarrow (\bar{\psi}, \bar{y_2})$.
Then $\bar{\psi}$ satisfies the following equations for ${\bf x}\in \mathbb{R}\times(0, 1)$:
\begin{eqnarray*}
\begin{cases}
\nabla\bar{\psi}=\bar{\rho}(-\bar{u}_2, \bar{u}_1),\\[1mm]
\mathbb{B}(\bar{\psi})=\frac{1}{2}|\bar{\bu}|^2+\frac{\gamma \bar{p}}{(\gamma-1)\bar{\rho}},\\[1mm]
\mathbb{S}(\bar{\psi})= \frac{\gamma \bar{p}}{(\gamma-1)\bar{\rho}^\gamma},\\[1mm]
\partial_{x_1}\bar{u}_2-\partial_{x_2}\bar{u}_1=-\bar{\rho} \mathbb{B}'(\bar{\psi})+\frac{\bar{\rho}^{\gamma}}{\gamma}\mathbb{S}'(\bar{\psi}),
\end{cases}
\end{eqnarray*}
and the boundary condition: $(\bar{\rho} \bar{u}_2)|_{x_2=j}=0$ for $j=0,1$.
On the upper and lower boundary, $\bar{\psi}$ satisfies that $\bar{\psi}|_{x_2=j}=jm$,
and $\partial_{x_2}\bar{\psi}\ge 0$ in $K$.
Moreover, from the equation and boundary conditions that $\bar{\psi}$ satisfies,
similarly to the previous proof,
we can obtain that $\partial_{x_2}\bar{\psi}$ has the positive lower bound in $K$
via Hopf's lemma along the boundary and the strong maximum principle.
This contradicts the fact that
$$
\inf \partial_{x_2}\psi(x_1^n, x_2)\rightarrow 0 \qquad\mbox{as $n\rightarrow\infty$}.
$$
Thus, the streamline is not degenerate at the far field.
This completes the proof.
\end{proof}

\begin{remark} Estimate \eqref{3.11} in Theorem {\rm \ref{prop:5.1}} implies that
the cut-off function introduced in \eqref{cut-off using} can be removed,
which implies that the solution of Theorem {\rm \ref{prop:5.1}}
is actually a solution of {\bf Problem 3.1($m$)}.
\end{remark}

\section{Further Properties of Smooth Solutions}\label{sec:uniqueness}

\subsection{Asymptotic states at the inlet}\label{subset:5.1far field}
In order to study the uniqueness, the subsonic-sonic limit, and the existence of solutions
with discontinuities, we need a uniform state at the far field such that
the restriction on the largeness of $m$ can be released.
Therefore, we study the far field behavior of the solutions obtained via Theorem \ref{prop:5.1}.
First, we consider the states at the inlet: $x_1=-\infty$.

\begin{lemma}\label{prop:5.2}
The solution obtained in Theorem {\rm \ref{prop:5.1}} satisfies
the asymptotic behavior at the inlet as stated in Theorem {\rm \ref{thm:smooth}}.
\end{lemma}

\begin{proof}
First, by \eqref{ine:5.32},
$\partial_{x_2}\bar{\psi}\geq \delta>0$ at the inlet.
Based on this observation, we divide the proof into two steps.

\smallskip
1. We show that $u_2\equiv0$ at both $\pm\infty$.
Let $\bz=(z_1, z_2):=(x_1, \bar{\psi})$ and $\bar{\varphi}:=x_2$.
Then $\bar{\varphi}$ satisfies
\begin{equation}\label{Leqninfity}
\partial_{z_1}\Big(\frac{\partial_{z_1}\bar{\varphi}}{\rho(\nabla_\bz\bar{\varphi}, z_2)\partial_{z_2}\bar{\varphi}}\Big)
+\partial_{z_2}p(\nabla_\bz\bar{\varphi}, z_2)=0.
\end{equation}

To estimate $\bar{u}_2$, it is natural to introduce $\bar{\mathcal{V}}:=\partial_{z_1}\bar{\varphi}$
and the corresponding equation by taking $\partial_{z_1}$ on \eqref{Leqninfity}:
\begin{equation*}
0=
\partial_{z_1}\Big(\frac{\partial_{z_1 z_1}\bar{\varphi}}{\rho(\nabla_\bz\bar{\varphi}, z_2)\partial_{z_2}\bar{\varphi}}\Big)
-\partial_{z_1}\Big(\frac{\partial_{z_1}\bar{\varphi}\partial_{z_1}\rho(\nabla_\bz\bar{\varphi},z_2)}{\rho^2(\nabla_\bz\bar{\varphi}, z_2)\partial_{z_2}\bar{\varphi}}\Big)
+\partial_{z_1}\Big(\frac{\partial_{z_1}\bar{\varphi}\partial_{z_1 z_2}\bar{\varphi}}{\rho(\nabla_\bz\bar{\varphi}, z_2)(\partial_{z_2}\bar{\varphi})^2}\Big)
+\partial^2_{z_1z_2}p(\nabla_\bz\bar{\varphi}, z_2).
\end{equation*}
From the state equation, we have
$$
\partial_{z_1}p(\nabla_\bz\bar{\varphi}, z_2)=c^2\partial_{z_1}\rho(\nabla_\bz\bar{\varphi}, z_2).
$$
For $\partial_{z_1}\rho$, we have
\begin{eqnarray*}
\partial_{z_1}\rho(\nabla_\bz\bar{\varphi}, z_2)
&=&\frac{\partial_{z_1}\bar{\varphi}}{(\partial_{z_2}\bar{\varphi})^2\rho(q^2-c^2)}\partial_{z_1}\bar{\mathcal{V}}
-\frac{(\partial_{z_1}\bar{\varphi})^2+1}{(\partial_{z_2}\bar{\varphi})^3\rho(q^2-c^2)}
\partial_{ z_2}\bar{\mathcal{V}},
\end{eqnarray*}
where $2\rho\mathbb{B}(z_2)-(\gamma+1)\rho^\gamma\mathbb{S}(z_2)=\rho(q^2-c^2)$.
A direct computation yields
\begin{eqnarray*}
&&\partial_{z_1}\Big(\big(-\frac{(\partial_{z_1}\bar{\varphi})^2}{\rho^3(\partial_{z_2}\bar{\varphi})^3(q^2-c^2)}+\frac{1}{\rho\partial_{z_2}\bar{\varphi}}\big)\partial_{z_1}\bar{\mathcal{V}}
+\big(\frac{\partial_{z_1}\bar{\varphi}((\partial_{z_1}\bar{\varphi})^2+1)}{\rho^3(\partial_{z_2}\bar{\varphi})^4(q^2-c^2)}
  +\frac{\partial_{z_1}\bar{\varphi}}{\rho(\nabla_z\bar{\varphi}, z_2)(\partial_{z_2}\bar{\varphi})^2}\big)
\partial_{ z_2}\bar{\mathcal{V}}\Big)\nonumber\\[2mm]
&&+\partial_{z_2}\Big(\frac{c^2\partial_{z_1}\bar{\varphi}}{(\partial_{z_2}\bar{\varphi})^2\rho(q^2-c^2)}\partial_{z_1}\bar{\mathcal{V}}
-\frac{c^2((\partial_{z_1}\bar{\varphi})^2+1)}{(\partial_{z_2}\bar{\varphi})^3\rho(q^2-c^2)}
\partial_{ z_2}\bar{\mathcal{V}}\Big)=0.
\end{eqnarray*}
By the standard Cacciopolli inequality and the Poincar\'{e} inequality, we have
\begin{equation*}
\int_{[-L, L]\times[0, m]}|\nabla_\bz\bar{\mathcal{V}}|^2\, \dd\bz
\le  C \int_{([-L-1, -L]\times[0, m])\cup ([L, L+1]\times[0, m])}|\nabla_\bz\bar{\mathcal{V}}|^2\, \dd\bz,
\end{equation*}
where  $C$ is independent of $L$.
From the fact that $|\nabla_\bz\bar{\mathcal{V}}|$ is bounded in $\mathbb{R}\times [0, m]$,
we can use the above estimate to obtain that $\nabla_\bz\bar{\mathcal{V}}$ is in $L^2$.
This implies
$$
\int_{[-L, L]\times[0, m]}|\nabla_\bz\bar{\mathcal{V}}|^2 \,\dd\bz \rightarrow 0 \qquad\mbox{as $L\rightarrow \infty$}.
$$
Hence, we have
\begin{equation*}
\int_{\mathbb{R}\times[0, m]}|\nabla_\bz\bar{\mathcal{V}}|^2\, \dd\bz \equiv 0.
\end{equation*}
Since  $\bar{\mathcal{V}}=0$ vanishes at $z_2=0$,
we conclude that $\bar{\mathcal{V}}\equiv0$, which leads to $\bar{u}_2\equiv0$.

\smallskip
2. In this step, we show that the limit solution $\bar{\psi}$
with $\bar{\psi}_{x_1}=0$ is unique.
Now $\bar{\psi}$ is a function of $x_2$ only.
If $\bar{\psi}(x_2)$ is not unique, then there are two different solutions $\bar{\psi}^{(i)}(x_2)$ for $i=1, 2$.
We employ the Lagrangian formulation \eqref{Leqninfity} again.
For each $i=1, 2$, the respective potential $\bar{\varphi}^{(i)}$ satisfies
$$
\partial_{z_2}p(\partial_{z_2}\bar{\varphi}^{i}, z_2)=0.
$$
Let $\varphi^{D}=\bar{\varphi}^{(1)}-\bar{\varphi}^{(2)}$.
Then we have
\begin{equation*}
\partial_{z_2}\Big(\int_1^2\frac{c^2\, \dd\tau}{\rho(\partial_{z_2}\bar{\varphi}^{(\tau)}, z_2)(\partial_{z_2}\bar{\varphi}^{(\tau)})^3(q^2-c^2)}
  \,\partial_{z_2}\varphi^{D}\Big)=0,
\end{equation*}
where $\bar{\varphi}^{(\tau)}=(2-\tau)\bar{\varphi}^{(1)}+(\tau-1)\bar{\varphi}^{(2)}$
for $\tau\in (1,2)$.
Multiplying $\varphi^{D}$ and then integrating on $[0, m]$, we have
 $$
 \int_0^m |\partial_{z_2}\varphi^{D}|^2\, \dd z_2\leq 0.
 $$
Since the flow is subsonic, and $\varphi^{D}$ vanishes at $z_2=0$ and $z_2=m$,
we conclude that $\bar{\psi}^{(1)}\equiv\bar{\psi}^{(2)}$, which is a contradiction.
Therefore, the special solution obtained in \S \ref{subset:asymptotic state}
is actually $\bar{\psi}$ by the uniqueness.
\end{proof}

\subsection{Asymptotic states at the outlet}\label{subset:asymptotic state plus}

We now study the asymptotic behavior of solutions obtained via Theorem \ref{prop:5.1}
at the outlet. Following the proof of Lemma \ref{prop:5.2},
we know that, if the asymptotic states exist, then the solutions obtained in Theorem \ref{prop:5.1}
satisfy the asymptotic behavior at the outlet, as stated in Theorem \ref{thm:smooth}.
Therefore, it suffices to consider the existence of the asymptotic states at the outlet.

We introduce a function $x_2(y)$:
$$
x_2(s)\,:\, [0, 1]\rightarrow [a, b],
$$
so that the streamline from the inlet $(-\infty, y)$
flows to the outlet $(\infty, x_2(y))$.
Define  $\tilde{u}_{1+}(y)=u_{1+}(x_2(y))$ and $\tilde{\rho}_{+}(y)=\rho_{+}(x_2(y))$.
From $\eqref{1.5}_1$, we have
\begin{equation}\label{mass-+infty1}
\int_0^y (\rho_-u_{1-})(t)\, \dd t= \int_0^{x_2(y)}(\rho_+u_{1+})(t)\, \dd t.
\end{equation}
Taking $\frac{d}{d y}$ on \eqref{mass-+infty1} yields
\begin{equation}\label{x_2-equation}
\frac{d x_2}{d y}=\frac{\rho_{-}u_{1-}}{\tilde{\rho}_{+}\tilde{u}_{1+}}(y),
\end{equation}
from which
\begin{equation}\label{b-a}
\int_0^1 \frac{\rho_{-}u_{1-}}{\tilde{\rho}_{+}\tilde{u}_{1+}}(y)\, \dd y=b-a.
\end{equation}

First, similar to \eqref{Ber--infty}--\eqref{entropy--infty},
the conservative quantities at the outlet are
\begin{eqnarray}
&& B_-(y)=B_+(x_2(y))= \frac{1}{2}|\tilde{u}_{1+}(y)|^2+\frac{\gamma p_{+}}{(\gamma-1)\tilde{\rho}_+(y)}, \label{Ber-+infty}\\
&& S_-(y)=S_+(x_2(y))=\frac{\gamma p_+}{(\gamma-1)\tilde{\rho}^\gamma_+(y)}. \label{entropy-+infty}
\end{eqnarray}
Then
\begin{eqnarray*}
&& \tilde{\rho}_+(y, p_+)=\big(\frac{\gamma}{\gamma-1}\big)^{\frac{1}{\gamma}}p_+^{\frac{1}{\gamma}} S^{-\frac{1}{\gamma}}_-(y),\\
&& \tilde{u}_{1+}(y, p_+)=\Big(2 B_-(y)-2\big(\frac{\gamma}{\gamma-1}\big)^{1-\frac{1}{\gamma}}p_+^{\frac{\gamma-1}{\gamma}}S^{\frac{1}{\gamma}}_-(y)\Big)^{\frac{1}{2}},
\end{eqnarray*}
where $p_+$ is the constant pressure at the outlet.
For the non-stagnation subsonic flow,
$0<\tilde{u}_{1+}(y, p_+)<\sqrt{2\frac{\gamma-1}{\gamma+1}B_-(y)}$, which implies
\begin{equation}\label{p_+-non-de-cont}
\underline{p}_+<p_+< \overline{p}_+,
\end{equation}
where
\begin{eqnarray*}
&&\overline{p}_+:=
\Big( \frac{1}{2}\big(\frac{\gamma}{\gamma-1}\big)^{\frac{1}{\gamma}-1} \min_{y\in[0,1]}(u_{1-}^2S_-^{-\frac{1}{\gamma}})(y)
   +p_-^{\frac{\gamma-1}{\gamma}}\Big)^{\frac{\gamma}{\gamma-1}},\nonumber\\
&&\underline{p}_+:=\Big(\frac{2}{\gamma+1}\Big)^{\frac{\gamma}{\gamma-1}}
\Big( \frac{1}{2}\big(\frac{\gamma}{\gamma-1}\big)^{\frac{1}{\gamma}-1}
  \max_{y\in[0,1]}(u_{1-}^2 S_-^{-\frac{1}{\gamma}})(y)
  + p_-^{\frac{\gamma-1}{\gamma}}\Big)^{\frac{\gamma}{\gamma-1}}.\nonumber
\end{eqnarray*}	
From \eqref{b-a}, we also have
\begin{equation}\label{x_2-equation1}
b-a=\int_0^1\frac{\rho_{-}u_{1-}}{\tilde{\rho}_{+}\tilde{u}_{1+}}(y)\textrm{d}y.
\end{equation}
In other words, if there exists a non-stagnation subsonic flow for {\bf Problem 2.1($m$)},
then there must have a constant pressure $p_+$ at the outlet satisfying \eqref{p_+-non-de-cont}--\eqref{x_2-equation1}.
Therefore, to show the existence of subsonic flows for {\bf Problem 2.1($m$)},
a necessary condition is to find the constant pressure $p_+$ satisfying \eqref{p_+-non-de-cont}--\eqref{x_2-equation1}.
For this, define
\begin{equation*}
J(p_+; p_-):=\int_0^1 \frac{\rho_{-}u_{1-}}{\tilde{\rho}_{+}\tilde{u}_{1+}}(y) \dd y.
\end{equation*}
In fact, we have the following lemma.

\begin{lemma}
For the given incoming flow $(u_{1-},S_-)$ at the inlet, there exists $\hat{m}>0$ such that,
for any given $m>\hat{m}$,
there is a unique $p_+$ with $p_+\in(\underline{p}_+,\overline{p}_+)$ so that $J(p_+; p_-)=b-a$.
Then density $\rho_+$ and the horizontal velocity $u_{1+}$
at the outlet are uniquely determined by \eqref{Ber-+infty} and \eqref{entropy-+infty}.
 \end{lemma}

\begin{proof}
Note that
\begin{align}\nonumber
\frac{\rho_{-}u_{1-}}{\tilde{\rho}_{+}\tilde{u}_{1+}}
=&\frac{\big(\frac{\gamma}{\gamma-1}\big)^{\frac{1}{\gamma}}p_-^{\frac{1}{\gamma}}(u_{1-}S_-^{-\frac{1}{\gamma}})(y)}
{\sqrt{2}\big(\frac{\gamma}{\gamma-1}\big)^{\frac{1}{\gamma}}p_+^{\frac{1}{\gamma}}S^{-\frac{1}{\gamma}}_-(y)
 \big(B_-(y)-(\frac{\gamma}{\gamma-1})^{1-\frac{1}{\gamma}}p_+^{\frac{\gamma-1}{\gamma}}S^{-\frac{1}{\gamma}}_-(y)\big)^{\frac{1}{2}}}\\\nonumber
=&\frac{ \big(\frac{\gamma}{\gamma-1}\big)^{\frac{1-\gamma}{2\gamma}} (u_{1-}S_-^{-\frac{1}{2\gamma}})(y)}
{\sqrt{2}(p_+ p_-^{-1})^{\frac{1}{\gamma}}
\big((\frac{\gamma}{\gamma-1})^{\frac{1}{\gamma}-1} B_-(y)S^{-\frac{1}{\gamma}}_-(y)-p_+^{\frac{\gamma-1}{\gamma}}\big)^{\frac{1}{2}}}.\nonumber
\end{align}
It follows from a direct calculation that $\frac{d J}{d p_+}>0$ for $p_+\in(\underline{p}_+,\overline{p}_+)$.
Therefore, in order to show that $J(p_+; p_-)=b-a$, it suffices to show that
$$
J(\underline{p}_+; p_-)<b-a<J(\overline{p}_+; p_-).
$$

First, by a direct calculation, we have
\begin{align*}
J(\underline{p}_+; p_-)
=\int_0^1  \frac{\big(\frac{\gamma}{\gamma-1}\big)^{\frac{1-\gamma}{2\gamma}}(\underline{p}_+^{-1} p_-)^{\frac{1}{\gamma}} (u_{1-}S_-^{-\frac{1}{2\gamma}})(y)}
{\big((\frac{\gamma}{\gamma-1})^{\frac{1-\gamma}{\gamma}}\big( (u_{1-}^2S_-^{-\frac{1}{\gamma}})(y)
 -\frac{2}{\gamma+1}\max_{y\in[0,1]}(u_{1-}^2 S_-^{-\frac{1}{\gamma}})(y)\big) + \frac{2(\gamma-1)}{\gamma+1}p_-^{\frac{\gamma-1}{\gamma}}\big)^{\frac{1}{2}}}\dd y.\nonumber
\end{align*}
Note that
\begin{align*}
 & \frac{ (u_{1-}S_-^{-\frac{1}{2\gamma}})(y)} {\big((\frac{\gamma}{\gamma-1})^{\frac{1}{\gamma}-1}
  \big((u_{1-}^2S_-^{-\frac{1}{\gamma}})(y)-\frac{2}{\gamma+1}\max_{y\in[0,1]}(u_{1-}^2S_-^{-\frac{1}{\gamma}})(y)\big) + \frac{2(\gamma-1)}{\gamma+1}p_-^{\frac{\gamma-1}{\gamma}}\big)^{\frac{1}{2}}}\nonumber\\
 &=\bigg(\big(\frac{\gamma}{\gamma-1}\big)^{\frac{1-\gamma}{\gamma}} \Big(1-\frac{2}{\gamma+1}\frac{\max_{y\in[0,1]}(u_{1-}^2S_-^{-\frac{1}{\gamma}})(y)}{(u_{1-}^2S_-^{-\frac{1}{\gamma}})(y)}\Big)
   +\frac{2(\gamma-1)}{\gamma+1}\frac{p_-^{\frac{\gamma-1}{\gamma}}}{(u_{1-}^2S_-^{-\frac{1}{\gamma}})(y)}\bigg)^{-\frac{1}{2}},
\end{align*}
and
\begin{equation*}
p_- \underline{p}_+^{-1}
=
\Big(\frac{\gamma+1}{2}\Big)^{\frac{\gamma}{\gamma-1}}
 \Big( \frac{1}{2}(\frac{\gamma}{\gamma-1})^{\frac{1-\gamma}{\gamma}} p_-^{-\frac{\gamma-1}{\gamma}}
  \max_{y\in[0,1]}(u_{1-}^2S_-^{-\frac{1}{\gamma}})(y) + 1 \Big)^{-\frac{\gamma}{\gamma-1}}.
\end{equation*}
Then it can directly be seen that $J(\underline{p}_+; p_-)<a-b$, when $m$ is sufficiently large, \emph{i.e.} $p_-$ is sufficiently large.

\smallskip
Next, let us consider $J(\overline{p}_+; p_-)$.
By the straightforward calculation, we have
$$
J(\overline{p}_+; p_-)
=\int_0^1 (\overline{p}_+^{-1} p_-)^{\frac{1}{\gamma}}
 \Big(1-\frac{\min_{y\in[0,1]}(u_{1-}^2S_-^{-\frac{1}{\gamma}})(y)}{(u_{1-}^2S_-^{-\frac{1}{\gamma}})(y)}\Big)^{-\frac{1}{2}}\dd y,
$$
where
\begin{equation*}
p_-\overline{p}_+^{-1} =\Big(\frac{1}{2}(\frac{\gamma}{\gamma-1})^{\frac{1-\gamma}{\gamma}} p_-^{-\frac{\gamma-1}{\gamma}} \min_{y\in[0,1]}(u_{1-}^2S_-^{-\frac{1}{\gamma}})(y)
+ 1\Big)^{-\frac{\gamma}{\gamma-1}}.
\end{equation*}
For $p_->\hat{p}>0$, the monotonicity implies that $p_- \overline{p}_+^{-1}>C(\hat{p})>0$ for some constant $C(\hat{p})$ depending on $\hat{p}$. Then
\begin{equation*}
J(\overline{p}_+; p_-)>C(\hat{p}) A,
\end{equation*}
where
\begin{equation*}
A:=\int_0^1  \frac{
 	(u_{1-}S_-^{-\frac{1}{2\gamma}})(y)} {\big((u_{1-}^2 S_-^{-\frac{1}{\gamma}} ) (y)
   -\min_{y\in[1,0]}(u_{1-}^2 S_-^{-\frac{1}{\gamma}} ) (y)\big)^{\frac{1}{2}}}\dd y.
 \end{equation*}

We now consider the value of $A$.
If there is $y_0\in(0,1 )$ such
that
$$
(u_{1-}^2 S_-^{-\frac{1}{\gamma}})(y_0)=\min_{y\in[1,0]}(u_{1-}^2 S_-^{-\frac{1}{\gamma}})(y),
$$
the $C^{1, 1}$--regularity implies that $A=\infty$.
When $y_0=0$ is the only minimum point of $(u_{1-}^2 S_-^{-\frac{1}{\gamma}})(y)$,
then $(u_{1-}^2 S_-^{-\frac{1}{\gamma}})'(0)\geq0$.
On the other hand, condition (2.16) requires that $(u_{1-}^2 S_-^{-\frac{1}{\gamma}} )'(0)\leq0$,
which implies the derivative vanishes at $y_0=0$, so that $A=\infty$.
Similarly, when $y_0=1$ is the only minimum point of $(u_{1-}^2 S_-^{-\frac{1}{\gamma}} )(y)$,
then $(u_{1-}^2 S_-^{-\frac{1}{\gamma}} )'(1)\leq0$.
On the other hand, condition (2.16) requires that $(u_{1-}^2 S_-^{-\frac{1}{\gamma}})'(1)\geq0$,
which implies the derivative vanishes at $y_0=1$, so that $A=\infty$.
Therefore, there exists $p_+<\overline{p}_+$ such that $J(p_+; p_-)=b-a$,
since  $J(p_+; p_-)$ is a increasing function of $p_+$.
\end{proof}

\subsection{Uniqueness of general non-degenerate flows}
Based on equation \eqref{equ:6.14}, we are now going to prove the uniqueness of solutions of {\bf Problem 3.1($m$)}.

\begin{lemma}\label{prop:6.1}
The solutions obtained in Theorem {\rm \ref{prop:5.1}} satisfying the asymptotic behavior
as in Theorem {\rm \ref{thm:smooth}} are unique.
\end{lemma}

\begin{proof}
Under the non-degenerate streamline condition \eqref{ine:5.32},
the uniqueness of flows in the $\bx$--coordinates is equivalent
to the uniqueness of flows in the Lagrangian coordinates.

Let $(\rho^{(1)}, u_1^{(1)}, u_2^{(1)}, p^{(1)})$ and $(\rho^{(2)}, u_1^{(2)}, u_2^{(2)}, p^{(2)})$
be two solutions of {\bf Problem 3.1($m$)}.
Then we introduce the corresponding Lagrangian coordinates for these two solutions as $\varphi^{(1)}$ and $\varphi^{(2)}$
so that
\begin{eqnarray*}
\partial_{z_1}\Big(\frac{\partial_{z_1}\varphi^{(i)}}{\rho(\nabla_\bz\varphi^{(i)}, z_2)\partial_{z_2}\varphi^{(i)}}\Big)+\partial_{z_2}p(\nabla_\bz\varphi^{(i)}, z_2)=0
\qquad\,\,\,\mbox{for $i=1,2$}.
\end{eqnarray*}
Subtracting these two equations, we find that $\hat{\varphi}=\varphi^{(1)}-\varphi^{(2)}$ satisfies  the following equation:
\begin{equation*}
\partial_{z_1}\big(\hat{a}_{11}\partial_{z_1}\hat{\varphi}+\hat{a}_{12}\partial_{z_2}\hat{\varphi}\big)+\partial_{z_2}\big(\hat{a}_{21}\partial_{z_1}\hat{\varphi}+\hat{a}_{22}\partial_{z_2}\hat{\varphi}\big)
=0,
\end{equation*}
where
\begin{eqnarray*}
&&\hat{a}_{11}=\int_1^2\Big(\frac{1}{\rho(\nabla_\bz\varphi^{(\tau)}, z_2)\partial_{z_2}\varphi^{(\tau)}}
    -\frac{(\partial_{z_1}\varphi^{(\tau)})^2}{\rho^3(\nabla_\bz\varphi^{(\tau)}, z_2)(\partial_{z_2}\varphi^{(\tau)})^3\big((q^{(\tau)})^2-(c^{(\tau)})^2\big)}\Big)\, \dd \tau, \\[2mm]
&&\hat{a}_{12}=-\int_1^2\Big(\frac{\partial_{z_1}\varphi^{(\tau)}}{\rho(\nabla_\bz\varphi^{(\tau)}, z_2)(\partial_{z_2}\varphi^{(\tau)})^2}
-\frac{\partial_{z_1}\varphi^{(\tau)}\big((\partial_{z_1}\varphi^{(\tau)})^2+1\big)}{\rho^3(\nabla_\bz\varphi^{(\tau)},z_2)(\partial_{z_2}\varphi^{(\tau)})^4\big((q^{(\tau)})^2-(c^{(\tau)})^2\big)}\Big)\, \dd \tau,\\[2mm]
&&\hat{a}_{21}=\int_1^2\frac{c^2\,\partial_{z_1}\varphi^{(\tau)}}{\rho(\nabla_\bz\varphi^{(\tau)}, z_2)(\partial_{z_2}\varphi^{(\tau)})^2\big((q^{(\tau)})^2-(c^{(\tau)})^2\big)}\, \dd \tau,\\[2mm]
&&\hat{a}_{22}=-\int_1^2\frac{c^2\big((\partial_{z_1}\varphi^{(\tau)})^2+1\big)}{\rho(\nabla_\bz\varphi^{(\tau)}, z_2)(\partial_{z_2}\varphi^{(\tau)})^3\big((q^{(\tau)})^2-(c^{(\tau)})^2\big)}\,\dd \tau,
\end{eqnarray*}
and
$
\varphi^{(\tau)}=(2-\tau)\varphi^{(1)}+(\tau-1)\varphi^{(2)}
$
for some $\tau\in(1,2)$, and $(q^{(\tau)})^2$ and $(c^{(\tau)})^2$ are defined in the same way as for $\varphi^{(\tau)}$.

On the other hand, noticing that $\hat{\varphi}=0$ on $z_2=0,m$, we  have the Poinc\'{a}re inequality:
\begin{equation*}
\int_0^m|\hat{\varphi}(\bz)|^2dz_2\le C(m)\int_0^m|\partial_{z_2}\hat{\varphi}(\bz)|^2\dd z_2,
\end{equation*}
where $C(m)>0$ depends only on $m$.

Now, we can select a smooth cut-off function $\eta_L(z_1)$ so that
$\eta_L\equiv1$ when $|z_1|\le L$ and $\eta_L\equiv0$ when $|z_1|\ge L+1$.
Taking  $\eta_L^2\hat{\varphi}$ as the test function and using the standard Cacciopolli inequality, we have
\begin{equation*}
\int_{[-L, L]\times[0, m]}|\nabla_\bz\hat{\varphi}|^2 \, \dd\bz
\le C(m)\int_{[-L-1, -L]\times[0, m]\cup[L, L+1]\times[0, m]}|\nabla_\bz\hat{\varphi}|^2\, \dd\bz,
\end{equation*}
where $C(m)$ is independent of $L$.
From the fact that $|\nabla_\bz\hat{\varphi}|$ is bounded in $\mathbb{R}\times [0, m]$,
we can use the above estimate to show that $\nabla_\bz\hat{\varphi}$ in $L^2$. This implies
$$
\int_{[-L-1, -L]\times[0, m]\cup[L, L+1]\times[0, m]}|\nabla_\bz\hat{\varphi}|^2\, \dd\bz \rightarrow 0 \qquad\mbox{as $L\rightarrow \infty$},
$$
and then
$$
\int_{[-L, L]\times[0, m]}|\nabla_\bz\hat{\varphi}|^2\, \dd\bz \rightarrow 0  \qquad\mbox{as $L\rightarrow \infty$}.
$$
Therefore, we obtain
\begin{equation*}
\iint_{\mathbb{R}\times[0, m]}|\nabla_\bz\hat{\varphi}|^2 \dd z_1 \dd z_2 \equiv 0.
\end{equation*}
Since $\hat{\varphi}$ vanishes at $z_2=0$, we conclude that $\hat{\varphi}\equiv0$.
This shows the uniqueness of non-degenerate flows of the system.
\end{proof}

\subsection{Uniform estimate on $\theta$ and $p$}
Based on the non-degeneracy of the speed from Corollary \ref{cor:5.10speed},
we consider the maximum principle of the direction of velocities and the pressure via the second-order equations
of $\theta$ and $p$, respectively.
This is crucial for our later study of the existence of weak solutions with vortex sheets
and entropy waves via the compensated compactness framework in \cite{ChenHuangWang}.

\begin{lemma}\label{lem:5.4qmax}
	For the given solution $\psi$, let
	\begin{equation}\label{5.13a}
	(q \cos\theta,q \sin\theta):=(u_1, u_2).
	\end{equation}
Then both $\theta$ and $p$ satisfy the maximum principle at the interior point{\rm ;} that is,
$\theta$ and $p$ cannot attain the maximum at the interior point
so that
\begin{eqnarray}
&&\min_{\partial\Omega}p=\min_{\Omega}p\leq p\leq\max_{\Omega}p=\max_{\partial\Omega}p,\label{est:theta5.20a}\\
&&\min_{\partial\Omega}\theta=\min_{\Omega}\theta\leq\theta\leq\max_{\Omega}\theta=\max_{\partial\Omega}\theta.\label{est:theta5.20}
\end{eqnarray}
\end{lemma}
	
For the notational simplicity, inequality \eqref{est:theta5.20} is denoted
as $|\bar{\theta}|\leq\theta_B$ later on.

\begin{proof} We divide the proof into four steps.

\smallskip
1. Consider the equations of the Euler system corresponding to the genuinely nonlinear characteristic fields:
	\begin{eqnarray}\label{A.5}
	\begin{cases}
	\partial_{x_1}(\rho u_1)+\partial_{x_2}(\rho u_2)=0,\\[1mm]
	\partial_{x_1}u_2-\partial_{x_2}u_1=\omega.
	\end{cases}
	\end{eqnarray}
From $(\ref{A.5})_1$ and the Bernoulli law \eqref{E:2.22}, we have
	\begin{eqnarray*}
	(1-M^2)\cos\theta\,\partial_{x_1} q - q \sin\theta\,\partial_{x_1} \theta +(1-M^2)\sin\theta\,\partial_{x_2} q + q \cos\theta\,\partial_{x_2} \theta=0.
	\end{eqnarray*}
Here we observe that $(\mathbb{B}', \mathbb{S}')$ do not explicitly appear in the equation above.
	From $(\ref{A.5})_2$, we have
	\begin{equation*}
	\omega
	=\partial_{x_1}(q \sin\theta)-\partial_{x_2}(q \cos\theta)
	=\sin\theta\,\partial_{x_1}q+q\cos\theta\,\partial_{x_1}\theta -\cos\theta\,\partial_{x_2} q +q\sin\theta\,\partial_{x_2} \theta.
	\end{equation*}
	
\smallskip	
2. We now compute the relationship between $\nabla q$ and $\nabla p$. Recall the definition of $(B, S)$:
\begin{equation*}
S:=\mathbb{S}(\psi)=\frac{\gamma p}{(\gamma-1)\rho^\gamma},
\qquad\,\,
B:=\mathbb{B}(\psi)=\frac{q^2}{2}+\frac{\gamma p}{(\gamma-1)\rho}.
\end{equation*}
Taking $\partial_{x_i}$ on the first equation, we have
\begin{equation*}
\mathbb{S}'\partial_{x_i}\psi
=\frac{\gamma}{(\gamma-1)\rho^\gamma}\partial_{x_i}p-\frac{\gamma^2 p}{(\gamma-1)\rho^{\gamma+1}}\partial_{x_i}\rho,
\end{equation*}
which is equivalent to
\begin{equation*}
-\frac{\gamma p}{(\gamma-1)\rho^2}\partial_{x_i}\rho=\frac{1}{\gamma}\rho^{\gamma-1}\mathbb{S}'\partial_{x_i}\psi-\frac{1 }{(\gamma-1)\rho}\partial_{x_i}p.
\end{equation*}
Taking $\partial_{x_i}$ on the second equation, we have
\begin{eqnarray*}
\mathbb{B}'\partial_{x_i}\psi&=&q\partial_{x_i}q+\frac{\gamma }{(\gamma-1)\rho}\partial_{x_i}p-\frac{\gamma p}{(\gamma-1)\rho^2}\partial_{x_i}\rho\nonumber\\
&=&q\partial_{x_i}q+\frac{1 }{\rho}\partial_{x_i}p+\frac{1}{\gamma}\rho^{\gamma-1}\mathbb{S}'\partial_{x_i}\psi,
\end{eqnarray*}
which is equivalent to
\begin{eqnarray*}
\partial_{x_i} q
=-\frac{1}{\rho q}\partial_{x_i} p-\frac{\partial_{x_i}\psi}{q\rho}\omega,
\end{eqnarray*}
where we have used the identity that $\omega=-\rho\mathbb{B}'+\frac{1}{\gamma}\rho^{\gamma}\mathbb{S}'$.
Now we plug the above identity into the two equations in \eqref{A.5}: For the first equation,
we note that $\cos\theta\,\partial_{x_1}\psi+\sin\theta\,\partial_{x_2}\psi=0$ to obtain
\begin{equation}\label{5.33a}
-\rho q \sin\theta\,\partial_{x_1} \theta+\rho q \cos\theta\,\partial_{x_2} \theta-\frac{(1-M^2)}{q}\cos\theta\,\partial_{x_1} p
-\frac{(1-M^2)}{ q}\sin\theta\,\partial_{x_2} p=0;
\end{equation}
For the second equation, noting that  $-\sin\theta\,\partial_{x_1} \psi+\cos\theta\,\partial_{x_2} \psi=\rho q$, we have
\begin{equation}\label{5.34a}
q\cos\theta\,\partial_{x_1}\theta +q\sin\theta\,\partial_{x_2} \theta-\frac{\sin\theta}{\rho q}\partial_{x_1}p
+\frac{\cos\theta}{\rho q}\partial_{x_2}p=0.
\end{equation}

Combining \eqref{5.33a}--\eqref{5.34a} together yields
\begin{eqnarray}\label{5.35}
\begin{cases}
\sin\theta\,\partial_{x_1} \theta-\cos\theta\,\partial_{x_2} \theta
  =-\frac{1-M^2}{ \rho q^2}\cos\theta\,\partial_{x_1} p-\frac{1-M^2}{\rho q^2}\sin\theta\,\partial_{x_2} p,\\[2mm]
\cos\theta\,\partial_{x_1}\theta +\sin\theta\,\partial_{x_2} \theta=\frac{\sin\theta}{\rho q^2}\partial_{x_1}p-\frac{\cos\theta}{\rho q^2}\partial_{x_2}p.
\end{cases}
\end{eqnarray}

\smallskip
3. We can now solve \eqref{5.35} for $\nabla\theta$ to obtain
\begin{eqnarray*}
\begin{cases}
\partial_{x_1}\theta= \frac{M^2}{2\rho q^2}\sin(2\theta)\partial_{x_1} p - \frac{1}{\rho q^2}\big(1-M^2\sin^2\theta\big)\partial_{x_2} p,\\[2mm]
\partial_{x_2}\theta=\frac{1}{\rho q^2}\big(1-M^2\cos^2\theta\big)\partial_{x_1}p-\frac{M^2}{2\rho q^2}\sin(2\theta)\partial_{x_2}p.
\end{cases}
\end{eqnarray*}
Then, by applying the identity that $\partial_{x_1x_2}\theta=\partial_{x_1x_2}\theta$, we have the following equation of second-order for $p$:
\begin{eqnarray}
&&\partial_{x_1}\left(\frac{1}{\rho q^2}\big(1-M^2\cos^2\theta\big)\partial_{x_1}p-\frac{M^2}{2\rho q^2}\sin(2\theta)\partial_{x_2}p\right)\nonumber\\[1mm]
&&=\partial_{x_2}\left(\frac{M^2}{2\rho q^2}\sin(2\theta)\partial_{x_1} p - \frac{1}{\rho q^2}\big(1-M^2\sin^2\theta\big)\partial_{x_2} p\right).
\label{5.35a}
\end{eqnarray}
Equation \eqref{5.35a} is strictly elliptic, provided that the flow is subsonic without stagnation points.
Therefore, the maximum principle implies that $p$ can not attain its maximum and minimum at any interior point, which implies
\eqref{est:theta5.20a}.

\medskip
4. We can also solve \eqref{5.35} for $\nabla p$ to obtain
\begin{eqnarray*}
\begin{cases}
\partial_{x_1} p=-\frac{M^2}{2(1-M^2)}\rho q^2 \sin(2\theta)\partial_{x_1}\theta
  +\frac{\rho q^2}{1-M^2}\big(1-M^2\sin^2\theta\big)\partial_{x_2} \theta\\[2mm]
\partial_{x_2} p= -\frac{\rho q^2}{1-M^2}\big(1-M^2\cos^2\theta\big)\partial_{x_1} \theta
 +\frac{M^2}{2(1-M^2)}\rho q^2\sin(2\theta)\partial_{x_2} \theta.
\end{cases}
\end{eqnarray*}

Then, by applying the identity that $\partial_{x_1x_2}p=\partial_{x_1x_2}p$, we have the following equation of second-order for $\theta$:
\begin{eqnarray}
&&\partial_{x_1}\left( -\frac{\rho q^2}{1-M^2}\big(1-M^2\cos^2\theta\big)\partial_{x_1} \theta
 +\frac{M^2}{2(1-M^2)}\rho q^2\sin(2\theta)\partial_{x_2} \theta\right)\nonumber\\[1mm]
&&=\partial_{x_2}\left(-\frac{M^2}{2(1-M^2)}\rho q^2 \sin(2\theta)\partial_{x_1}\theta
  +\frac{\rho q^2}{1-M^2}\big(1-M^2\sin^2\theta\big)\partial_{x_2} \theta\right).\label{5.35b}
\end{eqnarray}
Equation \eqref{5.35b} is also strictly elliptic, provided that the flow is subsonic without stagnation points.
Then, again, the maximum principle implies that $\theta$ can not attain its maximum and minimum at any interior point.
This completes the proof of \eqref{est:theta5.20}.
\end{proof}

\subsection{Uniform estimates of $|\nabla\psi|$}
Based on the maximum principle for the pressure in Lemma \ref{lem:5.4qmax},
in this section, we  show the subsonicity of the Euler flows such that  the Bers skill in \cite{Bers2} can be applied to release
the assumption of the largeness of flux $m$ to the critical value $\tilde{m}_{\rm c}$ which depends only on $(w_1,w_2)$ (the boundary of the nozzle),
$(u_-, S_-)$ at the inlet,
and $(u_-', S_-')$ away from the discontinuity.

\begin{lemma}\label{lem:6.5q}
There exists $\tilde{m}_{\rm c}>0$, depending only on $(w_1,w_2)$, $(u_-,S_-)$ at the inlet, and $(u_-', S_-')$ near the
walls {\rm (}i.e. $x_2=0$ and $x_2=1${\rm )},
such that, for any given $m>\tilde{m}_{\rm c}$, the solutions satisfy \eqref{3.11}.
\end{lemma}

\begin{proof}
For any given $m$, and $(u_-, S_-)$ at the inlet,
we can define $(\mathbb{B}, \mathbb{S})$ that are invariant along each streamline.
For any fixed $(B, S)$, \emph{i.e.} along each streamline,  by the Bernoulli law:
$
B=\frac{q^2}{2}+S \rho^{\gamma-1},
$
we have
\begin{equation*}
\frac{d \rho}{d q}=-\frac{q}{(\gamma-1) S \rho^{\gamma-2}}=-\frac{\rho q}{c^2}<0.
\end{equation*}
Then, for the subsonic flow, we have
\begin{equation*}
\frac{d (\rho q)}{d q}=\rho (1-M^2)>0.
\end{equation*}
Next, note that $\frac{d p}{d\rho}=c^2$, so that
\begin{equation}\label{5.41}
\frac{d p}{d  q}=-\rho q<0, \qquad \,\, \frac{d p}{d (\rho q)}=-\frac{q}{1-M^2}<0.
\end{equation}
Thus, for the subsonic flow, we can use the lower bound of $p$ along each streamline
to control the upper bounds of $|\nabla\psi|$ and $q$ to control the subsonicity.
The critical pressure along each streamline is obtained by adding the constraint
that $c^2=q^2$, that is,
\begin{equation*}
B=\frac{(\gamma+1)c^2}{2(\gamma-1)}=\frac{\gamma+1}{2}\Big(\frac{\gamma}{\gamma-1}\Big)^{\frac{\gamma-1}{\gamma}} p^{\frac{\gamma-1}{\gamma}} S^{\frac{1}{\gamma}}.
\end{equation*}
Therefore, the critical pressure along each streamline (\emph{i.e.} for fixed $(B, S)$) is
\begin{eqnarray*}
\breve{p}
=\frac{\gamma-1}{\gamma}\Big(\frac{2}{\gamma+1}\Big)^{\frac{\gamma}{\gamma-1}}(S^{-\frac{1}{\gamma}}B)^{\frac{\gamma}{\gamma-1}}.
\end{eqnarray*}
Assume that the corresponding values of $(B,S)$ at the inlet which share the same streamlines are $(B_-, S_-)$. Then
\begin{eqnarray*}
\breve{p}
&=&\frac{\gamma-1}{\gamma}\Big(\frac{2}{\gamma+1}\Big)^{\frac{\gamma}{\gamma-1}}(S_-^{-\frac{1}{\gamma}}B_-)^{\frac{\gamma}{\gamma-1}}\nonumber\\
&=&\frac{\gamma-1}{\gamma}\Big(\frac{2}{\gamma+1}\Big)^{\frac{\gamma}{\gamma-1}}\Big(\frac{1}{2} u_{1-}^2 S_-^{-\frac{1}{\gamma}}
+\big(\frac{\gamma}{\gamma-1}\big)^{\frac{\gamma-1}{\gamma}}p_-^{\frac{\gamma-1}{\gamma}}\Big)^{\frac{\gamma}{\gamma-1}}.\nonumber
\end{eqnarray*}
This means that, at each streamline,  $M<1$ if $p>\breve{p}$, and $M=1$ if $p=\breve{p}$.

Define
\begin{equation}\label{5.44}
\breve{p}_{0}:=\frac{\gamma}{2(\gamma-1)}\Big(\frac{\gamma+1}{2}\Big)^{\frac{\gamma}{\gamma-1}}
\Big(\max_{s\in[0, 1]}( u_{1-}^2S_-^{-\frac{1}{\gamma}})(s)+2\big(\frac{\gamma}{\gamma-1}\big)^{\frac{\gamma-1}{\gamma}}p_-^{\frac{\gamma-1}{\gamma}}\Big)^{\frac{\gamma}{\gamma-1}}.
\end{equation}
Then, if $p>\breve{p}_{0}$, the flow is subsonic in $\Omega$, \emph{i.e.} $M<1$.

Denote the minimal value of pressure $p$ on the boundary to be $p_B$.
By Lemma \ref{lem:5.4qmax}, $p\geq p_B$.
Therefore, it suffices to show $p_B\geq\breve{p}_{0}$.
By \eqref{5.41}, it is equivalent to show
$$
|\nabla\psi|\leq (\rho q) (\breve{p}_0) \qquad \mbox{on the boundary},
$$
where $(\rho q) (\breve{p}_0)$ means that $\rho q$ is determined by $\breve{p}_0$ on the boundary that is also a streamline.
By the direct calculation, we have
\begin{equation*}
\rho(\breve{p}_0)=\Big(\frac{\gamma \breve{p}_0}{(\gamma-1)S_-}\Big)^{\frac{1}{\gamma}}
\end{equation*}
and
\begin{eqnarray*}
q(\breve{p}_0)
&=&\Big(u_{1-}^2+2(\frac{\gamma}{\gamma-1})^{\frac{\gamma-1}{\gamma}}p_-^{\frac{\gamma-1}{\gamma}}S_-^{\frac{1}{\gamma}} -2(\frac{\gamma}{\gamma-1})^{\frac{\gamma-1}{\gamma}}\breve{p}_0^{\frac{\gamma-1}{\gamma}}S_-^{\frac{1}{\gamma}}\Big)^{\frac{1}{2}}.
\end{eqnarray*}
Then, as before,
we obtain
$$
(\rho q)(\breve{p}_0)=O(1)\big(1+O(1)p_-^{\frac{1}{\gamma}}\big)\big(1+O(1)p_-^{\frac{\gamma-1}{2\gamma}}\big)
=O(1)m^{\frac{\gamma+1}{2}}
$$
when $m$ is a sufficiently large constant depending only on $(u_-, S_-)$ at the inlet.

On the other hand, for any $m_0\in(0,m)$ with property \eqref{3.11} when $m_0<\psi<m$,
$$
{\rm{dist}}(\{\bx\, :\,\psi(\bx)=m_0\}, W_2)>r_0>0,
$$
where constant $r_0$ depends only on $\|(\mathbb{B},\mathbb{S})\|_{L^\infty([m_0,m])}$ and $m-m_0$.
Therefore, \eqref{3.23} yields that $|\nabla\psi|$ at the boundary is estimated as
\begin{equation*}
[\psi]_{1,\mu;\{m_0<\psi<m\}}
\leq C(r_0)\Big(1+m+\frac{\|\tilde{f}\|_{0,\{m_0<\psi<m\}}}{\lambda}\Big).
\end{equation*}
where $1+m+\frac{\|\tilde{f}\|_{0,\{m_0<\psi<m\}}}{\lambda}=O(1)m^{\frac{\gamma-1}{2}}$
when $m$ is sufficiently large, depending only on $(u_-, S_-)$ at the inlet and $(u_-', S_-')$ near the boundaries.

Then we can see that, if $m$ is sufficient large, depending only on $(u_-, S_-)$ at the inlet and $(u_-',S_-')$ near the boundaries,
then the flow is subsonic.
Therefore, we can release the assumption of $m$ from the very large number to be $m_{\rm c}$ for the ellipticity
on the walls to have \eqref{3.11}, which also implies that $q$ is uniformly bounded in the nozzle.

This completes the proof.
\end{proof}

\subsection{The Bers method}\label{Bers method}

Till now, we have obtained the existence and uniqueness of {\bf Problem 3.1($m$)} when $m$ is sufficiently large,
where we have used $m$ as a parameter to emphasize that our results depend on $m$.
Parameter $m$ belongs to an {\it a priori} unknown interval.

We now introduce a function of $\psi$ as $\mathcal{M}$ such that $\mathcal{M}(\psi)<0$ is equivalent to that the flow  is subsonic,
while $\mathcal{M}(\psi)=0$ is equivalent to that the flow is sonic.
More precisely, we define
\[
\mathcal{M}:=\frac{|\nabla_{\bx}\psi|}{\hat{Q}(\psi)}-1.
\]
Now we define the following problem:

\bigskip
{\bf Problem 5.1($m$, $\v$)}: {\it Seek a function $\psi(\bx; m, \v)$ such that
it is a solution of {\bf Problem 3.1($m$)} with
$\mathcal{M}(\psi(\bx; m, \v))<-2\v$,
where $\v$ is used for the uniform ellipticity.}
\bigskip

First, we introduce the Bers conditions ({\it cf.} \cite{Bers2}):
 \begin{definition}  The Bers conditions for {\bf Problem 5.1 ($m,\v$)} with $\mathcal{M}(\psi)$ are the following{\rm :}
 	\begin{itemize}
 		\item[(a)] $\mathcal{M}(\psi(\bx; m, \v))$ is right-continuous respect to $m$, for any fixed $\v>0${\rm ;}
 		\item[(b)] For any fixed $m$, the solution of {\bf Problem 5.1($m,\v$)} is unique.
 	\end{itemize}
 \end{definition}

The Bers method can be stated via the following proposition:

\begin{proposition}\label{prop:6.9bers}
If {\bf Problem 5.1($m,\v$)} with $\mathcal{M}(\psi)$ satisfies the Bers conditions {\rm (a)}--{\rm (b)},
then there exists a critical parameter $m_{\rm c}$ such that,
if $m\in(m_{\rm c},\infty)$, then there exists $\v_{j}$ so that $\psi(\bx; m, \v_{j})$
as the  unique solution of {\bf Problem 5.1($m$, $\v_{j}$)} is the unique solution of {\bf Problem 3.1($m$)}.
Furthermore, either $\mathcal{M}(\psi(\bx; m, \v_{j}))\rightarrow 0$ as $m\rightarrow m_{\rm c}$,
or there does not exist $\sigma>0$ such that {\bf Problem 3.1($m$)} has a solution
for all $m\in(m_{\rm c}-\sigma, m_{\rm c})$ such that $\mathcal{M}(\psi(\bx; m))<0$.
 \end{proposition}

 \begin{proof}
Let $\{\v_j\}_{j=1}^{\infty}$ be a strictly decreasing sequence of positive constants
such that $\v_j\rightarrow0$ as $j\rightarrow\infty$.
From the Bers condition (a), for fixed $j$, there exists a maximal interval $(\bar{m}_j, \infty)$ such that,
for $m\in(\bar{m}_j,\infty)$,
$$
\mathcal{M}(\psi(\bx; m, \v_j))<-2\v_j.
$$
Then, for $m\in(\bar{m}_j, \infty)$,
$\psi(\bx; m, \v_j)$ is the solution of {\bf Problem 5.1($m$, $\v_j$)}.
By the Bers condition (b),
we see that $\bar{m}_j\ge\bar{m}_k\ge 0$ for $j>k$.
Thus, $\{\bar{m}_j\}_{n=1}^{\infty}$ is a decreasing nonnegative sequence,
which implies the convergence of the sequence.
As a consequence, we have
$$
m_{\rm c}:=\lim_{j\rightarrow\infty}\bar{m}_j.
$$

If $\bar{m}_j>m_{\rm c}$ for any $j$, then, for any $m\in (m_{\rm c},\infty)$, there exists an index $j$
such that $\bar{m}_{j}<m\le\bar{m}_{j-1}$.
The solution of {\bf Problem 5.1($m$, $\v_j$)} is $\psi(\bx; m, \v_j)$ with the estimate:
$$
-2\v_{j-1}\le\mathcal{M}(\psi(\bx; m,\v_j))<-2\v_j.
$$
Then we conclude that
$\mathcal{M}(\psi(\bx; m, \v_{j}))\rightarrow 0$ as $m\rightarrow m_{\rm c}$.
 	
On the other hand, if $\bar{m}_j\equiv m_{\rm c}$ when $j$ is large enough,
then, by construction, there does not exist $\sigma>0$ such that
{\bf Problem 3.1($m$)} has a solution for all $m\in(m_{\rm c}-\delta,m_{\rm c})$
so that $\mathcal{M}(\psi(\bx; m))<0$.
This implies that some new phenomena like shock bubbles
may appear before the subsonic flows become subsonic-sonic flows.
\end{proof} 	

\begin{remark}
By Theorem {\rm \ref{prop:5.1}} and Lemma {\rm \ref{prop:6.1}},
the solution of {\bf Problem 5.1($m$, $\v$)} satisfies the Bers conditions.
Then, by Proposition {\rm \ref{prop:6.9bers}}, when $m>m_{\rm c}$,
{\bf Problem 3.1($m$)} has a unique subsonic solution.
 \end{remark}

 \begin{proof}[Proof of Theorem {\rm \ref{thm:smooth}}]
 The existence, uniqueness, and the properties listed in Theorem \ref{thm:smooth}
 have been proved in \S \ref{sec:existence smooth} and this section.
 Therefore, we complete the proof of Theorem \ref{thm:smooth}.
 \end{proof}

\section{Discontinuous Solutions}\label{sec:mathsettingweak}

Now we study the existence and uniqueness of weak solutions
of {\bf Problem 3.2($m$)} if $(u_{1-}, S_-)$ have discontinuities.
First, the weak solutions of  {\bf Problem 3.2($m$)}
are actually solutions $\psi$ of equations \eqref{equ:8.1discontinuous} with the boundary conditions \eqref{bou:2.9}.
We remark that, by the limiting process, the solutions of \eqref{equ:8.1discontinuous}
are the weak solutions of the full Euler equations \eqref{1.5}.

We now consider the existence of solutions with discontinuities.

\begin{lemma}\label{prop:9.1existence}
Let $(u_{1-}, S_-)$ satisfy \eqref{assumption:3.40}--\eqref{assumption:3.47a}, and let there be $\varepsilon_0>0$ such that
\begin{equation}\label{assumption2.2.xw}
(u_{1-},S_-)\in \big(BV([0,1])\cap C^{1,1}([0,2\varepsilon_0)\cup(1-2\varepsilon_0,1])\big)^2.
\end{equation}
Then there exists $\underline{m}$ such that, for any $m>\underline{m}$,
there is a weak solution $\bar{\psi}$ of  equation \eqref{equ:8.1discontinuous} so that
$\mathcal{M}(\bar{\psi})<1$, $\partial_{x_2}\bar{\psi}\geq0$, $|\bar{\theta}|\leq\theta_B$,
and $\min_{\partial \Omega}\bar{p}\le \bar{p}\le \max_{\partial \Omega}\bar{p}$\, a.e. in $\Omega$.
\end{lemma}

\begin{proof}  We divide the proof into two steps.

\smallskip
1. First, we introduce the process for the discontinuous stream-conserved quantities.

Let $j_\varepsilon\in C^\infty_0(\mathbb{R})$ be the standard modifier satisfying $\mbox{supp}(j_\varepsilon) =(-\varepsilon, \varepsilon)$ with
$\|j_\varepsilon\|_{L^1}=1$.
For any $\varepsilon<\varepsilon_0$, we introduce
\begin{equation*}
\mathcal{M}_\varepsilon (g)= (1-I_{[\varepsilon_0, 1-\varepsilon_0]}*j_\varepsilon)g+((I_{[\varepsilon_0, 1-\varepsilon_0]}*j_\varepsilon)g)*j_\varepsilon.
\end{equation*}
Define
\begin{eqnarray*}
\begin{cases}
S_-^\varepsilon:=\mathcal{M}_\varepsilon (S_-),\\
u_{1-}^\varepsilon:=\mathcal{M}^{\frac{1}{2}}_\varepsilon (u_{1-}^2S_-^{-\frac{1}{\gamma}})\mathcal{M}^{\frac{1}{2\gamma}}_\varepsilon (S_-)
\end{cases}
\end{eqnarray*}
such that $(u_{1-}^\varepsilon, S_-^\varepsilon)$  converge to $(u_{1-}, S_-)$ pointwise in $[0, 1]$.
It is clear that $(u_{1-}^\varepsilon, S_-^\varepsilon)$ satisfy the assumptions of Theorem \ref{thm:smooth}.
Then, for any $\varepsilon\in(0,\varepsilon_0)$,
there is $\underline{m}^{(\varepsilon)}$ such that, when $m>\underline{m}^{(\varepsilon)}$,
there exists a unique $\psi^\varepsilon$ satisfying
\begin{eqnarray*}
\begin{cases}
\nabla\psi^\varepsilon=\rho^\varepsilon(- u_2^\varepsilon, u_1^\varepsilon),\\[2mm]
\mathbb{B^\varepsilon}(\psi^\varepsilon)=\frac{1}{2}|\bu^\varepsilon|^2 +\frac{\gamma p^\varepsilon}{(\gamma-1)\rho^\varepsilon},\\[2mm]
\mathbb{S^\varepsilon}(\psi^\varepsilon)= \frac{\gamma p^\varepsilon}{(\gamma-1)(\rho^\varepsilon)^\gamma},\\[2mm]
\partial_{x_1}u_2^\varepsilon-\partial_{x_2}u_1^\varepsilon=-\rho^\varepsilon (\mathbb{B}^\varepsilon)'(\psi^\varepsilon)+\frac{(\rho^\varepsilon)^{\gamma}}{\gamma}(\mathbb{S}^\varepsilon)'(\psi^\varepsilon),
\end{cases}
\end{eqnarray*}
with
\begin{equation*}
\psi^\varepsilon|_{x_2=w_1(x_1)}=0,\qquad \psi^\varepsilon|_{x_2=w_2(x_1)}=m,
\end{equation*}
and $\mathcal{M}(\psi^\varepsilon)<1$, $\partial_{x_2}\psi^\varepsilon>0$, $|\theta^\varepsilon|\leq \theta_B$, and
$\min_{\partial \Omega} p^{\varepsilon}\le p^{\varepsilon}\le \max_{\partial \Omega} p^{\varepsilon}$  in $\Omega$.

Moreover, by Proposition \ref{prop:6.9bers}, $\underline{m}^{(\varepsilon)}$ is independent of the regularization parameter  $\varepsilon$.
Then $(\rho^{\varepsilon},\bu^{\varepsilon},p^{\varepsilon})$ obtained from $\psi^{\varepsilon}$
satisfies the full Euler equations.

\smallskip
2. We then show the convergence of $(\rho^\varepsilon, \bu^\varepsilon, p^\varepsilon)(\bx)$ {\it a.e.} from $\psi^\varepsilon$,
by employing the compensated compactness framework in Chen-Huang-Wang \cite{ChenHuangWang}.
Here $(B^\varepsilon, S^\varepsilon, \omega^\varepsilon)$ are defined by
\begin{eqnarray*}
\begin{cases}
B^\varepsilon(\bx):=\mathbb{B^\varepsilon}(\psi^\varepsilon(\bx))=\frac{1}{2}|\bu^\varepsilon(\bx)|^2+\frac{\gamma p^\varepsilon}{(\gamma-1)\rho^\varepsilon}(\bx),\\[2mm]
S^\varepsilon(\bx):=\mathbb{S^\varepsilon}(\psi^\varepsilon)(\bx)= \frac{\gamma p^\varepsilon}{(\gamma-1)(\rho^\varepsilon)^\gamma}(\bx),\\[2mm]
\omega^\varepsilon(\bx):=-\rho^\varepsilon (\mathbb{B}^\varepsilon)'(\psi^\varepsilon)(\bx)+\frac{(\rho^\varepsilon)^{\gamma}}{\gamma}(\mathbb{S}^\varepsilon)'(\psi^\varepsilon)(\bx).
\end{cases}
\end{eqnarray*}
From a direct calculation, we have
\begin{eqnarray*}
\nabla B^\varepsilon(\bx)=(\mathbb{B}^\varepsilon)'(\psi^\varepsilon)(\bx)\rho^\varepsilon(\bx) (-u_2^\varepsilon(\bx),u_1^\varepsilon(\bx)).
\end{eqnarray*}
Then $B^\varepsilon$ is uniformly bounded in $BV$, which implies its strong convergence.
The similar argument can lead to the strong convergence of $S^\varepsilon$.
The vorticity sequence $\omega^\varepsilon$ as a measure sequence is uniformly bounded,
which is compact in $H^{-1}_{loc}$.

Combining these with the subsonic condition: $M(\psi^\varepsilon)<1-\delta$,
Theorem 2.1 in Chen-Huang-Wang \cite{ChenHuangWang} implies that
the solution sequence has a subsequence (still denoted
by) $(\rho^\varepsilon, \bu^{\varepsilon}, p^\varepsilon)(\bx)$
that converges {\it a.e.} in $\Omega$ to
a vector function $(\bar{\rho}, \bar{\bu}, \bar{p})(\bx)$.

For the vector function $(\bar{\rho}, \bar{\bu}, \bar{p})(\bx)$ obtained as the limit,
we introduce $\bar{\psi}$ with
\begin{eqnarray*}
\nabla\bar{\psi}=\bar{\rho}(-\bar{u}_2,\bar{u}_1)
\end{eqnarray*}
such that $\mathcal{M}(\bar{\psi})<1$, $\partial_{x_2}\bar{\psi}\geq0$, $|\bar{\theta}|\leq\theta_B$,
and $\min_{\partial \Omega}\bar{p}\le \bar{p}\le \max_{\partial \Omega}\bar{p} \,$  {\it a.e.} in $\Omega$,
and $\bar{\psi}|_{x_2=w_1(x_1)}=0$ and $\bar{\psi}|_{x_2=w_2(x_1)}=m$.
Also, it can be checked that $\psi^\varepsilon$ converges to $\bar{\psi}$ in the Lipschitz space $Lip(\Omega)$.
\end{proof}

Now we show the $C^{2, \alpha}$--convergence of $\psi^\varepsilon$ away from the discontinuity under assumption \eqref{assumption:3.44dis}.

\begin{lemma}\label{lem:6.2x}
If, in addition, assumption \eqref{assumption:3.44dis} holds,
then $\psi$ is $C^{2,\alpha}$ and $\partial_{x_2}\bar{\psi}>0$ when $\psi\neq m_d$.
Moreover, the solution satisfies \eqref{equ:2.59a}--\eqref{equ:2.60a} uniformly
for $x_2\in K\Subset(0,1)$ as $x_1\rightarrow-\infty$ and for $x_2\in K\Subset(a,b)$ as $x_1\rightarrow\infty$ when $\psi\neq m_d$.
\end{lemma}

\begin{proof}
For the existence, we only need to repeat the proof of Lemma \ref{prop:9.1existence}
by changing only the modifiers as follows:
Without loss of generality, assume that there is only one jump point at $x_2=x_d$ at the inlet $x_1=-\infty$,
which implies that
there exist $x_{d-}$ and $x_{d+}$ such that $0< x_{d-}<x_d<x_{d+} < 1$.
Then $(u_{1-}^\varepsilon, S_-^\varepsilon)$ determined in the proof of Lemma \ref{prop:9.1existence}
converge to $(u_{1-}, S_-)$ pointwise in $[0, 1]$, and $C^{1,1}$ uniformly in any subset away from $x_2=x_d$.
Thus, we obtain the existence of weak solutions as stated in Lemma \ref{prop:9.1existence}.

\smallskip
Next,
for fixed $\psi^\varepsilon$, we define
\begin{equation*}
\Omega^\varepsilon(\dot{m}):=
\begin{cases}
\{\bx\; :\; \bx\in\Omega, \psi^\varepsilon(\bx)<\dot{m} \}  \qquad \mbox{for $0<\dot{m}<m_d$},\\[1mm]
\{\bx\; :\; \bx\in\Omega, \psi^\varepsilon(\bx)>\dot{m} \} \qquad \mbox{for $m_d<\dot{m}<m$}.
\end{cases}
\end{equation*}
Then $\Omega^\varepsilon(\dot{m})$ satisfies the following properties:
\begin{enumerate}
\item[(i)] For $0<\dot{m}_1<\dot{m}_2<m_d$, $\Omega^\varepsilon(\dot{m}_1)\subsetneq\Omega^\varepsilon(\dot{m}_2)$;

\smallskip
\item[(ii)] For $m_d<\dot{m}_1<\dot{m}_2<m$, $\Omega^\varepsilon(\dot{m}_1)\supsetneq\Omega^\varepsilon(\dot{m}_2)$;

\smallskip
\item[(iii)]
$\overline{\Omega}=\overline{\big(\cup_{0<\dot{m}<m_d} \Omega^\varepsilon(\dot{m})\big)\cup \big(\cup_{m_d<\dot{m}<m} \Omega^\varepsilon(\dot{m})\big)}$ for each $\varepsilon$.
\end{enumerate}

Similarly, we define the corresponding sets with respect to $\bar{\psi}$:
\begin{equation*}
\Omega(\dot{m}):=
\begin{cases}
\{\bx\; :\; \bx\in\Omega, \bar{\psi}(\bx)<\dot{m} \}  \qquad \mbox{for $0<\dot{m}<m_d$},\\[1mm]
\{\bx\; :\; \bx\in\Omega, \bar{\psi}(\bx)>\dot{m} \} \qquad \mbox{for $m_d<\dot{m}<m$}.
\end{cases}
\end{equation*}
and the level set:
$$
\Omega(m_d):=\{\bx\,:\, \bx\in\Omega, \bar{\psi}(\bx)=m_d \}
$$
with the following properties:
\begin{enumerate}
\item[(i)] For $0<\dot{m}_1<\dot{m}_2<m_d$, $\Omega(\dot{m}_1)\subset\Omega(\dot{m}_2)$;

\smallskip
\item[(ii)] For $m_d<\dot{m}_1<\dot{m}_2<m$, $\Omega(\dot{m}_1)\supset\Omega(\dot{m}_2)$;

\smallskip
\item[(iii)] $\overline{\Omega}
=\overline{\big(\cup_{0<\dot{m}<m_d} \Omega(\dot{m})\big)\cup \big(\cup_{m_d<\dot{m}<m} \Omega(\dot{m}) \cup \Omega(m_d)\big)}$.
\end{enumerate}

Since the Lipschitz convergence of $\psi^\varepsilon$ to $\bar{\psi}$ as $\varepsilon\to 0$,
we see that, for each $\dot{m}\neq m_d$, $\Omega^\varepsilon(\dot{m})$ converges to $\Omega(\dot{m})$.
Thus, for $0<\dot{m}<m_d$,
$(\mathbb{B}^\varepsilon, \mathbb{S}^\varepsilon)$
converge to $(\mathbb{B}, \mathbb{S})$ in $C^{2, \alpha'}$ with $\alpha'<\alpha$ in the interval $[0, \dot{m}]$.

On the other hand, for the approximate sequence $\psi^\varepsilon$, $\mathbb{B}^\varepsilon(\psi^\varepsilon)=B^\varepsilon$,
$\mathbb{S}^\varepsilon(\psi^\varepsilon)=S^\varepsilon$,
and $\partial_{x_1}u_2^\varepsilon-\partial_{x_2}u_1^\varepsilon=\omega^\varepsilon$.
Then, by taking the limit of subsequence, we find that, in $\Omega(\dot{m})$,
\begin{eqnarray*}
\begin{cases}
\mathbb{B}(\bar{\psi})=\frac{1}{2}|\bu|^2+\frac{\gamma \bar{p}}{\bar{\rho}},\\[1mm]
\mathbb{S}(\bar{\psi})= \frac{\gamma \bar{p}}{(\gamma-1)\bar{\rho}^\gamma},\\[1mm]
\partial_{x_1}\bar{u}_2-\partial_{x_2}\bar{u}_1=-\bar{\rho} \mathbb{B}'(\bar{\psi})+\frac{\bar{\rho}^{\gamma}}{\gamma}\mathbb{S}'(\bar{\psi}).
\end{cases}
\end{eqnarray*}

Employing the same argument for the smooth flow in \S 4--\S 5, we can prove that,
for any $K\Subset\cup_{0<\dot{m}<m_d} \Omega(\dot{m})$,
$\bar{\psi}$ is a $C^{2, \alpha}$--smooth function
with $\partial_{x_2}\bar{\psi}>0$.
The argument for $\cup_{m_d<\dot{m}<m} \Omega(\dot{m})$ is similar.
Now we can make the standard diagonal argument and the blowup argument done similarly in Step 5 of the proof of Theorem \ref{prop:5.1}
in \S \ref{sec:existence smooth} to show the asymptotic behavior \eqref{equ:2.59a}--\eqref{equ:2.60a} when $\psi\neq m_d$.
\end{proof}

Then we prove the interior of $\Omega(m_d)$ is empty.
In fact, by the monotonicity of $\psi$ with respect to the $x_2$-coordinate, we have

\begin{lemma}\label{lem:6.2xw}
If, in addition, assumption \eqref{assumption:4.2} or \eqref{assumption:4.3} holds,
then the discontinuity defined by the level set $\{\psi=m_d\}$ is a Lipschitz curve.
\end{lemma}

\begin{proof} We divide the proof into four steps.

\medskip
1. First, by Lemma \ref{lem:6.2x}, for any $m\neq m_d$,
the level set $\{\psi=m\}$ is a $C^{2,\alpha}$--curve.
Using the fact that $|\bar{\theta}|\leq\theta_B$ and $\mathcal{M}(\bar{\psi})<1$ in Lemma \ref{prop:9.1existence},
we know that, if $m\neq m_d$, the level set $\{\psi=m\}$ as a curve is uniformly bounded and equicontinuous,
so that there is a subsequence such that $\Omega(m)$ strongly converges to a continuous curve $\Gamma_+$ as $m\rightarrow m_d+$.
On the other hand, $\Omega(m)$ converges in the weak* sense to a Lipschitz curve $\tilde{\Gamma}_+$.
Note that, by $\partial_{x_2}\bar{\psi}>0$ in Lemma \ref{lem:6.2x},
$\Omega(m)$ is monotonically increasing with respect to $m$, and hence the limit is unique.
This implies that $\Gamma_+=\tilde{\Gamma}_+$, which is the boundary of $\cup_{m_d<\dot{m}<m} \Omega(\dot{m})$.
Similar argument yields that the upper boundary $\Gamma_-$ of $\cup_{m_d<\dot{m}<m} \Omega(\dot{m})$
is also a Lipschitz curve.

Let $\Omega(m_d)=\Omega\backslash\big((\cup_{0<\dot{m}<m_d} \Omega(\dot{m}))\cup (\cup_{m_d<\dot{m}<m} \Omega(\dot{m}))\big)$.
Then, by $\partial_{x_2}\bar{\psi}>0$ in Lemma \ref{lem:6.2x},
$\Omega(m_d)$ is connected, and its boundary $\Gamma_+\cup\Gamma_-$ is Lipschitz.

Therefore, it suffices to show that the interior of $\Omega(m_d)$ is an empty set.
This can be achieved by the contradiction argument.

\medskip
2. If the interior $Int(\Omega(m_d))$ of $\Omega(m_d)$ is non-empty,
then $\bar{u}_1=\bar{u}_2=0$ in $Int(\Omega(m_d))$, since $\bar{\psi}=m_d$.
By
Lemma \ref{lem:3.1xw}, without loss of generality, we assume
\begin{equation}\label{condition:8.2}
\big(\mathbb{B}\mathbb{S}^{-\frac{1}{\gamma}}\big)'(m_d-)>0,\qquad\,\,
\big(\mathbb{B}\mathbb{S}^{-\frac{2}{\gamma(\gamma+1)}}\big)'(m_d-)>0.
\end{equation}

The lower boundary of $\Omega(m_d)$ is $\Gamma_-$.
Then, for any test function $\eta$ compactly supported in a connected open set $\Sigma$ with $\Sigma\cap \Gamma_-\neq \emptyset$,
and for the bounded measure $\bar{\omega}$, we have
\begin{eqnarray}
\langle \eta, \bar{\omega}\rangle
=-\int_{\Sigma}\big(\bar{u}_2\partial_{x_1}\eta-\bar{u}_1\partial_{x_2}\eta\big)\,\dd\bf{x},
\end{eqnarray}
where $\langle\eta, \bar{\omega}\rangle$ is the limit of $\int \eta\omega^{\varepsilon} \dd\bf{x}$.

\begin{figure}[htbp]
\small \centering
\includegraphics[width=10cm]{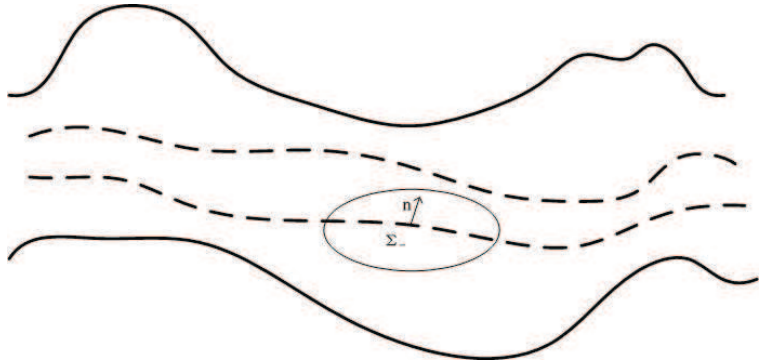}
\caption{Contradiction argument}
\label{Fig2}
\end{figure}

As shown in Figure \ref{Fig2},
let $\Sigma_-= Int(\cup_{0<\dot{m}<m_d} \Omega(\dot{m}))\cap \Sigma $ be the lower part of $\Sigma$,
while the upper part is  $\Sigma_+= Int ( \Omega(m_d))\cap \Sigma$.
Then
$$
\int_{\Sigma_+}\big(\bar{u}_2\partial_{x_1}\eta-\bar{u}_1\partial_{x_2}\eta\big) \dd\bx=0.
$$

3. Next, we  prove the integral:
$$
\int_{\Sigma_-}\big(\bar{u}_2\partial_{x_1}\eta-\bar{u}_1\partial_{x_2}\eta\big)\, \dd\bf{x}
$$
is almost negative.

First, we notice that, in $Int (\cup_{0<\dot{m}<m_d} \Omega(\dot{m}))$,
$\partial_{x_2}\bar{\psi}>0$, which implies that $\bar{\rho} \bar{u}_1>0$.
Divide $\Sigma_-$ into two subregions $\Sigma_-=\Sigma_s\cup\Sigma_b$ with
\begin{eqnarray*}
\Sigma_b:=\{\bx\, :\, \bar{u}_1>\sigma>0\}\cap\Sigma_-,\qquad
\Sigma_s:=\{\bx\,:\, \sigma\geq\bar{u}_1>0\}\cap\Sigma_-,
\end{eqnarray*}
where $\sigma$ is a given constant that is sufficiently small.
The reason of the division is that the integrand can be shown to be negative,
if the horizontal speed $u_1$ has a lower positive bound,
so that the integrand on $\Sigma_s$ is almost negative with an error controlled by the small constant $\sigma$.

Next, let
\begin{equation}\label{La coordinates9.11}
\bz=(z_1, z_2):=(x_1, \bar{\psi}(\bx)).
\end{equation}
In the new coordinates, $\Sigma_-$, $\Sigma_s$, and $\Sigma_b$ are transformed into $\Sigma_-'$, $\Sigma_s'$, and $\Sigma_b'$,
respectively.
Then
\begin{eqnarray}
&&\int_{\Sigma_-}\big(\bar{u}_2\partial_{x_1}\eta-\bar{u}_1\partial_{x_2}\eta\big)\, \dd\bf{x}\nonumber\\
&&=
\int_{\Sigma'_-}\Big(\bar{u}_2(\frac{\partial \eta}{\partial z_1}-\bar{\rho} \bar{u}_2\frac{\partial \eta}{\partial z_2})
   -\bar{u}_1(\bar{\rho} \bar{u}_1\frac{\partial \eta}{\partial z_2})\Big)\frac{1}{\bar{\rho} \bar{u}_1} \dd\mathbf{z}\nonumber\\
&&=\int_{\Sigma_-'}\Big(\bar{u}_2\frac{\partial \eta}{\partial z_1}-\bar{\rho} \bar{q}^2\frac{\partial \eta}{\partial z_2}\Big)
   \frac{1}{\bar{\rho} \bar{u}_1}\dd\mathbf{z}\nonumber\\
&&=\int_{\Sigma_s'}\Big(\bar{u}_2\frac{\partial \eta}{\partial z_1}-\bar{\rho} \bar{q}^2\frac{\partial \eta}{\partial z_2}\Big) \frac{1}{\bar{\rho} \bar{u}_1}\,\dd\mathbf{z}
  +\int_{\Sigma'_b}\Big(\bar{u}_2\frac{\partial \eta}{\partial z_1}-\bar{\rho} \bar{q}^2\frac{\partial \eta}{\partial z_2}\Big) \frac{1}{\bar{\rho} \bar{u}_1}\dd\mathbf{z}.
  \label{Idvertest}
\end{eqnarray}

By choosing a proper $\eta$,
$-\bar{\rho} \bar{q}^2\frac{\partial \eta}{\partial z_2} \frac{1}{\bar{\rho} \bar{u}_1}$ is always negative in $\Sigma_-'$,
while the sign of integral $\int_{\Sigma'_b}\bar{u}_2\frac{\partial \eta}{\partial z_1} \frac{1}{\bar{\rho} \bar{u}_1}\dd\bf{z}$ is not clear.
In order to control it,
we will use the negative term to control the other.

Let function $\varrho$, supported in $[-1,1]$, satisfy that $\varrho(0)=1$, $\varrho(y)=\varrho(-y)$, $\int_{-1}^{1}\varrho(y)\dd y=1$,
and $\varrho'(y)\ge 0$ when $y\le0$.
Denote
\begin{equation}\label{6.14a}
\eta(\bz):=\varrho(\frac{z_1-z_{01}}{b})\varrho(\frac{z_2-m_d}{b w}),
\end{equation}
where $b$ and $w$ are two parameters defined later.
Then $\Sigma'_-$ is $[z_{01}-b, z_{01}+b]\times[m_d-bw, m_d]$.
In this case,
\begin{eqnarray}
\int_{\Sigma'_-}\Big|\frac{\partial \eta}{\partial z_1}\Big|\,\dd\mathbf{z}
=\int_{z_{01}-b}^{z_{01}+b}\Big|\frac{\partial \varrho(\frac{z_1-z_{01}}{b})}{\partial z_1}\Big|\, \dd z_1
\int_{m_d-bw}^{m_d}\varrho(\frac{z_2-m_d}{b w})\,\dd z_2
=bw
\end{eqnarray}
and
\begin{eqnarray}
\int_{\Sigma'_-} \frac{\partial \eta}{\partial z_2} \dd\mathbf{z}
=\int_{z_{01}-b}^{z_{01}+b}\varrho(\frac{z_1-z_{01}}{b})\,\dd z_1
\int_{m_d-bw}^{m_d}\frac{\partial\varrho(\frac{z_2-m_d}{b w})}{\partial z_2}\dd z_2
=b.
\end{eqnarray}

Now we show the following inequalities:
\begin{align}
\Big|\int_{\Sigma'_b}\bar{u}_2\frac{\partial \eta}{\partial z_1} \frac{1}{\bar{\rho} \bar{u}_1}\,\dd\mathbf{z}\Big|
&\leq \sup_{\Sigma_b'} \big\{\frac{\bar{u}_2}{\bar{\rho} \bar{u}_1}\big\}\int_{\Sigma'_b}\Big|\frac{\partial \eta}{\partial z_1}\Big|\,\dd\mathbf{z}\nonumber\\
&\leq \sup_{\Sigma_b'} \big\{\frac{\bar{u}_2}{\bar{\rho} \bar{u}_1}\big\}\int_{\Sigma'_-}\Big|\frac{\partial \eta}{\partial z_1}\Big|\,\dd\mathbf{z}\nonumber\\
&= w \sup_{\Sigma_b'} \big\{\frac{\bar{u}_2}{\bar{\rho} \bar{u}_1}\big\}\int_{\Sigma'_-}\frac{\partial \eta}{\partial z_2}\,\dd\mathbf{z}\nonumber\\
&= w \sup_{\Sigma_b'} \big\{\frac{\bar{u}_2}{\bar{\rho} \bar{u}_1}\big\}\int_{\Sigma'_s}\frac{\partial \eta}{\partial z_2}\,\dd\mathbf{z}
+ w \sup_{\Sigma_b'} \big\{\frac{\bar{u}_2}{\bar{\rho} \bar{u}_1}\big\}\int_{\Sigma'_b}\frac{\partial \eta}{\partial z_2}\,\dd\mathbf{z}.\label{7.2.3.1}
\end{align}
Choosing $\sigma \rho_{\rm cr}\cot{\theta_B}\geq w$, we have
$$
w \sup_{\Sigma_b'} \big\{\frac{\bar{u}_2}{\bar{\rho} \bar{u}_1}\big\}\int_{\Sigma'_b}\frac{\partial \eta}{\partial z_2}\,\dd\mathbf{z}
\leq \inf_{\Sigma_b'} \big\{\frac{\bar{\rho}\bar{q}^2}{\bar{\rho} \bar{u}_1}\big\}\int_{\Sigma'_b}\frac{\partial \eta}{\partial z_2}\,\dd\mathbf{z}\leq
\int_{\Sigma'_b}\bar{\rho} \bar{q}^2\frac{\partial \eta}{\partial z_2} \frac{1}{\bar{\rho} \bar{u}_1}\,\dd\mathbf{z}.
$$
Then
$$
\Big|\int_{\Sigma'_b}\bar{u}_2\frac{\partial \eta}{\partial z_1} \frac{1}{\bar{\rho} \bar{u}_1}\,\dd\mathbf{z}\Big|
\leq \int_{\Sigma'_b}\rho q^2\frac{\partial \eta}{\partial z_2} \frac{1}{\rho u_1}\dd\mathbf{z}+w\sup_{\Sigma_b'}\big\{\frac{u_2}{\rho u_1}\big\}
\int_{\Sigma'_s} \frac{\partial \eta}{\partial z_2}\, \dd\mathbf{z}.
$$

By \eqref{Idvertest}, we have
\begin{align*}
\int_{\Sigma_-}\big(\bar{u}_2\partial_{x_1}\eta-\bar{u}_1\partial_{x_2}\eta\big)\, \dd\mathbf{x}
&\leq\int_{\Sigma_s'}\Big(\bar{u}_2\frac{\partial \eta}{\partial z_1}-\bar{\rho} \bar{q}^2\frac{\partial \eta}{\partial z_2}\Big)\frac{1}{\bar{\rho} \bar{u}_1}\,\dd\mathbf{z}
 +w\sup_{\Sigma_b'}\big\{\frac{\bar{u}_2}{\bar{\rho} \bar{u}_1}\big\} \int_{\Sigma'_s} \frac{\partial \eta}{\partial z_2}\, \dd\mathbf{z}\\
&\leq\int_{\Sigma_s'}\frac{\bar{u}_2}{\bar{\rho}\bar{u}_1}\frac{\partial \eta}{\partial z_1} \dd\mathbf{z}
 +w\sup_{\Sigma_b'}\big\{\frac{\bar{u}_2}{\bar{\rho} \bar{u}_1}\big\} \int_{\Sigma'_s} \frac{\partial \eta}{\partial z_2} \dd\bf{z}\\
&\leq\sup_{\Sigma_s'}\Big|\frac{\partial\eta}{\partial z_1}\Big|\int_{\Sigma_s}|\bar{u}_2|\dd\mathbf{x}
  +w\sup_{\Sigma_b'}\big\{\frac{\bar{u}_2}{\bar{\rho} \bar{u}_1}\big\}\sup_{\Sigma_s'}\big\{\frac{\partial\eta}{\partial z_2}\big\}\int_{\Sigma_s}\bar{\rho} \bar{u}_1 \dd\mathbf{x}\\
&\leq\frac{C\sigma}{b}|\Sigma_s|,
\end{align*}
where $C$ is the uniform positive constant independent of $w$,
and we have used the fact that $|u_2|\leq C\sigma$ in $\Sigma_s$,
by the estimate: $|\theta^{\varepsilon}|\leq\theta_B$.

On the other hand, by \eqref{condition:8.2},
the vorticity $\omega$ is a negative measure with a uniform lower bound.
In fact, we have
\begin{eqnarray}
\omega=-\rho \mathbb{B}'+\frac{\rho^{\gamma}}{\gamma}\mathbb{S}'
=-\rho\Big(\mathbb{B}'-\frac{\mathbb{B}\mathbb{S}'}{\gamma\mathbb{S}}+\frac{q^2}{2}\frac{\mathbb{S}'}{\gamma \mathbb{S}}\Big).
\end{eqnarray}
Note that
\eqref{condition:8.2} is equivalent to the inequality:
$$
\frac{\mathbb{B}'}{\mathbb{B}}-\frac{\mathbb{S}'}{\gamma\mathbb{S}}>C^{-1},\qquad\,\,
\frac{\mathbb{B}'}{\mathbb{B}}-\frac{2}{\gamma(\gamma+1)}\frac{\mathbb{S}'}{\mathbb{S}}>C^{-1}
$$
in the small neighborhood $\Sigma_-$.
Then, by the fact that $0<q<c$,
$\omega$ is strictly negative when $0<q<c$.
This implies that $w\leq-C^{-1}<0$ in $\Sigma_-$.
Therefore, we have
$$
-\frac{C\sigma}{b}|\Sigma_-|\leq -\frac{C\sigma}{b}|\Sigma_s|\leq-\int_{\Sigma_-}\big(u_2\partial_{x_1}\eta-u_1\partial_{x_2}\eta\big)\, \dd\mathbf{x}
= \langle w,\eta \rangle \leq -C^{-1}|\Sigma_-|<0,
$$
which is impossible, if $\sigma$ is chosen small enough.
Thus, we have
$$
\overline{\big(\cup_{0<\dot{m}<m_d} \Omega(\dot{m})\big)\cup \big(\cup_{m_d<\dot{m}<m} \Omega(\dot{m})\big)}\equiv\overline{\Omega}.
$$

4. Therefore, we have proved that $\bar{\psi}$ is a solution of {\bf Problem 3.2($m$)}
with $\mathcal{M}(\bar{\psi})<1$, $\partial_{x_2}\bar{\psi}>0$,  and $|\bar{\theta}|\leq\theta_B$ \textit{a.e.}.
Furthermore, $\Omega(m_d)$ is a Lipschitz curve.
In any compact set $K\subset \Omega\backslash\Omega(m_d)$, $\bar{\psi}\in C^\infty(K)$
and $\partial_{x_2}\bar{\psi}>0$.
This completes the proof.
\end{proof}

In the proof of Lemma \ref{lem:6.5q}, we have also proved the following corollary
on the distance between the
discontinuity and the nozzle walls.

\begin{corollary}
For the weak solution with
the discontinuity obtained by Lemma {\rm \ref{prop:9.1existence}},
the distance between the
discontinuity $\Gamma$ and the solid walls $W_1$ or $W_2$ has a positive lower bound.
\end{corollary}

\begin{proof}
Note that the level set  $\{\psi=m_d\in(0,m)\}$ is the discontinuity,
and $\psi\in C^{2,\alpha}$ in the domain determined by $\psi\in(m_d+\varepsilon,m)$.
Then
$$
m-m_d-\varepsilon=\int_{\hat{x}_2}^{w_2(x_1)}\psi_{x_2}\,\textrm{d}x\leq C\big(w_2(x_1)-\hat{x}_2\big),
$$
where $\psi(x_1,\hat{x}_2)=m_d+\varepsilon$, and
$C$ is determined by \eqref{3.11}.
Therefore, we have
$$
\textrm{dist}(\Gamma,W_2)\geq\frac{m-m_d-\varepsilon}{C}.
$$
This completes the proof.
\end{proof}

Next, we see that the discontinuity is a characteristic discontinuity.

\begin{lemma}
The solution satisfies
\eqref{equ:2.28xw}
as the normal trace in the sense of Chen-Frid \cite{Chen7}.
\end{lemma}

\begin{proof}
Note that the solutions satisfies \eqref{equ:2.28xw} as the normal trace in the sense of Chen-Frid \cite{Chen7} if and only if
\begin{equation}\label{6.17xw}
(\overline{\rho \bu}\cdot\bn)_{\pm}=0\qquad \mbox{on }\Gamma.
\end{equation}
By the construction, the solutions are actually the weak solutions of \eqref{1.5}. Then
\begin{eqnarray}\label{R-HConidtionVortex}
\begin{cases}
[\overline{\rho \bu}\cdot\bn]=0,\\
[\overline{\rho \bu B}\cdot \bn]=0,\\
[\overline{\rho \bu S}\cdot \bn]=0
\end{cases}
\end{eqnarray}
as the normal trace in the sense of Chen-Frid \cite{Chen7}.
Since $(B, S)\in BV$,  then $(\ref{R-HConidtionVortex})_2$--$(\ref{R-HConidtionVortex})_3$
can be rewritten as
$$
[\overline{\rho \bu}\,\bar{B}\cdot \bn]=[\overline{\rho \bu}\, \bar{S}\cdot \bn]=0.
$$
Then, by $(\ref{R-HConidtionVortex})_1$,
$$
(\overline{\rho \bu}\cdot\bn)_{\pm}[\bar{B}]=(\overline{\rho \bu}\cdot\bn)_{\pm}[\bar{S}]=0.
$$
By the definition, we know that, for the entropy wave, $[\bar{S}]\neq0$, so that \eqref{6.17xw} holds.
For the vortex sheet, $[\bar{B}]\neq0$, so that \eqref{6.17xw} also holds.
This completes the proof.
\end{proof}

Finally, let us consider the uniqueness of weak solutions with discontinuities.
Owing to the loss of estimate \eqref{ine:5.32},
the proof is different from that for the smooth case.

\begin{lemma}\label{prop:9.2uniqueness}
The weak solutions obtained in Lemma {\rm \ref{prop:9.1existence}} are unique.
\end{lemma}

\begin{proof}
Let $\psi^{(i)}$, $i=1, 2$, be any two weak solutions.
For each $\psi^{(i)}$, we can introduce the corresponding Lagrangian coordinates to obtain
the corresponding potential function $\varphi^{(i)}$.
From \eqref{weakLequation} and the fact that $(u^{(i)}, p^{(i)})\in L^{\infty}$ in Lemma \ref{prop:9.1existence},
\begin{eqnarray*}
0&=&\int_{\mathbb{R}\times[0, m]} \big(u_2^{(i)}\partial_{z_1} \eta+ p^{(i)}\partial_{z_2}\eta\big)\, \dd \mathbf{z}\nonumber\\
  &=&\int_{\mathbb{R}\times[0, m_d]} \big(u_2^{(i)}\partial_{z_1} \eta+ p^{(i)}\partial_{z_2}\eta\big)\, \dd \mathbf{z}
        +\int_{\mathbb{R}\times[m_d, m]} \big(u_2^{(i)}\partial_{z_1} \eta+ p^{(i)}\partial_{z_2}\eta\big)\, \dd \mathbf{z}\nonumber\\
  &=&\int_{\mathbb{R}\times[0, m_d]}
    \Big(\frac{\partial_{z_1}\varphi^{(i)}}{\rho(\nabla_\bz\varphi^{(i)}, z_2)\partial_{z_2}\varphi^{(i)}}\partial_{z_1} \eta+ p(\nabla_\bz\varphi^{(i)}, z_2)\partial_{z_2}\eta\Big)\,\dd \mathbf{z} \\
&& +\int_{\mathbb{R}\times[m_d, m]}
   \Big(\frac{\partial_{z_1}\varphi^{(i)}}{\rho(\nabla_\bz\varphi^{(i)}, z_2)\partial_{z_2}\varphi^{(i)}}\partial_{z_1} \eta
     + p(\nabla_\bz\varphi^{(i)}, z_2)\partial_{z_2}\eta\Big)\, \dd\mathbf{z}\nonumber
\end{eqnarray*}
for $i=1$, $2$, and any smooth test function $\eta$.

Let $\hat{\varphi}=\varphi^{(1)}-\varphi^{(2)}$.
Therefore, we have
\begin{eqnarray}
0&=&\int_{\mathbb{R}\times[0, m]} \big((u_2^{(1)}-u_2^{(2)})\partial_{z_1} \eta+ (p^{(1)}-p^{(2)})\partial_{z_2}\eta\big)\, \dd \mathbf{z}\nonumber\\
  &=&\int_{\mathbb{R}\times[0, m_d]}
\big((\hat{a}_{11}\partial_{z_1}\hat{\varphi}+\hat{a}_{12}\partial_{z_2}\hat{\varphi})
\partial_{z_1} \eta+ (\hat{a}_{21}\partial_{z_1}\hat{\varphi}+\hat{a}_{22}\partial_{z_2}\hat{\varphi})\partial_{z_2}\eta\big)\, \dd \mathbf{z}\nonumber\\
   & &  + \int_{\mathbb{R}\times[m_d, m]}
\big((\hat{a}_{11}\partial_{z_1}\hat{\varphi}+\hat{a}_{12}\partial_{z_2}\hat{\varphi})
\partial_{z_1} \eta+ (\hat{a}_{21}\partial_{z_1}\hat{\varphi}+\hat{a}_{22}\partial_{z_2}\hat{\varphi})\partial_{z_2}\eta\big)\, \dd \mathbf{z} \label{6.22x}
\end{eqnarray}
with
\begin{eqnarray*}
&&\hat{a}_{11}=\int_1^2\Big(\frac{1}{\rho(\nabla_\bz\varphi^{(\tau)}, z_2)\partial_{z_2}\varphi^{(\tau)}}
   -\frac{(\partial_{z_1}\varphi^{(\tau)})^2}{\rho^3(\nabla_\bz\varphi^{(\tau)}, z_2)(\partial_{z_2}\varphi^{(\tau)})^3(q^2-c^2)}\Big)\,\dd \tau,\\[1mm]
&&\hat{a}_{12}=-\int_1^2\Big(\frac{\partial_{z_1}\varphi^{(\tau)}}{\rho(\nabla_\bz\varphi^{(\tau)}, z_2)(\partial_{z_2}\varphi^{(\tau)})^2}
  -\frac{\partial_{z_1}\varphi^{(\tau)}\big((\partial_{z_1}\varphi^{(\tau)})^2+1\big)}{\rho^3(\nabla_\bz\varphi^{(\tau)},z_2)(\partial_{z_2}\varphi^{(\tau)})^4(q^2-c^2)}\Big)\,\dd \tau,\\[1mm]
&& \hat{a}_{21}=\int_1^2\frac{c^2(\partial_{z_1}\varphi^{(\tau)})}{\rho(\nabla_{\bz}\varphi^{(\tau)}, z_2)(\partial_{z_2}\varphi^{(\tau)})^2(q^2-c^2)}\,\dd \tau,\\[1mm]
&& \hat{a}_{22}=-\int_1^2\frac{c^2\big((\partial_{z_1}\varphi^{(\tau)})^2+1\big)}{\rho(\nabla_{\bz}\varphi^{(\tau)}, z_2)(\partial_{z_2}\varphi^{(\tau)})^3(q^2-c^2)}\, \dd \tau,
\end{eqnarray*}
where
$\varphi^{(\tau)}=(2-\tau)\varphi^{(1)}+(\tau-1)\varphi^{(2)}$
for some $\tau\in(1,2)$.

\smallskip
Note that $|\frac{u_2}{u_1}|=|\tan\theta|\leq \tan\theta_B$ and $\frac{1}{\rho u_1}>0$.
Since
$\nabla_\bz\varphi=(\frac{u_2}{u_1}, \frac{1}{\rho u_1})$,
then
$|\partial_{z_1}\varphi|$ is bounded in $L^\infty$
and, on the $z_2$--direction,
\begin{align*}
\int_0^{z_2}|\partial_{z_2}\varphi(z_1, s)| \dd s
&=\int_0^{z_2}|\frac{1}{\rho u_1}(z_1, s)| \dd s
=\int_0^{z_2}\frac{1}{\rho u_1}(z_1, s)\dd s
= \int_0^{z_2}\partial_{z_2}\varphi(z_1, s) \dd s\\[2mm]
&=\varphi(z_1, z_2)-\varphi(z_1, 0)
\le \max_{x_1\in \R}\,(w_2-w_1)(x_1).
\end{align*}
This implies that  $\varphi\in BV_{\rm loc}$.

Let $\eta_L$ be a smooth cut-off function supported on $[-L,L]$ and identically equal to $1$ on $[-L+1,L-1]$.
Let $\eta=\hat{\varphi}\eta_L$.
Then \eqref{6.22x} holds for this $\eta$, since $C_{\rm c}^{\infty}$ is dense in $BV$ and
$(u^{(i)}, p^{(i)})\in L^{\infty}$.
Note that, for the $L^{\infty}$ functions $(u,v)$,
\begin{align*}
&\int_{\mathbb{R}\times[0, m]} \big(u\partial_{z_1} \eta+ v\partial_{z_2}\eta\big)\, \dd \mathbf{z}\\
&= \int_{\mathbb{R}\times[0, m]} \big(u\hat{\varphi}\partial_{z_1} \eta_L+ v\hat{\varphi}\partial_{z_2}\eta_L\big)\, \dd \mathbf{z}
        +\int_{\mathbb{R}\times[0, m]} \big(u\eta_L\partial_{z_1}\hat{\varphi} + v\eta_L\partial_{z_2}\hat{\varphi}\big)\, \dd \mathbf{z}.
\end{align*}
Therefore, we have
\begin{align}
&\int_{[-L+1,L-1]\times[0, m]} \big(u\partial_{z_1}\hat{\varphi} + v\partial_{z_2}\hat{\varphi}\big)\, \dd \mathbf{z}\nonumber\\
&= \int_{\mathbb{R}\times[0, m]} \big(u\partial_{z_1} \eta+ v\partial_{z_2}\eta\big)\, \dd \mathbf{z}
-\int_{([-L, -L+1]\cup[L-1,L])\times[0, m]} \big(u\hat{\varphi}\partial_{z_1} \eta_L+ v\hat{\varphi}\partial_{z_2}\eta_L\big)\, \dd \mathbf{z}\nonumber\\
&\quad -\int_{([-L, -L+1]\cup[L-1,L])\times[0, m]} \big(u\eta_L\partial_{z_1}\hat{\varphi} + v\eta_L\partial_{z_2}\hat{\varphi}\big)\, \dd \mathbf{z}. \label{weakLx}
\end{align}

By \eqref{equ:2.59a}--\eqref{equ:2.60a} in Lemma \ref{lem:6.2x},
the fact that $(u, v)\in L^{\infty}$ and $\hat{\varphi}\in BV$ when $L$ is sufficiently large in Lemma \ref{prop:9.1existence},
and the dominant convergence theorem, passing the limit $L\rightarrow\infty$, \eqref{weakLx} becomes
\begin{eqnarray}
\lim_{L\rightarrow\infty}\int_{\mathbb{R}\times[0, m]}
\big(u\partial_{z_1} \eta+ v\partial_{z_2}\eta\big)\, \dd \mathbf{z}
&=&\int_{\mathbb{R}\times[0, m]} \big(u\partial_{z_1} \hat{\varphi}+ v\partial_{z_2}\hat{\varphi}\big)\, \dd \mathbf{z}.\label{weakLa}
\end{eqnarray}

Plugging \eqref{weakLa} into \eqref{6.22x}, we have
\begin{equation}\label{UUUUUU1}
\int_{\mathbb{R}\times[0, m]}
\big(
(\hat{a}_{11}\partial_{z_1}\hat{\varphi}+\hat{a}_{12}\partial_{z_2}\hat{\varphi})
\partial_{z_1} \hat{\varphi}+ (\hat{a}_{21}\partial_{z_1}\hat{\varphi}+\hat{a}_{22}\partial_{z_2}\hat{\varphi})\partial_{z_2}\hat{\varphi}\big)\, \dd \mathbf{z}=0.
\end{equation}

By the ellipticity, we know that, for any $\xxi\in\mathbb{R}^2\backslash\{\mathbf{0}\}$, there exist positive constants $C$ and $C'$ such that
$$
0 \leq C\int_1^2 \frac{\dd\tau}{\rho(\nabla_\bz\varphi^{(\tau)}, z_2)\big(\partial_{z_2}\varphi^{(\tau)}\big)^3}\, |\xxi|^2
\leq\sum_{j,k=1}^2 \hat{a}_{jk}\xi_j\xi_k
\leq C'\int_1^2 \frac{\dd\tau}{\rho(\nabla_\bz\varphi^{(\tau)}, z_2)\partial_{z_2}\varphi^{(\tau)}}\, |\xxi|^2.
$$
Then \eqref{UUUUUU1} becomes
\begin{equation}
\int_{\mathbb{R}\times[0, m]}
\Big(\int_1^2\frac{\dd\tau}{\rho(\nabla_\bz\varphi^{(\tau)}, z_2)\big(\partial_{z_2}\varphi^{(\tau)}\big)^3}\, |\nabla_{\mathbf{z}} \hat{\varphi}|^2\Big)
\,\dd \mathbf{z}=0.
\end{equation}
	
On the other hand,
$\int_1^2 \frac{\dd\tau}{\rho(\nabla_\bz\varphi^{(\tau)}, z_2)\big(\partial_{z_2}\varphi^{(\tau)}\big)^3}=0$
if and only if $\partial_{z_2}\hat{\varphi}=0$.
Therefore, we conclude that  $\varphi^{(1)}=\varphi^{(2)}$.
\end{proof}

\medskip
Now, we can complete the proof of Theorem \ref{thm:discontinuity}.

\begin{proof}[Proof of Theorem {\rm \ref{thm:discontinuity}}]
We employ a similar argument to the Bers skill in \cite{Bers2}.

Before the limiting process,  we introduce an approximate problem
in the case that $(u_-, S_-)$ have the jump point $x_2=x_d$,
so that \textbf{Problem 5.1($m$, $\varepsilon$)} satisfies the inlet conditions, since
$(u_-^{\varepsilon},S_-^{\varepsilon})$ satisfy the assumptions of Theorem \ref{thm:smooth}, which
converge to $(u_-, S_-)$ pointwise in $[0,1]$ and $C^{1,1}$ away from the jump point $x_2=x_d$.

Let $\{\epsilon_i\}_{i=1}^{\infty}$ be a strictly decreasing sequence such that $\epsilon_i\rightarrow0$ as $i\rightarrow\infty$.
By the construction of the solution, there exists a maximal interval $(\bar{m}_i, \infty)$ such that, for $m\in(\bar{m}_i,\infty)$,
$$
\mathcal{M}(\psi(x; m))<-2\epsilon_i.
$$
For each fixed $\varepsilon>0$, $(\bar{m}_i^\varepsilon, \infty)$ is the maximal interval for $m\in(\bar{m}_i^\varepsilon,\infty)$ such that
$$
\mathcal{M}(\psi^\varepsilon(x; m))<-2\epsilon_i,
$$
where $\psi^\varepsilon(x; m)$ is the solution of \textbf{Problem 5.1($m$, $\varepsilon$)}.
Also $\bar{m}_i=\varliminf_{\varepsilon\rightarrow0} \bar{m}_i^\varepsilon$,

By property of the lower limit, we can see that $\bar{m}_j\ge\bar{m}_i$ for $j>i$.
Thus, $\{\bar{m}_i\}_{i=1}^{\infty}$ is a nonnegative decreasing sequence,
which implies its convergence.
As a consequence, we have
$$
m_c:=\lim_{i\rightarrow\infty}\bar{m}_i=\lim_{i\rightarrow\infty}\varliminf_{\varepsilon\rightarrow0} \bar{m}_i^\varepsilon
=\varliminf_{\varepsilon\rightarrow0} \lim_{i\rightarrow\infty} \bar{m}_i^\varepsilon=\varliminf_{\varepsilon\rightarrow0}m_{\rm c}^\varepsilon,
$$
where $m_{\rm c}^\varepsilon$ is the critical mass of \textbf{Problem 5.1($m$, $\varepsilon$)}.

If $\bar{m}_i>m_{\rm c}$ for any $i$,
then, for any $m\in (m_{\rm c},\infty)$,
there exist an index $i$ such that $\bar{m}_{i}<m\le\bar{m}_{i-1}$.
The solution of \textbf{Problem 3.2($m$)} is $\psi(x; m)$ with the estimate:
$$
-2\epsilon_{i-1}\le\mathcal{M}(\psi(x; m))<-2\epsilon_i.
$$
Then $\mathcal{M}(\psi(x; m))\rightarrow 0$ as $m\rightarrow m_{\rm c}$.

On the other hand, if $\bar{m}_i\equiv m_c$ when $i$ is large enough, then, by construction,
there does not exist $\sigma>0$ such that, for $\varepsilon_0$ small enough,
\textbf{Problem 5.1($m$, $\varepsilon$)} has solutions
for all $0<\varepsilon\leq \varepsilon_0$ and  $m\in(m_{\rm c}-\sigma,m_{\rm c})$
so that $\mathcal{M}(\psi^\varepsilon(x; m))<0$.
This means that some structures like shock bubbles may appear before the subsonic flows become subsonic-sonic flows.
\end{proof}

\section{Subsonic-Sonic Limits and Incompressible Limits}\label{section:7b}
In this section, we study the subsonic-sonic limit and the incompressible limit of the solutions
obtained in Theorems \ref{thm:smooth}--\ref{thm:discontinuity}.
Since the arguments for the two cases are similar,
we consider only the harder case that the sequence of solutions
is obtained via Theorem \ref{thm:discontinuity}, which allows discontinuities.

\subsection{Subsonic-sonic limits}
We first show the subsonic-sonic limit.

\begin{theorem}\label{thm5.2}
Let $m^{(\v)}>\tilde{m}_{\rm c}$ be a sequence of mass fluxes,
and let $(\rho^{(\v)}, \bu^{(\v)}, p^{(\v)})(\bx)$
be the corresponding sequence of solutions of  {\rm \bf Problem 2.1($m^{(\v)}$)}.
Then, as $m^{(\v)}\rightarrow \tilde{m}_{\rm c}$, the solution
sequence possesses a subsequence {\rm (}still denoted by{\rm )}
$(\rho^{(\v)}, \bu^{(\v)}, p^{(\v)})(\bx)$ that converges  to
a vector function $(\rho, \bu, p)(\bx)$ strongly {\it a.e.} in $\Omega$.
Furthermore, the limit function $(\rho, \bu, p)(\bx)$ satisfies $(\ref{1.5})$
in the distributional sense and the boundary condition \eqref{cdx-sc}
on $\partial \Omega$ as the normal trace of the divergence-measure field
in the sense of Chen-Frid {\rm \cite{Chen7}}.
\end{theorem}

\begin{proof}
We divide the proof into three steps.

\smallskip
1. First, we recall the compactness condition in \cite{ChenHuangWang}:

\medskip
(A.1). $M^{(\v)}(\bx)\leq 1$  {\it a.e.} $\bx\in \Omega$;

\vspace{1mm}

(A.2). $(B^{(\v)}, S^{(\v)})$ are uniformly bounded and,
for any compact set $K$, there exists a uniform constant $c(K)$ such
that $\inf\limits_{\bx\in K} S^{(\v)}(\bx)\ge c(K)>0$.
Moreover,  $(B^{(\v)}, S^{(\v)})(\bx)\to (\overline{B}, \overline{S})$
{\it a.e.} $\bx\in \Omega$;

\vspace{1mm}
(A.3). $\mbox{curl}\  \bu^{(\v)}$ and $\mbox{div}(\rho^{(\v)} \bu^{(\v)})$ are in a compact set
in $W_{loc}^{-1, p}$ for some $1<p \le 2$.

\medskip
Next, we show that $(\rho^{(\v)}, \bu^{(\v)}, p^{(\v)})(\bx)$
satisfy conditions (A.1)--(A.3).

For $(B^{(\v)}, S^{(\v)})$, we have
\begin{eqnarray}\label{3.2.1}
\begin{cases}
\partial_{x_1}(\rho^{(\v)} u_1^{(\v)})+\partial_{x_2}(\rho^{(\v)} u_2^{(\v)})=0,\\[2mm]
\partial_{x_1}(\rho^{(\v)} u_1^{(\v)} B^{(\v)})+\partial_{x_2}(\rho^{(\v)} u_2^{(\v)} B^{(\v)})=0,\\[2mm]
\partial_{x_1}(\rho^{(\v)} u_1^{(\v)} S^{(\v)})+\partial_{x_2}(\rho^{(\v)} u_2^{(\v)} S^{(\v)})=0.
\end{cases}
\end{eqnarray}

From $(\ref{3.2.1})_1$, we introduce the following stream function $\psi^{(\v)}$:
\begin{eqnarray*}
\nabla_{\bx}\psi^{(\v)}=\rho^{(\v)} (-u_2^{(\v)}, u_1^{(\v)}),
\end{eqnarray*}
which means that $\psi^{(\v)}$ is constant along the streamlines.

From the far-field behavior of the Euler flows,
we define
$$
\psi^{(\v)}_-(x_2):=\lim\limits_{x_1\rightarrow-\infty}\psi^{(\v)}(x_1, x_2).
$$
Since both the upstream Bernoulli function $B_-$ and entropy function $S_-$ are given,
$(B^{(\v)}, S^{(\v)})$ have the following expression:
$$
B^{(\v)}(\bx)=B_{-}((\psi^{(\v)}_{-})^{-1}(\psi^{(\v)}(\bx))),
\qquad S^{(\v)}(\bx)=S_{-}((\psi^{(\v)}_{-})^{-1}(\psi^{(\v)}(\bx))),
$$
where $(\psi^{(\v)}_{-})^{-1}\psi^{(\v)}(\bx)$ is a function
from $\Omega$ to $[0,1]$, which
can be regarded as a backward characteristic map for fixed $x_1$ with
$$
\frac{\partial ((\psi^{(\v)}_{-})^{-1}\psi^{(\v)})}{\partial x_2}
(\bx)=\frac{\rho^{(\v)} u_1^{(\v)}}{\rho^{(\v)}_-u^{(\v)}_{1-}}(\bx)>0
\qquad  a.e.\,\, \bx\in \Omega.
$$
The boundedness and positivity of $\rho^{(\v)}_-u_{1-}^{(\v)}$
and  $\rho^{(\v)} u_1^{(\v)}$ show that
the map is not degenerate.
Then we have
\begin{eqnarray}
\begin{cases}
\partial_{x_1}B^{(\v)}(\bx)=-B'_{-}((\psi^{(\v)}_{-})^{-1}(\psi^{(\v)}(\bx)))\frac{\rho^{(\v)} u^{(\v)}_2}{\rho^{(\v)}_-u^{(\v)}_{1-}}(\bx),\\[2mm]
\partial_{x_2}B^{(\v)}(\bx)=B'_{-}((\psi^{(\v)}_{-})^{-1}(\psi^{(\v)}(\bx)))\frac{\rho^{(\v)} u^{(\v)}_1}{\rho^{(\v)}_-u^{(\v)}_{1-}}(\bx).
\end{cases}
\end{eqnarray}
Thus, $B^{(\v)}$ is uniformly bounded in $BV$, which implies its strong convergence.
The similar argument can lead to the strong convergence of $S^{(\v)}$.

\medskip
2. For the corresponding vorticity sequence $\omega^{(\v)}$,  we have
\begin{eqnarray}
\begin{cases}
\partial_{x_1} B^{(\v)}=  u_2^{(\v)} \omega^{(\v)}+\frac{1}{\gamma-1}(\rho^{(\v)})^{\gamma-1}\partial_{x_1} S^{(\v)},\\[2mm]
\partial_{x_2} B^{(\v)}=  - u_1^{(\v)} \omega^{(\v)}+\frac{1}{\gamma-1}(\rho^{(\v)})^{\gamma-1}\partial_{x_2} S^{(\v)}.
\end{cases}
\end{eqnarray}
By a direct calculation, we have
\begin{eqnarray}
\omega^{(\v)}&=&\frac{1}{(q^{(\v)})^2}
\Big(u_2^{(\v)}\big(\partial_{x_1} B^{(\v)}-\frac{1}{\gamma-1}(\rho^{(\v)})^{\gamma-1} \partial_{x_1} S^{(\v)}\big)
-u_1^{(\v)}\big(\partial_{x_2} B^{(\v)}-\frac{1}{\gamma-1}(\rho^{(\v)})^{\gamma-1}\partial_{x_2} S^{(\v)}\big)\Big)\nonumber\\
&=&\frac{1}{\rho_-^{(\v)} u_{1-}^{(\v)}}\Big(-\rho^{(\v)} B_-'+\frac{1}{\gamma-1}(\rho^{(\v)})^\gamma S_-'\Big),
\end{eqnarray}
which implies that $\omega^{(\v)}$ as a measure sequence is uniformly bounded and compact in $H^{-1}_{loc}$.

\smallskip
Then Theorem 2.4 in \cite{ChenHuangWang}
implies that
the solution sequence has a subsequence (still denoted
by) $(\rho^{(\v)}, \bu^{(\v)}, p^{(\v)})(\bx)$
that converges  to
a vector function $(\rho, \bu, p)(\bx)$ {\it a.e.} $\bx\in \Omega$.

\smallskip
Since
\eqref{1.5} holds for the sequence of subsonic solutions
$(\rho^{(\v)}, \bu^{(\v)}, p^{(\v)})(\bx)$,
then
the limit vector function $(\rho, \bu, p)(\bx)$ also
satisfies \eqref{1.5} in the distributional sense.

\medskip
3. The boundary condition is satisfied
in the sense of Chen-Frid \cite{Chen7}, which  implies
\begin{equation*}
\int_{\partial \Omega}\phi(\bx)(\rho \bu)(\bx)\cdot {\boldsymbol \nu} (\bx)\, \dd\mathcal{H}^{1}(\bx)
= \int_{\Omega} (\rho \bu)(\bx) \cdot \nabla \phi(\bx)\, \dd\bx + \langle \mbox{div}(\rho \bu)|_{\Omega}, \phi\rangle
\qquad \mbox{for $\psi\in C^1_0$}.
\end{equation*}
From above, we can see that $\langle \mbox{div}(\rho \bu)|_{\Omega}, \phi\rangle=0$. Moreover, we have
\begin{equation*}
\int_{\Omega}(\rho \bu)(\bx) \cdot \nabla \phi(\bx)\, \dd\bx
=\lim\limits_{\varepsilon \rightarrow 0}\int_{\Omega}(\rho^\varepsilon \bu^\varepsilon)(\bx) \cdot \nabla \phi(\bx)\, \dd\bx=0.
\end{equation*}
Then we have
\begin{equation*}
\int_{\partial \Omega}\phi(\bx)(\rho \bu)(\bx)\cdot{\boldsymbol \nu}(\bx)\, \dd\mathcal{H}^{1}(\bx)=0,
\end{equation*}
that is, $(\rho \bu)\cdot {\boldsymbol \nu} = 0$ on $\partial \Omega$ in $\mathcal{D}'$.
\end{proof}

\subsection{Incompressible limits}
Introduce the following quantity which is uniform with respect to the limit $\gamma\rightarrow\infty$:
$$
G := \rho p^{-\frac{1}{\gamma}}\ge 0.
$$
Then $S=\frac{\gamma}{\gamma-1}G^{-\gamma}$.
Similar to {\bf Problem 2.1($m$)}, we introduce

\medskip
\textbf{Problem 7.1($m$, $\gamma$)}:  For given $(u_{1-}, G_-)(x_2)$,
find $(\rho, \bu, p)$ that satisfies \eqref{1.5}, with the condition \eqref{cdx-6}, \eqref{cdx-11}, and
the upstream conditions:
\begin{equation}\label{cd-8}
\rho p^{-\frac{1}{\gamma}}\longrightarrow G_-(x_2)
\qquad \mbox{as} ~x_1\rightarrow -\infty.
\end{equation}

Now we have the following theorem.
\begin{theorem}\label{thm5.3}
For any $\bar{m}>0$, there exists a convergence sequence $\{m^{(\gamma)}\}$ with $\lim_{\gamma\rightarrow\infty}m^{(\gamma)}=\bar{m}$
such that $(\rho^{(\gamma)}, \bu^{(\gamma)}, p^{(\gamma)})(\bx)$ is the solution of {\bf Problem 7.1($m^{(\gamma)}, \gamma$)}.
Then, as $\gamma\rightarrow \infty$, the solution
sequence possesses a subsequence {\rm (}still denoted by{\rm )}
$(\rho^{(\gamma)}, \bu^{(\gamma)}, p^{(\gamma)})(\bx)$
that converges  to
a vector function $(\bar{\rho}, \bar{\bu}, \bar{p})(\bx)$ strongly {\it a.e.}  $\bx\in \Omega$,
which is a weak solution of 	
	\begin{eqnarray}\label{ICIHE}\label{ICEuler}
	\begin{cases}
	\mbox{\rm div}\,\bar{\bu}=0,\\[1mm]
	\mbox{\rm div}(\bar{\rho}\bar{\bu})=0,\\[1mm]
	\mbox{\rm div}(\bar{\rho}\bar{\bu}\otimes\bar{\bu})+\nabla\bar{p}=0.
	\end{cases}
	\end{eqnarray}
Furthermore, the limit solution $(\bar{\rho}, \bar{\bu}, \bar{p})(\bx)$ also satisfies
the boundary condition $\bu\cdot {\boldsymbol \nu} = 0$ as the normal trace of the divergence-measure
field $\bu$ on the boundary in the sense of Chen-Frid {\rm \cite{Chen7}}.
\end{theorem}

\begin{remark}
In Theorem {\rm \ref{thm5.3}}, $\bar{m}$ can be any arbitrary positive constant.
\end{remark}

\begin{proof}
For any fixed $\gamma>1$, denote $m^{(\gamma)}_{\rm c}$ as the critical mass of {\bf Problem 7.1($m^{(\gamma)}, \gamma$)}.
From estimates \eqref{equ:4.10a}--\eqref{equ:4.11a},
we see that $m_{\rm c}^{(\gamma)}\rightarrow0$ as $\gamma\rightarrow\infty$.
By Theorem \ref{thm:smooth},
for any given $\gamma>1$,
there exists $\underline{m}^{(\gamma)}$ such that, when $m>\underline{m}^{(\gamma)}$,
there exists a unique $\psi^{(\gamma)}$ satisfying
	\begin{eqnarray*}
	\begin{cases}
	\partial_{x_1}\psi^{(\gamma)}=-\rho^{(\gamma)} u_2^{(\gamma)},\\[1mm]
	\partial_{x_2}\psi^{(\gamma)}=\rho^{(\gamma)} u_1^{(\gamma)},\\[1mm]
	\mathbb{B^{(\gamma)}}(\psi^{(\gamma)})=\frac{1}{2}|\bu^{(\gamma)}|^2+\frac{\gamma}{\gamma-1}\frac{ p^{(\gamma)}}{\rho^{(\gamma)}},\\[1mm]
G^{(\gamma)} (\psi^{(\gamma)}) = \frac{\rho^{(\gamma)}}{(p^{(\gamma)})^{\frac{1}{\gamma}}},\\[1mm]
	\partial_{x_1}u_2^{(\gamma)}-\partial_{x_2}u_1^{(\gamma)}=-\rho^{(\gamma)} (\mathbb{B}^{(\gamma)} )'-\frac{\gamma}{\gamma-1}\frac{(p^{(\gamma)})^{\frac{\gamma-1}{\gamma}}} {\rho^{(\gamma)}}(\mathbb{G}^{(\gamma)} )',
	\end{cases}
	\end{eqnarray*}
	with
	\begin{equation*}
	\psi^{(\gamma)}|_{x_2=w_1(x_1)}=0,\qquad \psi^{(\gamma)}|_{x_2=w_2(x_1)}=m,
	\end{equation*}
$\partial_{x_2}\psi^{(\gamma)}>0$, and $|\theta^{(\gamma)}|\leq \theta_B$.

We then show the almost everywhere convergence of $(\rho^{(\gamma)}, \bu^{(\gamma)}, p^{(\gamma)})(\bx)$,
by employing the compensated compactness framework established in Chen-Huang-Wang-Xiang \cite{ChenHuangWangXiang}.
Here $B^{(\gamma)}$, $S^{(\gamma)}$, and $\omega^{(\gamma)}$ are defined as:
\begin{eqnarray*}
\begin{cases}
	B^{(\gamma)}(\bx):=\mathbb{B^{(\gamma)}}(\psi^{(\gamma)}(\bx))=\frac{1}{2}|\bu^{(\gamma)}(\bx)|^2+\frac{\gamma}{\gamma-1}\frac{ p^{(\gamma)}}{\rho^{(\gamma)}}(\bx),\\[1mm]
	G^{(\gamma)}(\bx):=\mathbb{G^{(\gamma)}}(\psi^{(\gamma)})(\bx)=(\rho^{(\gamma)}(p^{(\gamma)})^{-\frac{1}{\gamma}})(\bx),\\[1mm]
	\omega^{(\gamma)}(\bx):=-\rho^{(\gamma)} (\mathbb{B}^{(\gamma)} )'-\frac{\gamma}{\gamma-1}\frac{(p^{(\gamma)})^{\frac{\gamma+1}{\gamma}}} {\rho^{(\gamma)}}(\mathbb{G}^{(\gamma)} )'.
	\end{cases}
	\end{eqnarray*}
By a direct calculation, we have
	\begin{eqnarray*}
	\begin{cases}
	\partial_{x_1}B^{(\gamma)}(\bx)=-(\mathbb{B}^{(\gamma)})'(\psi^{(\gamma)})(\bx)\,(\rho^{(\gamma)} u_2^{(\gamma)})(\bx),\\[2mm]
	\partial_{x_2}B^{(\gamma)}(\bx)=(\mathbb{B}^{(\gamma)})'(\psi^{(\gamma)})(\bx)\,(\rho^{(\gamma)} u_1^{(\gamma)})(\bx).
	\end{cases}
	\end{eqnarray*}
Then $B^{(\gamma)}$ is uniformly bounded in $BV$, which implies its strong convergence.
The similar argument leads to the strong convergence of $G^{(\gamma)}$.
The vorticity sequence $\omega^{(\gamma)}$ is a measure sequence uniformly bounded, and compact in $H^{-1}_{loc}$.
	
Combining the subsonic condition $M(\psi^{(\gamma)})<1$ and
\begin{equation*}
\Big(\frac{2(\gamma-1)}{\gamma(\gamma+1)}\min (\mathbb{B}^{(\gamma)} \mathbb{G}^{(\gamma)})\Big)^{\frac{\gamma}{\gamma-1}}
	< p^{(\gamma)}
\le\Big(\frac{\gamma-1}{\gamma}\max (\mathbb{B}^{(\gamma)} \mathbb{G}^{(\gamma)})\Big)^{\frac{\gamma}{\gamma-1}}.
\end{equation*}
Theorem 2.2 in Chen-Huang-Wang-Xiang \cite{ChenHuangWangXiang} implies that
the solution sequence has a subsequence (still denoted 	by)
$(\rho^{(\gamma)}, \bu^{(\gamma)}, p^{(\gamma)})(\bx)$
that converges to a vector function
$(\bar{\rho}, \bar{\bu}, \bar{p})(\bx)$ {\it a.e.} in $\Omega$.
Then $\bar{\psi}$ can also be introduced as
	\begin{eqnarray*}
	\nabla_\bx\bar{\psi}=\bar{\rho}(-\bar{u}_2, \bar{u}_1),
	\end{eqnarray*}
with  $\partial_{x_2}\bar{\psi}\geq0$, $|\bar{\theta}|\leq\theta_B$,
$\bar{\psi}|_{x_2=w_1(x_1)}=0$, and $\bar{\psi}|_{x_2=w_2(x_1)}=m$.
Also, it is easy to see that $\psi^{(\gamma)}$ converges to $\bar{\psi}$
in $Lip(\Omega)$.
Moreover, $\bar{\rho}>0$ and
$p^{(\gamma)}=\frac{\gamma-1}{\gamma}\rho^{(\gamma)}\big(B^{(\gamma)}-\frac{1}{2}|\bu^{(\gamma)}|^2\big)\rightarrow \bar{p}$ almost everywhere in $\Omega$.
\end{proof}

\begin{remark}
Different from the proof of the subsonic-sonic limit, we need to construct the incompressible solutions with free boundary
via the incompressible limit with the approximating solutions obtained by Theorem {\rm \ref{thm:smooth}}.
Therefore, the proof of the incompressible limit is different from that of the subsonic-sonic limit,
{i.e.} Theorem {\rm \ref{thm5.2}}, from the very beginning.
\end{remark}

\section{Well-posedness of Inhomogeneous Incompressible Euler Flows}\label{section:8b}
Finally, we study the regularity and uniqueness of solutions of
the incompressible Euler equations obtained by Theorem \ref{thm5.3}.

For the conditions at the inlet for the smooth case,
we require that $(\rho_-, u_{1-})\in (C^{1,1}([0,1]))^2$
and satisfy
\begin{equation}\label{assumption:3.40ic}
m=\int_0^1(\rho_-u_{1-})(x_2)\textrm{d}x_2,\quad\, \inf\limits_{x_2\in[0,1]}u_{1-}(x_2)>0,\quad\,
\inf\limits_{x_2\in[0,1]}\rho_{-}(x_2)>0.
\end{equation}
On the boundary, we always assume the following monotone properties for $(\rho_-, u_{1-})$:
\begin{equation}\label{assumption:3.47aic}
(\rho_-u_{1-}^2)'(0)\leq0,
\qquad\,
(\rho_-u_{1-}^2)'(1)\geq0.
\end{equation}

Similar to the smooth case, for the conditions at the inlet of the non-smooth case,
we require that $(\rho_-, u_{1-})\in (C^{1,1}[0,x_d))^2\cup (C^{1,1}(x_d,1])^2$
and satisfy \eqref{assumption:3.40ic}--\eqref{assumption:3.47aic}.
Moreover, $(\rho_-, u_{1-})$ satisfy either
\begin{equation}\label{assumption:4.2ic}
\begin{cases}
\inf\limits_{x_{d}-\varepsilon_0<s<x_{d}}(\rho_-u_{1-}^2)'(s)\geq0,\\[1mm]
\sup\limits_{x_{d}-\varepsilon_0<s< x_{d}}\rho_{-}'(s)\leq0,\\[2mm]
[u_{1-}^2\rho_-]\geq0,\\[1mm]
[\rho_{-}]\leq0,
\end{cases}
\end{equation}
or
\begin{equation}\label{assumption:4.3ic}
\begin{cases}
\sup\limits_{x_{d}<s<x_{d}+\varepsilon_0}(\rho_-u_{1-}^2)'(s)\leq0,\\[1mm]
\inf\limits_{x_{d}<s<x_{d}+\varepsilon_0}\rho_{-}'(s)\geq0,\\[2mm]
[u_{1-}^2\rho_-]\leq0, \\[1mm]
[\rho_{-}]\geq0.
\end{cases}
\end{equation}

Finally, we have the following theorem for the regularity and uniqueness of solutions of
the inhomogeneous incompressible Euler equations via the incompressible limit.

\begin{theorem}[Incompressible Euler flows]\label{thm:smoothIC}
For any given density function $\rho_{-}$ and horizontal velocity $u_{1-}$ at the inlet,
which either are $C^{1,1}$ and satisfy assumptions \eqref{assumption:3.40ic}--\eqref{assumption:3.47aic}
or are piecewise $C^{1,1}$ and satisfy assumptions \eqref{assumption:3.40ic}--\eqref{assumption:4.3ic},
there exists a solution $(\rho, \bu, p)$ of the inhomogeneous incompressible Euler equations \eqref{ICIHE}
with conditions \eqref{cdx-6}--\eqref{cdx-11} satisfying
$$
|\theta|\le \theta_B, \qquad \min_{\partial\Omega} p\le p\le \max_{\partial\Omega} p,
$$
and
\begin{enumerate}
\item[(i)] When $(\rho_{-}, u_{1-})$ are $C^{1,1}$--functions,
then solution $(\rho, \bu,p)\in (C^{1,\alpha}(\Omega))^4$.
The solution satisfies the additional properties \eqref{equ:2.59a}--\eqref{equ:2.60a}
as $x_1\rightarrow\pm\infty$
uniformly for $x_2\in K\Subset(0,1)$.
Moreover, the solution satisfies
\eqref{ieq:nondegeneracy0}--\eqref{ieq:streamline}.
	
\smallskip
\item[(ii)] When $(\rho_{-}, u_{1-})$ are piecewise $C^{1,1}$--functions,
then $(\rho, \bu,p)\in (C^{1,\alpha}(\Omega)\backslash\Gamma)^4$
and is a solution of the inhomogeneous incompressible Euler equations \eqref{ICEuler} in the weak sense.
The solution satisfies the additional properties \eqref{equ:2.59a}$-$\eqref{equ:2.60a} as $x_1\rightarrow\pm\infty$
uniformly for $x_2\in K\Subset(0,1)$ away from the discontinuity $\Gamma$.
The discontinuity $\Gamma$ is a streamline with the Lipschitz regularity.
Moreover, the solution satisfies
\eqref{ieq:nondegeneracy0}--\eqref{ieq:streamline}.
\end{enumerate}
\end{theorem}

\begin{proof}
The proof of the result corresponding to the smooth solutions is similar, but simpler than that of the result
corresponding to the weak solutions with discontinuities.
Therefore, we consider only the non-smooth case.
We suppose that $(\bar{\rho}, \bar{\bu}, \bar{p})$ is the limit in Theorem \ref{thm5.3},
satisfying \eqref{ICEuler} weakly. The proof is divided into three steps.
	
\medskip
1. First, we show the $C^{2, \alpha}$--convergence away from the discontinuity. The argument here is similar to the one of  Lemma \ref{prop:9.1existence}.
We give its outline for completeness.

For fixed $\psi^{(\gamma)}$, define
$$
\Omega^{(\gamma)}(\dot{m})=\begin{cases}
\{\bx \;:\; \bx\in\Omega, \psi^{(\gamma)}(\bx)<\dot{m} \} \qquad \mbox{when $0<\dot{m}<m_d$},\\
\{\bx \;:\; \bx\in\Omega, \psi^{(\gamma)}(\bx)>\dot{m} \}\qquad \mbox{when $m_d<\dot{m}<m$}.
\end{cases}
$$
Then the following properties hold:
\begin{enumerate}
\item[(i)] For $0<\dot{m}_1<\dot{m}_2<m_d$, $\Omega^{(\gamma)}(\dot{m}_1)\subsetneq\Omega^{(\gamma)}(\dot{m}_2)$;

\smallskip
\item[(ii)] For $m_d<\dot{m}_1<\dot{m}_2<m$, $\Omega^{(\gamma)}(\dot{m}_1)\supsetneq\Omega^{(\gamma)}(\dot{m}_2)$;

\smallskip
\item[(iii)] For each $\gamma$, $\overline{\Omega}=\overline{\big(\cup_{0<\dot{m}<m_d} \Omega^{(\gamma)}(\dot{m})\big)\cup \big(\cup_{m_d<\dot{m}<m} \Omega^{(\gamma)}(\dot{m})\big)}$.
\end{enumerate}

Similarly, we can define the sets with respect to $\bar{\psi}$:
$$
\Omega(\dot{m})=\begin{cases}
\{\bx \,:\, \bx\in\Omega, \bar{\psi}(\bx)<\dot{m} \} \qquad\mbox{when $0<\dot{m}<m_d$},\\[1mm]
\{\bx \,:\, \bx\in\Omega, \bar{\psi}(\bx)>\dot{m} \} \qquad\mbox{when $m_d<\dot{m}<m$},
\end{cases}
$$
and
$$
\Omega(m_d)=\{\bx\,:\, \bx\in\Omega, \bar{\psi}(\bx)=m_d \},
$$
with the following properties:
\begin{enumerate}	
\item[(i)] For $0<\dot{m}_1<\dot{m}_2<m_d$, $\Omega(\dot{m}_1)\subset\Omega(\dot{m}_2)$;

\smallskip
\item[(ii)] For $m_d<\dot{m}_1<\dot{m}_2<m$, $\Omega(\dot{m}_1)\supset\Omega(\dot{m}_2)$;

\smallskip
\item[(iii)]  $\overline{\Omega}=\overline{\cup_{0<\dot{m}<m_d} \Omega(\dot{m})\cup \cup_{m_d<\dot{m}<m} \Omega(\dot{m}) \cup \Omega(m_d)}$.
\end{enumerate}	
Since the Lipschitz convergence of $\psi^{(\gamma)}$ to $\bar{\psi}$ as $\gamma\to \infty$,
we can obtain that, for each $\dot{m}\neq m_d$, $\Omega^\varepsilon(\dot{m})\to \Omega(\dot{m})$.
	
From the definition of modification, we see that,
for $0<\dot{m}<m_d$, in the interval $[0, \dot{m}]$,
$(\mathbb{B}^{(\gamma)}, \mathbb{S}^{(\gamma)})\to (\mathbb{B}, \mathbb{S})$ in $C^{2, \alpha'}$ with $\alpha'<\alpha$.
On the other hand, for the approximate sequence $\psi^{(\gamma)}$,
$$
\mathbb{B}^{(\gamma)}(\psi^{(\gamma)})=B^{(\gamma)}, \quad
\mathbb{G}^{(\gamma)}(\psi^{(\gamma)})=G^{(\gamma)}, \quad
\partial_{x_1}u_2^{(\gamma)}-\partial_{x_2}u_1^{(\gamma)}=\omega^{(\gamma)}.
$$
Then, by taking a subsequence, we conclude that, in $\Omega(\dot{m})$,
	\begin{eqnarray*}
	\begin{cases}
	\mathbb{B}(\bar{\psi})=\frac{1}{2}|\bar{\bu}|^2+\frac{ \bar{p}}{\bar{\rho}},\\[1mm]
	\mathbb{G}(\bar{\psi})= \bar{\rho},\\
	\partial_{x_1}\bar{u}_2-\partial_{x_2}\bar{u}_1=
	-\bar{\rho} \mathbb{B}'-\frac{\bar{p} }{\bar{\rho}}\mathbb{G}'.
	\end{cases}
	\end{eqnarray*}
Moreover, $\bar{\psi}$ satisfies
\begin{equation*}
\partial_{x_1}(\frac{\partial_{x_1}\bar{\psi}}{\mathbb{G}})+\partial_{x_2}(\frac{\partial_{x_2}\bar{\psi}}{\mathbb{G}})
=-\bar{\rho} \mathbb{B}'-\frac{\bar{p} }{\bar{\rho}}\mathbb{G}'.
\end{equation*}
Employing the same argument as in the previous sections,
we can prove that, for any compact set $K\Subset\cup_{0<\dot{m}<m_d} \Omega(\dot{m})$,
$\bar{\psi}$ is a $C^{2, \alpha}$ smooth function with $\partial_{x_2}\bar{\psi}>0$.
The argument for $\cup_{m_d<\dot{m}<m} \Omega(\dot{m})$ is similar.
	
\medskip
2. We now show the interior of $\Omega(m_d)$ is empty.
The argument here is similar to the one for Lemma \ref{prop:9.1existence} via replacing $\mathbb{S}$ by $\mathbb{G}$.
More precisely, we prove the claim by contradiction.

If the interior $Int(\Omega(m_d))$ of $\Omega(m_d)$ is non-empty, then $\bar{u}_1=\bar{u}_2=0$ in $Int(\Omega(m_d))$, since $\bar{\psi}=m_d$.
Without loss of generality, we assume
\begin{equation}\label{condition:8.2ic}
	(\mathbb{B}\mathbb{G})'(m_d-)>0,\qquad\,\,
	\mathbb{B}'(m_d-)>0.
\end{equation}
	
Let the lower boundary of $\Omega(m_d)$ be $\ell$.
Then, for any test function $\eta$ compactly supported in a connected open set $\Sigma$ with  $\Sigma\cap \ell\neq \emptyset$,
and for the bounded measure $\bar{\omega}$, we have
\begin{eqnarray*}
\langle\eta, \bar{\omega}\rangle
=-\int_{\Sigma}\big(\bar{u}_2\partial_{x_1}\eta-\bar{u}_1\partial_{x_2}\eta\big)\,\dd\bf{x},
\end{eqnarray*}
where $\langle \eta, \bar{\omega}\rangle$
is the limit of $\int \eta\omega^{\varepsilon} \dd\bf{x}$.

Similar to the proof of Lemma \ref{lem:6.2xw},
let $\Sigma_-:= Int(\cup_{0<\dot{m}<m_d} \Omega(\dot{m}))\cap \Sigma $ be the lower part of $\Sigma$,
while the upper part is  $\Sigma_+:= Int ( \Omega(m_d))\cap \Sigma$.
It is obvious that
	$$
	\int_{\Sigma_+}\big(\bar{u}_2\partial_{x_1}\eta-\bar{u}_1\partial_{x_2}\eta\big)\, \dd\bx=0.
	$$
Divide $\Sigma_-$ into two subregions: $\Sigma_-=\Sigma_s\cup \Sigma_b$,
where
$$
\Sigma_s:=\{\bx \,:\,
\bar{u}_1>\sigma>0\}\cap\Sigma_-,
\qquad 	\Sigma_b:=\{\bx\, :\,
	\sigma\geq\bar{u}_1>0\}\cap\Sigma_-,
$$
where $\sigma$ is a given small constant.
Then, following the same argument in the proof of Lemma \ref{lem:6.2xw}, we have
\begin{align*}
\int_{\Sigma_-}\big(\bar{u}_2\partial_{x_1}\eta-\bar{u}_1\partial_{x_2}\eta\big)\, \dd\mathbf{x}
	\leq\frac{C\sigma}{b}|\Sigma_s|
	\end{align*}
for $\eta$ as chosen in \eqref{6.14a} involving $b$ and $w$,
where $C$ is the uniform positive constant independent of $w$.
	
However, by \eqref{condition:8.2ic}, we find  that
vorticity $\omega$ is a negative measure with a uniform lower bound,
which is the contradiction.
In fact, by  \eqref{condition:8.2ic},
\begin{eqnarray*}
\bar{\omega}&=&-\rho\mathbb{B}'-\bar{p}\frac{\mathbb{G}'}{\mathbb{G}}
<-\rho\mathbb{B}'+\bar{p}\frac{\mathbb{B}'}{\mathbb{B}}\nonumber\\
&=&-\bar{\rho}\mathbb{B}'(1-\frac{\bar{p}}{\bar{\rho}\mathbb{B}})
=-\frac{1}{2}\bar{\rho}|\bar{\bu}|^2\mathbb{B}'\leq0.
\end{eqnarray*}
Then $\bar{\omega}$ is strictly negative.
This implies that $\bar{\omega}\leq-C^{-1}<0$ in $\Sigma_-$.
Therefore, we have
$$
-\frac{C\sigma}{b}|\Sigma_-|\leq -\frac{C\sigma}{b}|\Sigma_s|
\leq-\int_{\Sigma_-}\big(\bar{u}_2\partial_{x_1}\eta-\bar{u}_1\partial_{x_2}\eta\big)\, \dd\mathbf{x}
= \langle \bar{\omega},\eta\rangle
\leq -C^{-1}|\Sigma_-|<0,
$$
which is impossible if $\sigma$ is chosen small enough.
Thus, we conclude
$$
\overline{\big(\cup_{0<\dot{m}<m_d} \Omega(\dot{m})\big)\cup\big(\cup_{m_d<\dot{m}<m} \Omega(\dot{m})\big)}\equiv\overline{\Omega}.
$$
Therefore, we prove that $\bar{\psi}$ is a solution of the inhomogeneous incompressible equation \eqref{ICEuler}
 with $\partial_{x_2}\bar{\psi}\geq0$, $|\bar{\theta}|\leq\theta_B$, and $\min_{\partial \Omega}\le p \le\max_{\partial\Omega} p\,$ \textit{a.e.}.
 Furthermore, $\Omega(m_d)$ is a Lipschitz curve.
 In particular, in any compact set $K\subset \Omega-\Omega(m_d)$,
 $\bar{\psi}\in C^{2, \alpha}(K)$ and $\partial_{x_2}\bar{\psi}>0$.

\smallskip
3. It is clear that the solutions satisfy properties
\eqref{ieq:nondegeneracy0}--\eqref{ieq:streamline} and \eqref{equ:2.59a}--\eqref{equ:2.61a}.
We now consider  the uniqueness of weak solutions with discontinuities.
Suppose that $\bar{\psi}^{(i)}$, $i=1, 2$, are any two solutions.
For each $\bar{\psi}^{(i)}$, we can obtain the corresponding potential function $\bar{\varphi}^{(i)}$ in the Lagrangian coordinates.
Then \eqref{weakLequation} yields
\begin{eqnarray*}
0
=\int_{\mathbb{R}\times[0, m]} \Big(\frac{\partial_{z_1}\bar{\varphi}^{(i)}}{\rho(z_2)\partial_{z_2}\bar{\varphi}^{(i)}}\partial_{z_1} \eta
 + p(\nabla_z\bar{\varphi}^{(i)}, z_2)\partial_{z_2}\eta\Big)\, \dd \mathbf{z}\qquad\mbox{for $i=1, 2$},  \label{weakLequationi}
\end{eqnarray*}
where density $\rho$ is a function of only one variable $z_2$,
and pressure $p=\rho(z_2)\mathbb{B}(z_2)-\frac{1+(\partial_{z_1}\bar{\varphi}^{(i)})^2}{(\partial_{z_2}\bar{\varphi}^{(i)})^2}$.
	
By a direct calculation, $\hat{\varphi}=\bar{\varphi}^{(1)}-\bar{\varphi}^{(2)}$ satisfies
	\begin{equation*}
	\int_{\mathbb{R}\times[0, m]} \big(
	(\hat{a}_{11}\partial_{z_1}\hat{\varphi}+\hat{a}_{12}\partial_{z_2}\hat{\varphi})
	\partial_{z_1} \eta+ (\hat{a}_{21}\partial_{z_1}\hat{\varphi}+\hat{a}_{22}\partial_{z_2}\hat{\varphi})\partial_{z_2}\eta\big)\, \dd \mathbf{z}=0
	\end{equation*}
with
\begin{align*}
&\hat{a}_{11}=\int_1^2\frac{\dd\tau}{\rho( z_2)\partial_{z_2}\varphi^{(\tau)}},
\qquad\qquad
	\hat{a}_{12}=-\int_1^2\frac{\partial_{z_1}\varphi^{(\tau)}}{\rho(z_2)(\partial_{z_2}\varphi^{(\tau)})^2}\,\dd \tau,\\[1.5mm]
&	\hat{a}_{21}=-\int_1^2\frac{\partial_{z_1}\varphi^{(\tau)}}{(\partial_{z_2}\varphi^{(\tau)})^2}\,\dd \tau,
\qquad\quad	
	\hat{a}_{22}=2\int_1^2\frac{1+(\partial_{z_1}\varphi^{(\tau)})^2}{(\partial_{z_2}\varphi^{(\tau)})^3}\,\dd \tau,
\end{align*}
where
$\varphi^{(\tau)}=(2-\tau)\bar{\varphi}^{(1)}+(\tau-1)\bar{\varphi}^{(2)}$ 	for some $\tau\in(1,2)$,
and $\tilde{\eta}_{L}(z_1)$ is the domain cut-off function supported on $[-L, L]$.
We pick $\eta=\hat{\varphi}\tilde{\eta}_L$ and let $L\rightarrow \infty$. Then
\begin{equation*}
\int_{\mathbb{R}\times[0, m]}
\Big(	\hat{a}_{11}\partial_{z_1}\hat{\varphi}\partial_{z_1}\hat{\varphi}+\hat{a}_{12}\partial_{z_1}\hat{\varphi}\partial_{z_2}\hat{\varphi}
	+ \hat{a}_{21}\partial_{z_1}\hat{\varphi}\partial_{z_2}\hat{\varphi}+\hat{a}_{22}\partial_{z_2}\hat{\varphi}\partial_{z_2}\hat{\varphi}\Big)\, \dd \mathbf{z}=0.
\end{equation*}
Note that, by the ellipticity, for any ${\boldsymbol \xi}\in\mathbb{R}^2\backslash\{\mathbf{0}\}$,
$$
0 \leq C\int_1^2\frac{\dd\tau}{\rho(z_2)(\partial_{z_2}\varphi^{(\tau)})^3}\, |{\boldsymbol \xi}|^2
\leq\sum_{j,k=1}^2 \hat{a}_{jk}\xi_j\xi_k.
$$
Moreover, note that
$\int_1^2\frac{\dd\tau}{\rho(z_2)(\partial_{z_2}\varphi^{(\tau)})^3}=0$ holds
	if and only if $\partial_{z_2}\hat{\varphi}=0$.
Therefore, we conclude that  $\bar{\varphi}^{(1)}=\bar{\varphi}^{(2)}$.
\end{proof}

\section{Remarks on steady full Euler flows with conservative exterior forces}

The steady full Euler flows with an exterior force are governed by
\begin{eqnarray}\label{EulerCF}
\begin{cases}
\mbox{div}_{\bx}(\rho \bu)=0,\\
\mbox{div}_{\bx}(\rho \bu\otimes \bu)+ \nabla p=\rho \bf{G},\\
\mbox{div}_{\bx} (\rho \bu E+ p \bu)=\rho \bf{u}\cdot \bf{G},
\end{cases}
\end{eqnarray}
where $\rho \bf{G}$ is the exterior force.
When $\bf{G}$ is in the conservative form, the situation can be reduced to the case without the exterior force.
As a conservative force, we can introduce a potential function $\phi$ such that $\bf{G}=\nabla\phi$.
Then we obtain that, from $(\ref{EulerCF})_3$,
\begin{equation}
\mbox{div}_{\bx} (\rho \bu E+ p \bu)=\rho \bf{u}\cdot \bf{G}=\rho \bf{u}\cdot \nabla\phi=\mbox{div}_{\bx}(\rho \bu \phi)-\phi\mbox{div}_{\bx}(\rho \bu )=\mbox{div}_{\bx}(\rho \bu \phi).
\end{equation}
Then the new transport equation of the Bernoulli function is
\begin{equation}
\mbox{div}_{\bx}(\rho \bu (B-\phi))=0.
\end{equation}
The transport equation of the entropy function $S$ is the same as the case without the exterior force
from $\bu\cdot((\ref{EulerCF})_2-\bu(\ref{EulerCF})_1)+\frac{|\bu|^2}{2}(\ref{EulerCF})_1-(\ref{EulerCF})_3$:
$$
\mbox{div}_{\bx}(\rho \bu S)=0.
$$
Furthermore, the Bernoulli-vortex relation stays as before.

By replacing $B$ by $B-\phi$, all the corresponding results in Theorems 2.1--2.2, Theorems 7.1--7.2, and Theorem 8.1
are still valid for this case,
since the same arguments in all the proofs apply to this case as well.
Formally, the conservative force just influences the density-speed relations,
but does not change the vorticity of fluid field, since the exterior force is an irrotational
vector field.
Moreover, choosing ${\bf G}=(1, 0)^\top$ as the gravity is the motivation of jump condition
at Lemma \ref{lem:3.1xw}.  See also Gu-Wang \cite{GW} for a similar argument.

\medskip
\bigskip
\textbf{Acknowledgments}: The research of Gui-Qiang Chen was supported in part by
the UK
Engineering and Physical Sciences Research Council Award
EP/E035027/1 and
EP/L015811/1, and the Royal Society--Wolfson Research Merit Award (UK).
The research of Fei-Min Huang was supported in part by NSFC Grant
No. 11688101,
and Key Research Program of Frontier Sciences, CAS.
The research of Tian-Yi Wang was supported in part by NSFC Grant
No. 11601401 and the Fundamental Research Funds for the Central
Universities (WUT: 2017 IVA 072 and 2017 IVB 066).
The research of Wei Xiang was supported in part by the CityU
Start-Up Grant for New Faculty 7200429(MA), and the Research Grants Council of the
HKSAR,
China (Project No. CityU 21305215, Project No. CityU 11332916,
Project No. CityU 11304817, and Project No. CityU 11303518).

\end{document}